\documentclass[12pt,a4paper]{article}
\usepackage[T1]{fontenc}
\usepackage[utf8]{inputenc}
\DeclareUnicodeCharacter{2192}{\ensuremath{\to}}
\usepackage{natbib}
\setcitestyle{authoryear,open={(},close={)},citesep={,}}
\usepackage{amsmath,amssymb,amsthm,mathtools, bm,mathrsfs,dsfont}
\mathtoolsset{showonlyrefs=false,showmanualtags=true}
\usepackage{thmtools}
\usepackage[dvipsnames]{xcolor}
\usepackage{multirow}

\usepackage{graphicx, caption,subcaption, array,tabularx, booktabs, multirow}
\usepackage[shortlabels]{enumitem}
\setlist{topsep=0.35em,itemsep=0.15em,parsep=0pt,partopsep=0pt}
\setlist[enumerate,1]{leftmargin=*,label=\textup{(\roman*)},ref=\textup{(\roman*)}}
\allowdisplaybreaks
\usepackage{microtype}
\usepackage{titling,abstract}
\usepackage[a4paper,left=2.5cm,right=2.5cm,top=2.5cm,bottom=2.5cm]{geometry}
\usepackage[colorlinks=true,linkcolor=blue,citecolor=blue,urlcolor=blue]{hyperref}
\usepackage[capitalize,noabbrev]{cleveref}

\pretitle{\begin{center}\begin{minipage}{0.94\textwidth}\centering\LARGE\bfseries}
\posttitle{\par\end{minipage}\end{center}\vspace{0.2em}}
\preauthor{\begin{center}\large\scshape}
\postauthor{\par\end{center}\vspace{-0.6em}}
\predate{\begin{center}\normalsize\itshape}
\postdate{\par\end{center}\vspace{-0.45em}}

\newcommand{\paperkeywords}[1]{\begin{center}\begin{minipage}{0.86\textwidth}\small\noindent\textbf{Keywords:} #1\end{minipage}\end{center}}

\declaretheoremstyle[spaceabove=6pt,spacebelow=6pt,headfont=\normalfont\bfseries,notefont=\normalfont,bodyfont=\itshape,headpunct={.},postheadspace=0.5em,notebraces={(}{)}]{plainthm}
\declaretheoremstyle[spaceabove=6pt,spacebelow=6pt,headfont=\normalfont\bfseries,notefont=\normalfont,bodyfont=\normalfont,headpunct={.},postheadspace=0.5em,notebraces={(}{)}]{definitionthm}
\declaretheoremstyle[spaceabove=6pt,spacebelow=6pt,headfont=\normalfont\bfseries,notefont=\normalfont,bodyfont=\normalfont,headpunct={.},postheadspace=0.5em,notebraces={(}{)}]{remarkthm}
\declaretheorem[style=plainthm,name=Theorem,numberwithin=section]{theorem}
\declaretheorem[style=plainthm,name=Lemma,sibling=theorem]{lemma}
\declaretheorem[style=plainthm,name=Corollary,sibling=theorem]{corollary}
\declaretheorem[style=plainthm,name=Proposition,sibling=theorem]{proposition}

\newcounter{rlfrcondition}
\makeatletter
\newenvironment{rlfrcondition}[1]{%
  \par\smallskip
  \refstepcounter{rlfrcondition}%
  \def\@currentlabel{#1}%
  \noindent\textbf{Condition \textup{#1}.}\ignorespaces
}{\par\smallskip}
\makeatother
\declaretheorem[style=definitionthm,name=Example,numberwithin=section]{example}
\declaretheorem[style=remarkthm,name=Remark,numberwithin=section]{remark}

\crefname{theorem}{Theorem}{Theorems}
\Crefname{theorem}{Theorem}{Theorems}
\crefname{lemma}{Lemma}{Lemmas}
\Crefname{lemma}{Lemma}{Lemmas}
\crefname{corollary}{Corollary}{Corollaries}
\Crefname{corollary}{Corollary}{Corollaries}
\crefname{proposition}{Proposition}{Propositions}
\Crefname{proposition}{Proposition}{Propositions}
\crefname{rlfrcondition}{Assumption}{Assumptions}
\Crefname{rlfrcondition}{Assumption}{Assumptions}
\crefname{remark}{Remark}{Remarks}
\Crefname{remark}{Remark}{Remarks}
\crefname{example}{Example}{Examples}
\Crefname{example}{Example}{Examples}

\newcommand{\R}{\mathbb R}
\newcommand{\E}{\mathbb E}
\newcommand{\Pbb}{\mathbb P}
\newcommand{\argmin}{\operatorname*{arg\,min}}
\newcommand{\PaperTitle}{Robust Local Fr\'echet Regression: Nonasymptotic Theory and Minimax Optimality}%Robust Local Fr\'echet Regression with High-Confidence Guarantees}
\newcommand{\PaperAuthors}{Chang Jun Im\textsuperscript{1}\quad Daisuke Kurisu\textsuperscript{2}}
\newcommand{\PaperAffiliations}{%
\textsuperscript{1}The Institute for Data Innovation in Science, Seoul National University, South Korea\\
\textsuperscript{2}Faculty of Economics, The University of Tokyo, Japan}

\newif\ifcompilewithsupplement
\compilewithsupplementtrue

\begin{document}
\title{\PaperTitle}
\author{\PaperAuthors}
\date{\PaperAffiliations}
\maketitle

\begin{abstract}
We develop robust local Fr\'echet regression for heavy-tailed responses in a general metric space and propose a distance-based median-of-means estimator. We derive nonasymptotic pointwise high-probability bounds separately under local second-order smoothness and local $1$-Wasserstein H\"older continuity, and show that the corresponding optimized rates match the classical pointwise rates for nonparametric regression. We also establish matching pointwise minimax lower bounds over embedded Gaussian regression subclasses. Remarkably, our results do not require an entropy condition on the response space. Simulation studies with positive-definite matrix-valued responses, together with a real-data application involving Wasserstein-valued responses, illustrate the robustness--efficiency tradeoff and practical performance of the proposed estimator.
\end{abstract}

\paperkeywords{Fr\'echet regression, heavy-tailed data, median-of-means, weighted empirical Fr\'echet mean, high-probability bounds, minimax lower bounds.}

\section{Introduction}\label{sec:introduction}                  

Responses in many applications are not naturally represented as vectors. Diffusion tensors, covariance matrices, shapes, rotations, and probability distributions are typical examples. In many such applications, the predictors are Euclidean variables such as age, time, dose, or clinical measurements. \citet{Frechet1948} introduced the Fr\'echet mean as a generalization of the usual mean to metric spaces, providing a general framework for defining averages of non-Euclidean data. Building on this idea, \citet{PetersenMuller2019} proposed Fr\'echet regression for metric-space-valued responses with Euclidean predictors, with the conditional Fr\'echet mean as the regression target. Their local estimator is obtained by minimizing a kernel-weighted sum of squared distances. This localization, while essential for nonparametric regression, creates a particular vulnerability to heavy-tailed responses: a small number of observations with unusually large distances can have a disproportionate effect on the local criterion, even when the response has a finite second moment.

Regression with non-Euclidean responses has been studied in a variety of geometric settings. \citet{Fletcher2013} and \citet{HinkleEtAl2012} considered geodesic and polynomial regression for Riemannian-manifold-valued responses. \citet{YuanZhuLinMarron2012} developed intrinsic local polynomial regression for positive-definite matrix responses, while \citet{LinEtAl2015} studied extrinsic local regression for manifold-valued responses. For general metric-space-valued responses, \citet{PetersenMuller2019} developed local Fr\'echet regression and established its asymptotic theory. \citet{Schoetz2022} subsequently obtained nonasymptotic expectation bounds under fixed design, with conditions involving the entropy of the response space. More recently, \citet{Schoetz2026WeightedPower} established nonasymptotic expectation bounds for weighted power Fr\'echet means with applications to local constant Fr\'echet regression under fixed design without imposing response-space entropy conditions. In the Euclidean setting, classical pointwise minimax theory for nonparametric regression was established by \citet{Stone1980}.

Robust estimation under weak moment assumptions has a long history. The median-of-means principle dates back to \citet{NemirovskiYudin1983} and has since been developed in a variety of settings. \citet{Minsker2015} used geometric medians to aggregate independent estimators and obtain exponential deviation bounds in Banach spaces. \citet{DevroyeEtAl2016} and \citet{LugosiMendelson2019} developed sub-Gaussian mean estimation for real- and multivariate-valued observations under only finite second moments. \citet{HsuSabato2016} extended the median-of-means principle to general metric spaces and proposed a distance-based aggregation rule for estimation and loss minimization. For Fr\'echet means, \citet{BrunelSerres2024,BrunelSerres2025}  derived nonasymptotic bounds for empirical barycenters, while \citet{BartlEtAl2025} developed robust procedures under heavy tails and contamination. For Fr\'echet regression, \citet{HongXuWood2026} established uniform convergence results for generalized conditional Fr\'echet means and proposed weighted aggregation procedures, including a median-of-means version of Fr\'echet regression that aggregates blockwise estimators through a Fr\'echet median.

However, the results of \citet{HongXuWood2026} do not directly yield an explicit pointwise nonasymptotic theory for the random-design local kernel setting considered here. In this setting, the kernel weights are random and self-normalized, and the conditional response distribution varies with the predictor. Thus, the stochastic error must be controlled under unequal random weights and heterogeneous conditional response distributions, with explicit dependence on the bandwidth, local predictor density, and confidence level. Moreover, \citet{Schoetz2026WeightedPower} did not develop a robust estimator for classical Fr\'echet regression.

In this paper, we develop robust local Fr\'echet regression (RLFR) under a local second-moment condition. We divide the observations into blocks, compute a kernel-weighted empirical Fr\'echet mean within each block, and aggregate the resulting block estimators using the distance-based rule of \citet{HsuSabato2016}. The procedure therefore inherits the robustness of median-of-means methods while avoiding any additional optimization over the response space. A key technical ingredient is a nonasymptotic bound for weighted empirical Fr\'echet means based on independent observations with unequal weights and potentially different distributions. Conditional on the predictors, this result can be applied directly to the blockwise local estimators. 

Under suitable geometric conditions and a local second-order smoothness condition, we show that RLFR attains the classical pointwise convergence rate for local regression at any fixed confidence level. We further establish a pointwise minimax lower bound on an embedded Gaussian regression subclass, which matches the dependence of our upper bound on the sample size and confidence level over their common range. Such a minimax lower bound for local Fr\'echet regression under the setting considered here seems not available in this literature. Importantly, our upper bound requires only a local second-moment condition on the response and does not impose response-space entropy conditions that are commonly used in theoretical analyses of local Fr\'echet regression, such as \citet{PetersenMuller2019} and \citet{ChenMuller2022}. Thus, the robustness of RLFR does not come at the cost of stronger complexity assumptions on the response space. We also consider a complementary setting in which the conditional response distributions satisfy local H\"older continuity with respect to the $1$-Wasserstein distance. In this setting, we derive analogous nonasymptotic upper bounds and matching pointwise minimax lower bounds on an embedded Gaussian regression subclass. These results provide a unified characterization of the statistical difficulty of robust local Fr\'echet regression under these two smoothness conditions.

We assess the finite-sample performance of RLFR through simulation studies with positive-definite matrix-valued responses. We further illustrate its practical application through an analysis of wrist-actigraphy data from the Multi-Ethnic Study of Atherosclerosis (MESA), using body mass index as the scalar predictor and distributions of waking activity counts as Wasserstein-valued responses.

The rest of the paper is organized as follows. Section~\ref{sec:setting} introduces the problem and the RLFR estimator. Section~\ref{sec:main-results} presents the nonasymptotic bounds and the pointwise lower bounds. Section~\ref{sec:simulation} reports the simulation study, Section~\ref{subsec:realdata_mesa} presents a real-data application, and Section~\ref{sec:discussion} concludes the paper. The Supplementary Material contains all technical proofs and additional details.

\section{Problem setting and the estimator}\label{sec:setting}

\subsection{Problem setting}

Let $(\Omega,\mathcal F,\mathbb P)$ be a probability space. The predictor $X$ takes values in an open set $\mathcal X\subset\R^d$, and the response $Y$ takes values in a metric space $(\mathcal Y,d_{\mathcal Y})$. Both spaces are equipped with their Borel $\sigma$-fields. Let $P_X$ and $P_Y$ denote the marginal distributions of $X$ and $Y$, respectively, and fix a regular conditional distribution $P_{Y\mid X=x}$. We assume that the random minimizers used below admit measurable versions. Then, for $\omega\in\mathcal Y$, define
\begin{align*}
    M_\oplus(x, \omega):=\E\left[d_{\mathcal Y}^2(Y,\omega)\mid X=x\right], \quad x\in\mathcal{X}.
\end{align*}
Following \citet{PetersenMuller2019}, whenever the minimizer of the above exists and is unique, define the conditional Fr\'echet mean by
\begin{align*}
    m_{\oplus}(x):=\argmin_{\omega\in\mathcal Y}M_\oplus(x, \omega),\quad x\in\mathcal X.
\end{align*}
For $\mathcal Y=\R$ with the Euclidean metric, this reduces to the usual conditional mean whenever the conditional second moment is finite.

In this paper, we estimate $m_\oplus(x_0)$ at a fixed interior point $x_0\in\mathcal X$ from the i.i.d. sample $\mathcal D_n:=\left((X_1,Y_1),\ldots,(X_n,Y_n)\right)$. For the notations throughout the paper, $\|\cdot\|$ denotes the Euclidean norm, and $B(x,r)$ and $\overline B(x,r)$ denote the open and closed Euclidean balls centered at $x\in\R^d$ with radius $r\in(0,\infty)$. Also, $B_{\mathcal Y}(y,r)$ and $\overline B_{\mathcal Y}(y,r)$ denote the open and closed balls in $\mathcal Y$ centered at $y \in\mathcal{Y}$ with radius $r\in(0,\infty)$.

\subsection{Robust local Fr\'echet regression (RLFR) estimator}\label{subsec:smooth-estimator}

Let $K:\R^d\to[0,\infty)$ be a bounded measurable kernel. For a bandwidth $h\in(0,\infty)$, define
\begin{align*}
    K_h(z,x) := \frac{1}{h^d}K\left(\frac{z-x}{h}\right), \quad z,x\in\mathcal X.
\end{align*}
Fix $B\in\{1,\ldots,\lfloor n/2\rfloor\}$ and partition $\{1,\ldots,n\}$ deterministically into disjoint blocks $I_1,\ldots,I_B$ whose sizes differ by at most one. For $b=1, \ldots, B$, define
\begin{align} \label{eq:means-selector}
    \widehat m_{h, b}(x_0) := \begin{cases} \argmin_{\omega\in\mathcal Y}\sum_{i\in I_b}K_h(X_i,x_0)d_{\mathcal Y}^2(Y_i,\omega), & \text{if} ~ \sum_{i\in I_b}K_h(X_i,x_0)>0, \\
    \argmin_{\omega\in\mathcal Y}\sum_{i\in I_b}d_{\mathcal Y}^2(Y_i,\omega), & \text{if} ~ \sum_{i\in I_b}K_h(X_i,x_0)=0. \end{cases}
\end{align}
We then aggregate the block estimates by the distance-based median-of-means rule of \citet{HsuSabato2016}. For $b=1,\ldots,B$, let
\begin{align} \label{eq:median-of-means-selector}
    r_{h, b}(x_0) := \inf\left\{ r\ge0: \sum_{l=1}^B \mathds{1}_{\left\{d_{\mathcal Y}\left(\widehat m_{h,b}(x_0),\widehat m_{h,l}(x_0)\right)\le r\right\}}>\frac B2 \right\},
\end{align}
where $\mathds{1}$ denotes the indicator function. Note that $r_{h,b}(x_0)$ is the smallest radius such that the ball centered at $\widehat m_{h,b}(x_0)$ contains a strict majority of the block estimates. Choose $\widehat b_h(x_0)\in\argmin_{1\le b\le B}r_{h,b}(x_0)$, where ties are resolved according to a fixed deterministic rule, such as choosing the smallest minimizing index. Then, define
\begin{align*}
    \widehat m^{\mathrm{RLFR}}_{h,B}(x_0)
    :=
    \widehat m_{h,\widehat b_h(x_0)}(x_0).
\end{align*}
We call $\widehat m^{\mathrm{RLFR}}_{h,B}(x_0)$ the robust local Fr\'echet regression (RLFR) estimator. 

\section{Theoretical results}\label{sec:main-results}

\subsection{Weighted empirical Fr\'echet means}\label{subsec:general-framework}

In this subsection, we establish a nonasymptotic bound for weighted empirical Fr\'echet means. We first introduce a condition on the kernel.

\begin{rlfrcondition}{K1}\label{ass:K1}
The kernel $K:\R^d\to[0,\infty)$ is a bounded probability density such that $\operatorname{supp}(K)\subset\overline B(0,1)$.
\end{rlfrcondition}

Conditions of this type are standard in kernel density estimation and local regression; see, for example, \citet{WandJones1995} and \citet{FanGijbels1996}. They are satisfied by many commonly used compactly supported kernels, including uniform, Epanechnikov, biweight, and triweight kernels.

We also impose a condition on $f_X:\R^d\to[0,\infty)$, which denotes the density of $X$.

\begin{rlfrcondition}{D1}\label{ass:D1}
The density $f_X$ is continuous at $x_0$ and satisfies $f_X(x_0)>0$.
\end{rlfrcondition}

Condition~\ref{ass:D1} is standard in kernel smoothing and nonparametric regression. By Condition~\ref{ass:D1} and the openness of $\mathcal{X}$, we can fix $h_0\in(0,\infty)$ sufficiently small such that
\begin{align} 
\begin{split}
    \overline B(x_0,h_0)&\subset\mathcal X,\\
    \frac{f_X(x_0)}{2}\le f_X(x)&\le2f_X(x_0), \quad
    x\in\overline B(x_0,h_0).
\end{split} \label{eq:local-density-bound}
\end{align}
By Lemma~\ref{lem:kernel-target-positive} (in the Supplementary Material), Condition~\ref{ass:K1} and Condition~\ref{ass:D1} ensure $\E[K_h(X,x_0)]>0$ for $h\in(0,h_0]$. Therefore, for $\omega\in\mathcal Y$ and $h\in(0,h_0]$, we can define the smoothed version of $M_\oplus(x_0, \omega)$ by
\begin{align} \label{def:bigmhtilde}
    \widetilde M_h(x_0,\omega):=\frac{\E\left[K_h(X,x_0)d_{\mathcal Y}^2(Y,\omega)\right]}{\E[K_h(X,x_0)]}.
\end{align}
Whenever the minimizer exists and is unique, define
\begin{align*}
    \widetilde m_h(x_0):=\argmin_{\omega\in\mathcal Y}\widetilde M_h(x_0,\omega).
\end{align*}
Next, we define three weighted Fr\'echet means that will be used in the theoretical analysis. Let $\mathcal L_2(\mathcal Y)$ denote the class of $\mathcal Y$-valued random variables with finite second moments, that is, 
\begin{align*}
    \mathcal{L}_2(\mathcal{Y}):=\left\{Y:\mathbb{P}(Y\in\mathcal{Y})=1, \quad \mathbb{E}\bigl[d_{\mathcal Y}^2(Y,y_0)\bigr]<\infty\quad\text{for some}\quad y_0\in\mathcal{Y}\right\}.
\end{align*}
The definition is independent of the choice of $y_0\in\mathcal{Y}$ by the triangle inequality. For $N\in\mathbb N$, define
\begin{align*}
    \mathcal S_N:=\left\{(\lambda_1,\ldots,\lambda_N)\in[0,1]^N:\sum_{i=1}^N\lambda_i=1\right\}.
\end{align*}
For $\bm\lambda\in\mathcal S_N$ and $\omega\in\mathcal Y$, define
\begin{align*}
    \overline M_{\oplus,\bm\lambda}(\bm x,\omega)&:=\sum_{i=1}^N\lambda_iM_\oplus(x_i,\omega), \quad \bm x=(x_1,\ldots,x_N)\in\mathcal X^N, \\
    \overline B_{\oplus,\bm\lambda}(\bm y,\omega)&:=\sum_{i=1}^N\lambda_i d_{\mathcal Y}^2(y_i,\omega), \quad \bm y=(y_1,\ldots,y_N)\in\mathcal Y^N, \\
    \overline Q_{\oplus,\bm\lambda}(\widetilde{\bm Y},\omega)&:=\sum_{i=1}^N\lambda_i\E\left[d_{\mathcal Y}^2(\widetilde Y_i,\omega)\right], \quad \widetilde{\bm Y}=(\widetilde Y_1,\ldots,\widetilde Y_N)\in\mathcal L_2(\mathcal Y)^N.
\end{align*}
Whenever each minimizer exists and is unique, write
\begin{align*}
    \overline m_{\oplus,\bm\lambda}(\bm x)&:=\argmin_{\omega\in\mathcal Y}\overline M_{\oplus,\bm\lambda}(\bm x,\omega), \quad \bm x=(x_1,\ldots,x_N)\in\mathcal X^N, \\
    \overline b_{\oplus,\bm\lambda}(\bm y)&:=\argmin_{\omega\in\mathcal Y}\overline B_{\oplus,\bm\lambda}(\bm y,\omega), \quad \bm y=(y_1,\ldots,y_N)\in\mathcal Y^N, \\
    \overline q_{\oplus,\bm\lambda}(\widetilde{\bm Y})&:=\argmin_{\omega\in\mathcal Y}\overline Q_{\oplus,\bm\lambda}(\widetilde{\bm Y},\omega), \quad \widetilde{\bm Y}=(\widetilde Y_1,\ldots,\widetilde Y_N)\in\mathcal L_2(\mathcal Y)^N.
\end{align*}
Then, we impose conditions on $\mathcal Y$ using these minimizers.

\begin{rlfrcondition}{M1}\label{ass:M1}
The following minimizers exist and are unique.
\begin{enumerate}[(i)]
\item $\widetilde m_h(x_0)$ for every $h\in(0,h_0]$.

\item $\overline m_{\oplus,\bm\lambda}(\bm x)$ for every
$N\in\mathbb N$, $\bm\lambda\in\mathcal S_N$, and $\bm x\in\overline B(x_0,h_0)^N$.

\item $\overline b_{\oplus,\bm\lambda}(\bm y)$ for every
$N\in\mathbb N$, $\bm\lambda\in\mathcal S_N$, and $\bm y\in\mathcal Y^N$.
\item $\overline q_{\oplus,\bm\lambda}(\widetilde{\bm Y})$ for every $N\in\mathbb N$, $\bm\lambda\in\mathcal S_N$, and $\widetilde{\bm Y}\in\mathcal L_2(\mathcal Y)^N$.
\end{enumerate}
\end{rlfrcondition}

\begin{rlfrcondition}{M2}\label{ass:M2}
There exists a constant $\kappa_{\mathcal Y}\in(0,\infty)$ such that the following inequalities hold.
\begin{enumerate}[(i)]
\item For every $h\in(0,h_0]$ and $\omega\in\mathcal Y$,
\begin{align*}
    \widetilde M_h(x_0,\omega) -\widetilde M_h\left(x_0,\widetilde m_h(x_0)\right) \ge \kappa_{\mathcal Y} d_{\mathcal Y}^2\left(\omega,\widetilde m_h(x_0)\right).
\end{align*}

\item For every $N\in\mathbb N$, $\bm\lambda\in\mathcal S_N$,
$\bm x\in\overline B(x_0,h_0)^N$, and $\omega\in\mathcal Y$,
\begin{align*}
    \overline M_{\oplus,\bm\lambda}(\bm x,\omega) -\overline M_{\oplus,\bm\lambda} \left(\bm x,\overline m_{\oplus,\bm\lambda}(\bm x)\right) \ge \kappa_{\mathcal Y} d_{\mathcal Y}^2 \left(\omega,\overline m_{\oplus,\bm\lambda}(\bm x)\right).
\end{align*}

\item For every $N\in\mathbb N$, $\bm\lambda\in\mathcal S_N$,
$\bm y\in\mathcal Y^N$, and $\omega\in\mathcal Y$,
\begin{align*}
    \overline B_{\oplus,\bm\lambda}(\bm y,\omega) -\overline B_{\oplus,\bm\lambda} \left(\bm y,\overline b_{\oplus,\bm\lambda}(\bm y)\right) \ge \kappa_{\mathcal Y} d_{\mathcal Y}^2 \left(\omega,\overline b_{\oplus,\bm\lambda}(\bm y)\right).
\end{align*}
\item For every $N\in\mathbb N$, $\bm\lambda\in\mathcal S_N$, $\widetilde{\bm Y}\in\mathcal L_2(\mathcal Y)^N$, and $\omega\in\mathcal Y$,
\begin{align*}
    \overline Q_{\oplus,\bm\lambda}(\widetilde{\bm Y},\omega)-\overline Q_{\oplus,\bm\lambda}\left(\widetilde{\bm Y},\overline q_{\oplus,\bm\lambda}(\widetilde{\bm Y})\right)\ge\kappa_{\mathcal Y}d_{\mathcal Y}^2\left(\omega,\overline q_{\oplus,\bm\lambda}(\widetilde{\bm Y})\right).
\end{align*}
\end{enumerate}
\end{rlfrcondition}

Condition~\ref{ass:M1} guarantees existence and uniqueness of the Fr\'echet means, while Condition~\ref{ass:M2} assumes a quadratic curvature condition. Conditions of this type have been used in previous work on Fr\'echet regression; see, for example, \citet{PetersenMuller2019}. 

\begin{rlfrcondition}{M3}\label{ass:M3}
There exists a constant $\sigma\in(0,\infty)$ such that
\begin{align*}
    \sup_{x\in\overline B(x_0,h_0)} M_\oplus\left(x,m_\oplus(x_0)\right) \le \sigma^2.
\end{align*}
\end{rlfrcondition}

Since $h_0$ can be chosen sufficiently small, we may assume without loss of generality that $h_0\leq\sigma$ throughout the paper. Condition~\ref{ass:M3} is a local second-moment condition around the target $m_\oplus(x_0)$. Unlike the response-space entropy conditions used in previous Fr\'echet regression theory, Condition~\ref{ass:M3} imposes no covering-number condition on $\mathcal Y$; see, for example, \citet{PetersenMuller2019}. In particular, Condition~\ref{ass:M3} can hold for responses in the 2-Wasserstein space $\mathcal W_2(\R)$, for which the corresponding entropy condition does not hold; see Section~\ref{subsec:hadamard-examples} of the Supplementary Material for details. When $\mathcal Y = \mathbb{R}$, Condition~\ref{ass:M3} is equivalent to
\begin{align*}
    \sup_{x\in\overline B(x_0,h_0)} \left( \operatorname{Var}\left[Y\mid X=x\right]+\left(m_\oplus(x)-m_\oplus(x_0)\right)^2 \right) \le \sigma^2.
\end{align*}
If Conditions~\ref{ass:M1}--\ref{ass:M2} hold, $m_\oplus(x_0)$ in Condition~\ref{ass:M3} may be replaced by any fixed point in $\mathcal Y$, with only a change in the constant; see Remark~\ref{rem:moment-center} in the Supplementary Material for a discussion on this condition.

\begin{rlfrcondition}{M4}\label{ass:M4}
There exists a constant $C_{\mathcal{Y}}\in(0,\infty)$ such that, for all $y,y',\omega,\omega'\in\mathcal Y$,
\begin{align*}
    \left| d_{\mathcal Y}^2(y,\omega)-d_{\mathcal Y}^2(y,\omega')-d_{\mathcal Y}^2(y',\omega)+d_{\mathcal Y}^2(y',\omega') \right| \le C_{\mathcal{Y}} d_{\mathcal Y}(y,y')d_{\mathcal Y}(\omega,\omega').
\end{align*}
\end{rlfrcondition}

Condition~\ref{ass:M4} is the quadruple inequality for the squared distance. This condition has been used in previous Fr\'echet mean and regression theory; see, for example, \citet{Schoetz2022}.

Now, we introduce several metric spaces satisfying the above conditions. Note that Condition~\ref{ass:M3} is automatically satisfied for bounded response spaces, while it still requires a distributional condition for unbounded response spaces; see Section~\ref{sec:examples-scope} of the Supplementary Material.

\begin{example}\leavevmode\label{ex:hadamard-main}
    Separable Hadamard spaces are complete separable CAT(0) spaces, that is, geodesic metric spaces of nonpositive curvature. They satisfy Conditions~\ref{ass:M1}, \ref{ass:M2}, and \ref{ass:M4} with $\kappa_{\mathcal Y}=1$ and $C_{\mathcal Y}=2$. Examples include separable Hilbert spaces, complete separable metric trees, BHV spaces of phylogenetic trees, symmetric positive-definite matrices under the affine-invariant metric, the $2$-Wasserstein space $\mathcal W_2(\R)$, the space of compact convex sets, and real hyperbolic spaces. For more details, see Section~\ref{subsec:hadamard-examples} of the Supplementary Material.
\end{example}

\begin{example}\leavevmode\label{ex:non-hadamard-main}
    The above conditions are not restricted to spaces of nonpositive curvature. For $q\ge2$, a closed spherical cap of radius $\rho\in(0,\pi/4)$ in the $q$-dimensional unit sphere $\mathbb S^q$ satisfies Conditions~\ref{ass:M1}, \ref{ass:M2}, \ref{ass:M3}, and \ref{ass:M4} with $\kappa_{\mathcal Y}=2\rho\cot(2\rho)$, $\sigma=2\rho$, and $C_{\mathcal Y}=4\rho/\sin(2\rho)$. Similarly, for $q\in\mathbb N$, a closed ball of radius $\rho\in(0,\pi/4)$ in the space of probability densities on $\mathbb R^q$, equipped with the normalized Fisher--Rao metric, satisfies the same conditions with the same constants. Finally, for $q\ge2$ and $m,M\in(0,\infty)$ with $m<M$, the space of $q\times q$ symmetric positive-definite matrices whose eigenvalues lie in $[m,M]$, equipped with the Bures--Wasserstein metric, satisfies Conditions~\ref{ass:M1}, \ref{ass:M2}, \ref{ass:M3}, and \ref{ass:M4} with $\kappa_{\mathcal Y}=(m/M)^3$, $\sigma=\sqrt q\,(\sqrt M-\sqrt m)$, and $C_{\mathcal Y}=2(M/m)^2$. For more details, see Section~\ref{subsec:non-hadamard-examples} of the Supplementary Material.
\end{example}

We first establish a nonasymptotic bound for weighted empirical Fr\'echet means. 

\begin{theorem}\label{thm:r2-weighted-barycenter}
Suppose Conditions~\ref{ass:M1}(iii)--(iv), \ref{ass:M2}(iii)--(iv), and \ref{ass:M4} hold. Let $N\in\mathbb N$, and let $\widetilde{\bm Y}=(\widetilde Y_1,\ldots,\widetilde Y_N)\in\mathcal L_2(\mathcal Y)^N$ have independent components, and let $\bm\lambda\in\mathcal S_N$ be deterministic. Let $\widetilde{\bm Y}'=(\widetilde Y_1',\ldots,\widetilde Y_N')$ be an independent copy of $\widetilde{\bm Y}$. Then
\begin{align*}
    \E\left[d_{\mathcal Y}^2\left(\overline b_{\oplus,\bm\lambda}(\widetilde{\bm Y}),\overline q_{\oplus,\bm\lambda}(\widetilde{\bm Y})\right)\right]
    \le\frac{C_{\mathcal Y}^2}{8\kappa_{\mathcal Y}^2}\sum_{i=1}^N\lambda_i^2\E\left[d_{\mathcal Y}^2(\widetilde Y_i,\widetilde Y_i')\right].
\end{align*}
\end{theorem}

Theorem~\ref{thm:r2-weighted-barycenter} shows that the mean squared error is controlled by the squared weights and the second moments of the responses. In particular, under a common second-moment bound, the right-hand side is of order $\sum_{i=1}^N\lambda_i^2$, which reduces to $N^{-1}$ for equal weights. Note that the result allows independent observations with different distributions and unequal weights, and requires no response-space entropy condition.

Independently, \citet{Schoetz2026WeightedPower} established a similar result for CAT spaces, assuming that the response distributions are supported on a sufficiently small convex set in the positively curved case. In contrast, Theorem~\ref{thm:r2-weighted-barycenter} does not require a CAT curvature bound or a restriction on the support of the response distributions. In particular, it applies to a class of symmetric positive-definite matrices equipped with the Bures--Wasserstein metric; see Example~\ref{ex:bures-wasserstein-spd} of the Supplementary Material.

To apply Theorem~\ref{thm:r2-weighted-barycenter} to local regression, let $\widetilde Y_i$ have the conditional distribution of $Y$ given $X=x_i$ for $i=1,\ldots,N$, and let
\begin{align} \label{def:lambdai}
    \lambda_i=\frac{K_h(x_i,x_0)}{\sum_{j=1}^N K_h(x_j,x_0)}, \quad i=1,\ldots,N,
\end{align}
whenever the denominator is positive. Then $\overline q_{\oplus,\bm\lambda}(\widetilde{\bm Y})=\overline m_{\oplus,\bm\lambda}(\bm x)$. For random design, conditioning on the predictors makes $\lambda_1,\ldots,\lambda_N$ deterministic and $\widetilde Y_1,\ldots,\widetilde Y_N$ independent, so Theorem~\ref{thm:r2-weighted-barycenter} applies conditionally on the predictors. The corresponding conditional bound is given in Corollary~\ref{cor:r2-conditional} of the Supplementary Material.

For the analyses under Conditions~\ref{ass:S1} and \ref{ass:S2} below, for $\omega\in\mathcal Y$, define
\begin{align} \label{def:criterion-contrast}
    S_\omega(x):=
    \begin{cases}
        M_\oplus(x,\omega)-M_\oplus(x,m_\oplus(x_0)),
        & \text{if} \quad x\in\overline B(x_0,h_0),\\
        0,
        & \text{if} \quad  x\in\mathcal X\setminus\overline B(x_0,h_0).
    \end{cases}
\end{align}

\subsection{Nonasymptotic theory under second-order smoothness}\label{subsec:second-order-theory}

In this subsection, we derive a pointwise nonasymptotic error bound for the RLFR estimator under local second-order smoothness and establish a pointwise minimax lower bound. We then optimize the upper bound over the bandwidth and show that its resulting rate matches the minimax lower bound on an embedded Gaussian regression subclass. We first impose additional conditions on the kernel and the density of $X$.

\begin{rlfrcondition}{K2}\label{ass:K2}
The kernel $K:\R^d\to[0,\infty)$ is symmetric about zero, that is, $K(u)=K(-u)$ for all $u\in\R^d$.
\end{rlfrcondition}

\begin{rlfrcondition}{D2}\label{ass:D2}
The density $f_X$ is continuously differentiable on an open neighborhood of $\overline B(x_0,h_0)$.
\end{rlfrcondition}

Conditions~\ref{ass:K2} and \ref{ass:D2} are standard in kernel smoothing and local regression; see, for example, \citet{WandJones1995} and \citet{FanGijbels1996}. Condition~\ref{ass:K2} is satisfied by the symmetric kernels listed above.

Now, we introduce a local second-order smoothness condition.

\begin{rlfrcondition}{S1}\label{ass:S1}
Let $\kappa_{\mathcal Y}$ and $\sigma$ be the constants in Conditions~\ref{ass:M2} and \ref{ass:M3}, respectively. For every $\omega\in\overline B_{\mathcal Y}\left(m_\oplus(x_0),\sigma/\sqrt{\kappa_{\mathcal Y}}\right)$, the map $S_{\omega}$ defined by \eqref{def:criterion-contrast} is twice continuously differentiable on $B(x_0,h_0)$. Moreover, there exists $L_S\in(0,\infty)$ such that
\begin{align*}
    \|\nabla S_{\omega}(x_0)\|\le L_Sd_{\mathcal Y}\left(\omega,m_\oplus(x_0)\right), \quad
    \sup_{x\in B(x_0,h_0)}\|\nabla^2S_{\omega}(x)\|_{\mathrm{op}}\le L_Sd_{\mathcal Y}\left(\omega,m_\oplus(x_0)\right)
\end{align*}
uniformly over $\omega\in\overline B_{\mathcal Y}\left(m_\oplus(x_0),\sigma/\sqrt{\kappa_{\mathcal Y}}\right)$, where $\|\cdot\|_{\mathrm{op}}$ denotes the matrix operator norm.
\end{rlfrcondition}

Condition~\ref{ass:S1} controls the local second-order smoothness of $S_{\omega}$ with respect to the predictor for $\omega\in\overline B_{\mathcal Y}\left(m_\oplus(x_0),\sigma/\sqrt{\kappa_{\mathcal Y}}\right)$. Conditions~\ref{ass:M2}--\ref{ass:M3} ensure that $\widetilde m_h(x_0)$ and $\overline m_{\oplus,\bm\lambda}(\bm x)$ lie in this ball, so it is sufficient to impose the smoothness condition only there. In a Hilbert space, Condition~\ref{ass:S1} follows from bounded first and second derivatives of the conditional mean. A sufficient condition based on smooth conditional densities is given in Proposition~\ref{prop:pm-sufficient} of the Supplementary Material.

We now establish a pointwise nonasymptotic upper bound for the RLFR estimator under Condition~\ref{ass:S1}. For $\delta\in(0,1/2]$, define
\begin{align}\label{eq:Bdelta-binomial}
    B_\delta
    :=
    \min\left\{
    B\in\mathbb N:
    \Pbb\left(
    \operatorname{Bin}\left(B,\frac{11}{16}\right)
    \le
    \left\lfloor\frac B2\right\rfloor
    \right)
    \le\delta
    \right\},
\end{align}
where $\operatorname{Bin}\left(B,p\right)$ denotes the binomial random variable with $B\in\mathbb{N}$ trials and a common success probability $p \in (0,1)$. Then $B_\delta\asymp\log(1/\delta)$, with $B_{0.10}=11$, $B_{0.05}=19$, and $B_{0.01}=37$.

\begin{theorem}\label{thm:second-order-kernel-regression}
Suppose Conditions~\ref{ass:K1}--\ref{ass:K2}, \ref{ass:D1}--\ref{ass:D2}, \ref{ass:M1}--\ref{ass:M4}, and \ref{ass:S1} hold. There exist constants $C_{S,U},C_{S,U}'\in(0,\infty)$, depending only on the fixed constants in these conditions, such that the following holds. For $\delta\in(0,1/2]$ and $h\in(0,h_0)$, if $B_\delta\le\lfloor n/2\rfloor$ and $nh^df_X(x_0)\ge C_{S,U}'\log(1/\delta)$, then \begin{align}\label{eq:second-order-main-full}
    \Pbb\left(d_{\mathcal Y}\left(\widehat m^{\mathrm{RLFR}}_{h,B_\delta}(x_0),m_\oplus(x_0)\right)>C_{S,U}\left(h^2+\sigma\sqrt{\frac{\log(1/\delta)}{nh^df_X(x_0)}}\right)\right)\le\delta.
\end{align}
\end{theorem}

Note that the first term in \eqref{eq:second-order-main-full} is the bias, while the second term corresponds to the stochastic error. 

\begin{remark}
Although Theorem~\ref{thm:second-order-kernel-regression} is nonasymptotic, it also yields the classical pointwise rate for nonparametric regression. For fixed $\delta\in(0,1/2]$, balancing the two terms in \eqref{eq:second-order-main-full} gives
\begin{align*}
    h_{n,\delta}^{\mathrm{opt}}:=\left(\frac{\sigma^2\log(1/\delta)}{nf_X(x_0)}\right)^{1/(d+4)}.
\end{align*}
Then the conditions of Theorem~\ref{thm:second-order-kernel-regression} hold for all sufficiently large $n\in\mathbb{N}$, and
\begin{align*}
    \Pbb\left(d_{\mathcal Y}\left(\widehat m^{\mathrm{RLFR}}_{h_{n,\delta}^{\mathrm{opt}},B_\delta}(x_0),m_\oplus(x_0)\right)>2C_{S,U}\left(\frac{\sigma^2\log(1/\delta)}{nf_X(x_0)}\right)^{2/(d+4)}\right)\le\delta.
\end{align*}
Thus, for fixed confidence, the RLFR estimator attains the pointwise rate $n^{-2/(d+4)}$. For the details, see Corollary~\ref{cor:second-order-optimized-rate} in the Supplementary Material.
\end{remark}

We next establish a pointwise minimax lower bound on an embedded Gaussian regression subclass and compare it with the upper bound. Suppose that $\mathcal Y$ admits an isometric embedding of $\R$, that is, there exists $\iota:\R\to\mathcal Y$ such that
\begin{align*}
    d_{\mathcal Y}\left(\iota(s),\iota(t)\right)=|s-t|, \quad s,t\in\R.
\end{align*}
Several spaces listed in Example~\ref{ex:hadamard-main} admit such an embedding; see Section~\ref{subsec:hadamard-examples} of the Supplementary Material for details. Fix $\tau, L\in(0,\infty)$, and define
\begin{align*}
    \mathcal G_L:=\left\{g\in C^2(\R^d):\operatorname{supp}(g)\subset\overline B(x_0,h_0), \quad \|\nabla g(x_0)\|\le L, \quad \sup_{x\in B(x_0,h_0)}\|\nabla^2g(x)\|_{\mathrm{op}}\le L\right\}.
\end{align*}
For $g\in\mathcal G_L$, consider the embedded Gaussian regression model
\begin{align} \label{eq:embeddedgaussian}
    X\sim P_X, \quad Y=\iota\left(g(X)+\varepsilon\right), \quad \varepsilon\sim N\left(0,\tau^2\right), \quad \varepsilon\perp X.
\end{align}
Let $P_{g,\tau}$ denote the joint distribution of $(X,Y)$ under this model, so that $P_{g,\tau}^{\otimes n}$ is the distribution of $\mathcal D_n$.

\begin{theorem}\label{thm:minimax-lower-second-order}
Suppose Conditions~\ref{ass:D1}--\ref{ass:D2} and \ref{ass:M1}(iv) hold. There exist constants $C_{S,L},C_{S,L}'\in(0,\infty)$, depending only on the fixed constants in these conditions, such that the following holds. For $\delta\in(0,1/16]$ and $\tau, L\in(0,\infty)$, if $L\left(2h_0+h_0^2\right)\le2\tau$ and $nf_X(x_0)L^2\ge C_{S,L}'\tau^2\log(1/\delta)$, then, for every measurable map $\mathcal{T}_n:(\mathcal X\times\mathcal Y)^n\to\mathcal Y$,
\begin{align*}
    \sup_{g\in\mathcal G_L}
    P_{g,\tau}^{\otimes n}\left(
    d_{\mathcal Y}\left(\mathcal{T}_n(\mathcal D_n),\iota(g(x_0))\right)
    \ge
    C_{S,L}L^{d/(d+4)}
    \left(\frac{\tau^2\log(1/\delta)}{nf_X(x_0)}\right)^{2/(d+4)}
    \right)
    \ge\delta.
\end{align*}
\end{theorem}

Under the conditions of Proposition~\ref{prop:gaussian-subclass-upper} in the Supplementary Material, the conditions of Theorem~\ref{thm:second-order-kernel-regression} hold uniformly over $\mathcal G_L$. Hence, for fixed $L\in(0,\infty)$, the lower bound in Theorem~\ref{thm:minimax-lower-second-order} matches the optimized upper bound in Theorem~\ref{thm:second-order-kernel-regression} in terms of $n$ and $\delta$, so that the RLFR estimator attains the pointwise minimax rate on this subclass. 

\subsection{Nonasymptotic theory under $1$-Wasserstein H\"older continuity}\label{subsec:w1-main-summary}

In this subsection, we derive a pointwise nonasymptotic error bound for the RLFR estimator under local $1$-Wasserstein H\"older continuity. We then establish a pointwise minimax lower bound on an embedded Gaussian regression subclass and show that it matches the optimized rate of the RLFR estimator. Let $\mathcal P_1(\mathcal Y)$ denote the class of probability measures on $\mathcal Y$ with finite first moment. For $P_1,P_2\in\mathcal P_1(\mathcal Y)$, define
\begin{align*}
    W_1(P_1,P_2):=\inf_{\gamma\in\Pi(P_1,P_2)}\int_{\mathcal Y\times\mathcal Y}d_{\mathcal Y}(y_1,y_2)\,d\gamma(y_1,y_2),
\end{align*}
where $\Pi(P_1,P_2)$ denotes the set of couplings of $P_1$ and $P_2$; see, for example, Definition 6.1 of \cite{Villani2009}. We now introduce a local $1$-Wasserstein H\"older continuity condition.

\begin{rlfrcondition}{S2}\label{ass:S2}
There exist $\beta\in(0,1]$ and $L_W\in(0,\infty)$ such that
\begin{align*}
    W_1\left(P_{Y\mid X=x},P_{Y\mid X=x_0}\right) \le L_W\|x-x_0\|^\beta, \quad x\in\overline B(x_0,h_0).
\end{align*}
\end{rlfrcondition}

Condition~\ref{ass:S2} assumes H\"older continuity of the conditional distributions with respect to the $1$-Wasserstein distance. A similar condition is used in distributional regression by \citet{DombryModestePic2025}. Unlike Condition~\ref{ass:S1}, Condition~\ref{ass:S2} does not require differentiability with respect to the predictor and therefore allows certain non-smooth regression models. For example, in a scalar location model $Y=g(X)+\varepsilon$ with $\varepsilon\perp X$, Condition~\ref{ass:S2} follows from H\"older continuity of $g$, even when $g$ is not differentiable at $x_0$.

We now establish a pointwise nonasymptotic upper bound for the RLFR estimator under Condition~\ref{ass:S2}.

\begin{theorem}\label{thm:w1-kernel-regression}
Suppose Conditions~\ref{ass:K1}, \ref{ass:D1}, \ref{ass:M1}--\ref{ass:M4}, and \ref{ass:S2} hold. There exist constants $C_{W,U},C_{W,U}'\in(0,\infty)$, depending only on the constants defined in Conditions~\ref{ass:K1}, \ref{ass:D1}, and \ref{ass:M1}--\ref{ass:M4}, such that the following holds. For $\delta\in(0,1/2]$ and $h\in(0,h_0)$, if $B_\delta\le\lfloor n/2\rfloor$ and $nh^df_X(x_0)\ge C_{W,U}'\log(1/\delta)$, then
\begin{align}\label{eq:w1-main-summary}
    \Pbb\left(d_{\mathcal Y}\left(\widehat m^{\mathrm{RLFR}}_{h,B_\delta}(x_0),m_\oplus(x_0)\right)>C_{W,U}\left(L_Wh^\beta+\sigma\sqrt{\frac{\log(1/\delta)}{nh^df_X(x_0)}}\right)\right)\le\delta.
\end{align}
\end{theorem}

The first term in \eqref{eq:w1-main-summary} is the bias, while the second term is the stochastic error. Unlike Theorem~\ref{thm:second-order-kernel-regression}, Theorem~\ref{thm:w1-kernel-regression} does not require Conditions~\ref{ass:K2} and \ref{ass:D2}.

\begin{remark}
Although Theorem~\ref{thm:w1-kernel-regression} is nonasymptotic, it also yields the classical pointwise H\"older rate for nonparametric regression. For fixed $\delta\in(0,1/2]$, balancing the two terms in \eqref{eq:w1-main-summary} gives
\begin{align*}
    h_{n,\delta}^{\mathrm{opt}}:=\left(\frac{\sigma^2\log(1/\delta)}{L_W^2nf_X(x_0)}\right)^{1/(2\beta+d)}.
\end{align*}
Then the conditions of Theorem~\ref{thm:w1-kernel-regression} hold for all sufficiently large $n\in\mathbb N$, and
\begin{align*}
    \Pbb\left(d_{\mathcal Y}\left(\widehat m^{\mathrm{RLFR}}_{h_{n,\delta}^{\mathrm{opt}},B_\delta}(x_0),m_\oplus(x_0)\right)>2C_{W,U}L_W^{d/(2\beta+d)}\left(\frac{\sigma^2\log(1/\delta)}{nf_X(x_0)}\right)^{\beta/(2\beta+d)}\right)\le\delta.
\end{align*}
Thus, for fixed confidence, the RLFR estimator attains the pointwise H\"older rate $n^{-\beta/(2\beta+d)}$. For the details, see Corollary~\ref{cor:w1-optimized-rate} in the Supplementary Material.
\end{remark}

We next establish a pointwise minimax lower bound on an embedded Gaussian regression subclass and compare it with the upper bound. Suppose that $\mathcal Y$ admits an isometric embedding of $\R$. Fix $\tau, L\in(0,\infty)$, $\beta\in(0,1]$, and define
\begin{align*}
    \mathcal G_{\beta,L}:=\left\{g\in C(\R^d):\operatorname{supp}(g)\subset\overline B(x_0,h_0), \quad \sup_{u,v\in\overline B(x_0,h_0), u\ne v}\frac{|g(u)-g(v)|}{\|u-v\|^\beta}\le L\right\}.
\end{align*}
For $g\in\mathcal G_{\beta,L}$, let $P_{g,\tau}$ denote the joint distribution under the embedded Gaussian regression model defined in \eqref{eq:embeddedgaussian}, so that $P_{g,\tau}^{\otimes n}$ is the distribution of $\mathcal D_n$.

\begin{theorem}\label{thm:minimax-lower-w1}
Suppose Condition~\ref{ass:D1} and Condition~\ref{ass:M1}(iv) hold. There exist constants $C_{W,L},C_{W,L}'\in(0,\infty)$, depending only on the fixed constants in these conditions, such that the following holds. For $\delta\in(0,1/16]$, $\tau, L\in(0,\infty)$, and $\beta\in(0,1]$, if $L h_0^\beta\le\tau$ and $nf_X(x_0)L^2\ge C_{W,L}'\tau^2\log(1/\delta)$, then, for every measurable map $\mathcal T_n:(\mathcal X\times\mathcal Y)^n\to\mathcal Y$,
\begin{align}\label{eq:minimax-lower-w1}
    \sup_{g\in\mathcal G_{\beta,L}}
    P_{g,\tau}^{\otimes n}\left(
    d_{\mathcal Y}\left(
    \mathcal T_n(\mathcal D_n),
    \iota(g(x_0))
    \right)
    \ge
    C_{W,L}
    L^{d/(2\beta+d)}
    \left(
    \frac{\tau^2\log(1/\delta)}
    {nf_X(x_0)}
    \right)^{\beta/(2\beta+d)}
    \right)
    \ge\delta.
\end{align}
\end{theorem}

Under the conditions of Proposition~\ref{prop:gaussian-holder-subclass-upper} in the Supplementary Material, the conditions of Theorem~\ref{thm:w1-kernel-regression} hold uniformly over $\mathcal G_{\beta,L}$. Hence, for fixed $\beta\in(0,1]$ and $L\in(0,\infty)$, the lower bound in Theorem~\ref{thm:minimax-lower-w1} matches the optimized upper bound in Theorem~\ref{thm:w1-kernel-regression} in terms of $n$ and $\delta$, so that the RLFR estimator attains the pointwise minimax rate on this subclass.

\section{Simulation study}
\label{sec:simulation}

We examine the performance of the proposed method for responses in the space $\mathcal S_{++}^2$ of $2\times2$ symmetric positive-definite matrices, equipped with the affine-invariant metric; see Example~\ref{ex:spd-airm} in the Supplementary Material. We generate the predictors from the density
\begin{align*}
    f_X(x)\propto
    \left\{\exp\left(-\frac{(x-0.6)^2}{2(0.25)^2}\right)
    +\exp\left(-\frac{(x+0.6)^2}{2(0.25)^2}\right)\right\},\quad x\in[-1, 1].
\end{align*}
Let
\begin{align*}
    E_1=\frac{1}{\sqrt{2}}\begin{pmatrix}1&0\\0&-1\end{pmatrix},
    \quad
    E_2=\frac{1}{\sqrt{2}}\begin{pmatrix}0&1\\1&0\end{pmatrix},\quad
    S=\exp\left(0.4\begin{pmatrix}
    \cos(0.3)&\sin(0.3)\\
    \sin(0.3)&-\cos(0.3)
    \end{pmatrix}\right).
\end{align*}
For $i=1,\ldots,n$, conditional on $X_i=x$, we independently generate
\begin{align*}
    Z_{i0},Z_{i1},Z_{i2}\sim N(0,1),
    \quad T_i\sim t_3,\quad
    J_i\sim\mathrm{Bernoulli}\{0.01(1-x^2)\},\quad x\in[-1, 1],
\end{align*}
where $t_3$ denotes the Student $t$ distribution with three degrees of freedom. We then set
\begin{align*}
    \epsilon_i&=0.01Z_{i0}+\sqrt{30}\,J_iT_i,\\
    Y_i&=\exp\left(x^3+\frac{\epsilon_i}{\sqrt{2}}\right)S^{1/2}\exp\left(0.004(Z_{i1}E_1+Z_{i2}E_2)\right)S^{1/2},\quad x\in[-1, 1].
\end{align*}
The conditional Fr\'echet mean is
\begin{align*}
    m_{\oplus}(x)=e^{x^3}S,\quad x\in[-1, 1].
\end{align*}
Note that the distance $d(Y_i,m_{\oplus}(x))$ has a finite second moment but an infinite third moment for $x\in(-1, 1)$. We consider $n\in\{200,400,800\}$ and use $R=5000$ Monte Carlo replications.

We compare RLFR with the local constant (LC) and local linear (LL) Fr\'echet regression estimators of \citet{PetersenMuller2019}, and with the Fr\'echet-median aggregation estimator of \citet{HongXuWood2026}. For RLFR and Hong's estimator, we take $B=11,19,37$ and use the same random partitions into blocks of sizes differing by at most one. 

We use the Epanechnikov kernel and select bandwidths by 5-fold cross-validation over the common candidate grid
\begin{align*}
    \mathcal H_n:=\{0.1,0.2,\ldots,2.0\}\,n^{-1/5}.
\end{align*}
For RLFR and Hong's estimator, bandwidth selection is carried out separately for each block number $B$. Since LC and LL do not use block aggregation, they do not depend on $B$. We evaluate performance using the integrated squared error
\begin{align*}
    \mathrm{ISE}^{[r]}
    :=\frac12\int_{-1}^{1}
    d^2\{\hat m_{\oplus}^{[r]}(x),m_{\oplus}(x)\}\,\mathrm dx,
    \qquad
    \mathrm{MISE}:=\frac1R\sum_{r=1}^{R}\mathrm{ISE}^{[r]},
\end{align*}
where $\hat m_{\oplus}^{[r]}$ denotes an estimator computed from the $r$th replication.

Table~\ref{tab:sim_mise} reports the MISEs. RLFR has a smaller MISE than Hong's estimator in eight of the nine settings and than LC in all settings. Compared with LL, RLFR has a smaller MISE except when $B=37$ and $n\in\{400,800\}$. The large MISE of LL at $n=200$ is dominated by one replication with nearly coincident local predictors and large signed extrapolation weights.

\begin{table}[!ht]
\centering
\caption{MISEs for the positive-definite matrix responses based on $5000$ Monte Carlo replications. Bold values indicate the smallest MISE in each row, and underlined values indicate the second smallest.}
\label{tab:sim_mise}
\begin{tabular}{cccccc}
\toprule
$n$ & $B$ & LC & LL & RLFR & Hong \\
\midrule
200 & 11 & 0.08801 & 0.66304 & \textbf{0.01300} & \underline{0.01416} \\
200 & 19 & 0.08801 & 0.66304 & \textbf{0.02292} & \underline{0.02730} \\
200 & 37 & 0.08801 & 0.66304 & \textbf{0.05097} & \underline{0.06216} \\
\midrule
400 & 11 & 0.01983 & 0.00778 & \textbf{0.00426} & \underline{0.00452} \\
400 & 19 & 0.01983 & \underline{0.00778} & \textbf{0.00759} & 0.00898 \\
400 & 37 & 0.01983 & \textbf{0.00778} & \underline{0.01795} & 0.02350 \\
\midrule
800 & 11 & 0.01003 & 0.00489 & \underline{0.00131} & \textbf{0.00130} \\
800 & 19 & 0.01003 & 0.00489 & \textbf{0.00202} & \underline{0.00235} \\
800 & 37 & 0.01003 & \textbf{0.00489} & \underline{0.00513} & 0.00707 \\
\bottomrule
\end{tabular}
\end{table}

Table~\ref{tab:sim_tail} reports the upper quantiles of ISE. Following the block-count calibration in \eqref{eq:Bdelta-binomial}, we pair $B=11,19,37$ with the empirical $90$th, $95$th, and $99$th percentiles, respectively. For these pairings, RLFR reduces the upper quantiles relative to LC by approximately $55\%$--$90\%$. Its upper quantiles are smaller than those of LL in seven settings and Hong's estimator in six settings. The differences from Hong's estimator are within approximately $7.2\%$ in all settings. Thus, RLFR provides substantial improvements over LC in upper-tail error and broadly comparable performance to Hong's estimator.

\begin{table}[!ht]
\centering
\caption{Upper quantiles of ISE for the positive-definite matrix responses. Bold values indicate the smallest value in each row, and underlined values indicate the second smallest.}
\label{tab:sim_tail}
\begin{tabular}{clcccc}
\toprule
$n$ & Tail level & LC & LL & RLFR & Hong \\
\midrule
200 & $q_{0.90}$ ($B=11$) & 0.06008 & \textbf{0.02361} & \underline{0.02382} & 0.02419 \\
200 & $q_{0.95}$ ($B=19$) & 0.10259 & \textbf{0.03953} & \underline{0.04596} & 0.04652 \\
200 & $q_{0.99}$ ($B=37$) & 0.48310 & 0.13250 & \underline{0.10476} & \textbf{0.09958} \\
\midrule
400 & $q_{0.90}$ ($B=11$) & 0.03569 & 0.01621 & \textbf{0.00755} & \underline{0.00759} \\
400 & $q_{0.95}$ ($B=19$) & 0.05806 & 0.02557 & \textbf{0.01670} & \underline{0.01673} \\
400 & $q_{0.99}$ ($B=37$) & 0.15410 & 0.08387 & \textbf{0.04503} & \underline{0.04519} \\
\midrule
800 & $q_{0.90}$ ($B=11$) & 0.02125 & 0.01014 & \underline{0.00220} & \textbf{0.00205} \\
800 & $q_{0.95}$ ($B=19$) & 0.03100 & 0.01496 & \underline{0.00495} & \textbf{0.00472} \\
800 & $q_{0.99}$ ($B=37$) & 0.06205 & 0.03323 & \textbf{0.01357} & \underline{0.01383} \\
\bottomrule
\end{tabular}
\end{table}

Table~\ref{tab:sim_time} reports the mean computation times, including bandwidth selection and final estimation. For each $n\in\{200, 400, 800\}$, we use three independently generated datasets and run each method twice on each dataset, reversing the execution order in the second run. The computation times of Hong's estimator are approximately $1.46$--$6.79$ times those of RLFR. RLFR is also faster than LL, while LC is the fastest method in all settings.

\begin{table}[!ht]
\centering
\caption{Mean computation times in seconds, including 5-fold cross-validation and final estimation. Each entry averages two runs on each of three independently generated datasets. Bold values indicate the smallest value in each row, and underlined values indicate the second smallest.}
\label{tab:sim_time}
\begin{tabular}{cccccc}
\toprule
$n$ & $B$ & LC & LL & RLFR & Hong \\
\midrule
200 & 11 & \textbf{0.279} & 1.001 & \underline{0.346} & 0.806 \\
200 & 19 & \textbf{0.279} & 1.001 & \underline{0.428} & 1.821 \\
200 & 37 & \textbf{0.279} & 1.001 & \underline{0.788} & 5.351 \\
\midrule
400 & 11 & \textbf{1.000} & 3.467 & \underline{1.146} & 2.172 \\
400 & 19 & \textbf{1.000} & 3.467 & \underline{1.357} & 4.188 \\
400 & 37 & \textbf{1.000} & 3.467 & \underline{2.164} & 11.108 \\
\midrule
800 & 11 & \textbf{3.683} & 12.374 & \underline{3.983} & 5.809 \\
800 & 19 & \textbf{3.683} & 12.374 & \underline{4.336} & 9.721 \\
800 & 37 & \textbf{3.683} & 12.374 & \underline{6.086} & 23.789 \\
\bottomrule
\end{tabular}
\end{table}

\section{Real-data analysis}\label{subsec:realdata_mesa}

We analyze wrist-actigraphy data from the Multi-Ethnic Study of Atherosclerosis (MESA) Sleep Ancillary Study; see \citet{ChenEtAl2015} and \citet{ZhangEtAl2018}. The actigraphy recordings were collected over approximately seven days and are available at 30-second resolution. We take body mass index (BMI) as the scalar predictor and construct a distribution-valued response from each participant's waking activity counts. For each participant, we form the empirical distribution of $\log(1+\text{activity count})$ over the qualifying epochs. After preprocessing, the analysis sample consists of $1972$ participants. We regard these empirical distributions as elements of the one-dimensional Wasserstein space $\mathcal W_2$.

Figure~\ref{fig:mesa_eda} displays the empirical quantile functions for
$40$ randomly selected participants together with the sample Wasserstein
Fr\'echet mean. For interpretability, the curves are back-transformed to the original activity-count scale. The observed distributions vary across participants, with larger differences in the upper quantiles. This pattern motivates the use of robust local Fr\'echet regression.

\begin{figure}[!ht]
\centering
\includegraphics[width=\textwidth]{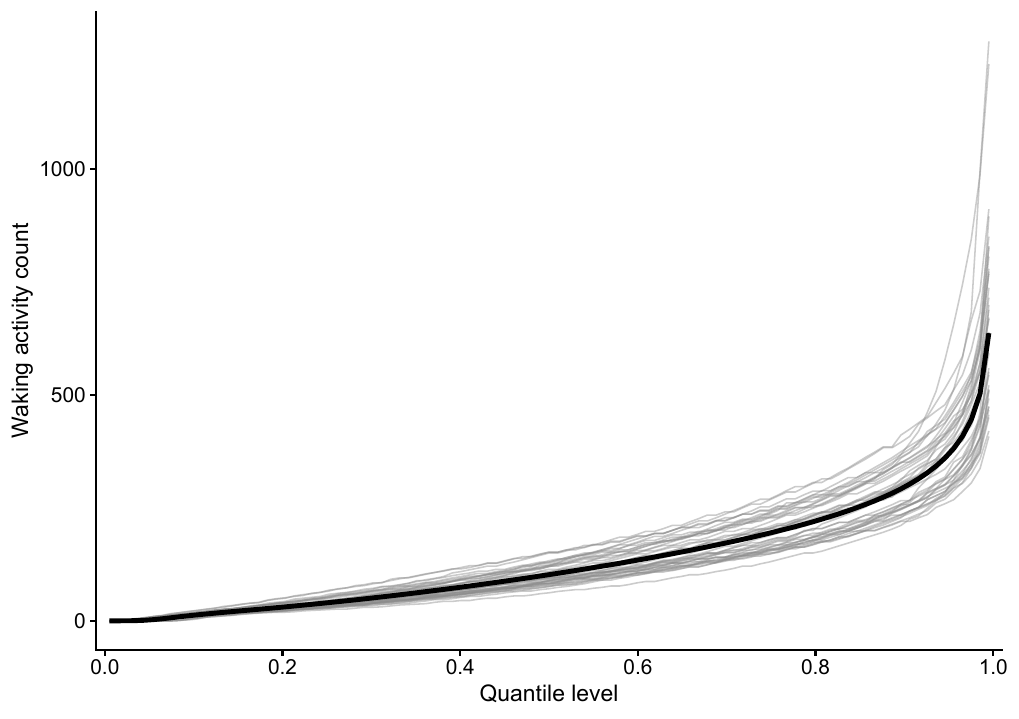}
\caption{
Empirical quantile functions of waking activity counts for $40$ randomly selected MESA participants. The black curve shows the sample Wasserstein Fr\'echet mean computed from all $1972$ participants after back-transformation to the original activity-count scale.
}
\label{fig:mesa_eda}
\end{figure}

We compare the proposed RLFR estimator with the local constant (LC) and local linear (LL) Fr\'echet regression estimators of \citet{PetersenMuller2019}, and with the Fr\'echet-median aggregation estimator of \citet{HongXuWood2026}. For RLFR and Hong's estimator, we consider $B=11,19,$ and $37$. Since LC and LL do not use block aggregation, their estimates do not depend on $B$. We use the Epanechnikov kernel and the common bandwidth grid
\begin{align*}
    \mathcal H:=\{0.1,0.2,\ldots,2.0\}
\end{align*}
for all methods. Bandwidths are selected separately for each method, and for each value of $B$ for RLFR and Hong's estimator, using $5\times5$ nested cross-validation with mean squared Wasserstein prediction loss as the selection criterion.

Table~\ref{tab:mesa_aspe} reports the average squared prediction errors (ASPEs) obtained from the pooled outer-fold predictions. The values are very similar across all methods. In particular, the ASPEs of RLFR differ only slightly from those of LC, LL, and Hong's estimator for all three choices of $B$, indicating that the robust aggregation does not materially sacrifice average predictive performance.

\begin{table}[!ht]
\centering
\caption{
Average squared Wasserstein prediction errors (ASPEs) for the MESA actigraphy data. Bold values indicate the smallest value in each row, and underlined values indicate the second smallest.
}
\label{tab:mesa_aspe}
\begin{tabular}{ccccc}
\toprule
$B$ & LC & LL & RLFR & Hong \\
\midrule
11 & \underline{0.09595} & \textbf{0.09584}
   & 0.09646 & 0.09609 \\
19 & 0.09595 & \textbf{0.09584}
   & 0.09607 & \underline{0.09591} \\
37 & \underline{0.09595} & \textbf{0.09584}
   & 0.09681 & 0.09598 \\
\bottomrule
\end{tabular}
\end{table}

Table~\ref{tab:mesa_tail} reports the upper quantiles of the squared prediction error. Using \eqref{eq:Bdelta-binomial}, we pair $B=11,19,$ and $37$ with the empirical $90$th, $95$th, and $99$th percentiles, respectively. At $B=11$, RLFR has a slightly smaller $90$th percentile error than LC and Hong's estimator and is nearly indistinguishable from LL. For these pairings, the reductions in tail error relative to LC are approximately $0.4\%$, $1.4\%$, and $3.6\%$ for $B=11$, $19$, and $37$, respectively, while the corresponding reductions relative to Hong's estimator are approximately $0.3\%$, $1.0\%$, and $1.8\%$. Thus, while the methods have similar average prediction errors, the advantage of RLFR becomes more apparent at higher quantiles of the prediction-error distribution.

\begin{table}[!ht]
\centering
\caption{
Upper quantiles of the squared Wasserstein prediction error for the MESA actigraphy data, with $B=11,19,$ and $37$ paired with the $90$th, $95$th, and $99$th percentiles, respectively. Bold values indicate the smallest value in each row, and underlined values indicate the second smallest.
}
\label{tab:mesa_tail}
\begin{tabular}{lcccc}
\toprule
Tail level & LC & LL & RLFR & Hong \\
\midrule
$q_{0.90}$ ($B=11$) & 0.21071 & \textbf{0.20959} & \underline{0.20983} & 0.21042 \\
$q_{0.95}$ ($B=19$) & 0.27340 & \underline{0.27152} & \textbf{0.26961} & 0.27241 \\
$q_{0.99}$ ($B=37$) & 0.49487 & 0.49466 & \textbf{0.47716} & \underline{0.48605} \\
\bottomrule
\end{tabular}
\end{table}

Figure~\ref{fig:mesa_fitted} displays the fitted conditional quantile functions using RLFR with $B=19$ at the clinically interpretable BMI values $20$, $25$, $30$, $35$, and $40$, back-transformed to the original activity-count scale. The fitted curves shift downward as BMI increases. The separation is visible over a broad range of quantile levels and remains substantial in the upper part of the activity distribution. Thus, higher BMI is associated with a lower distribution of waking activity counts, including reduced activity at moderate and relatively high activity levels. 

\begin{figure}[!ht]
\centering
\includegraphics[width=\textwidth]{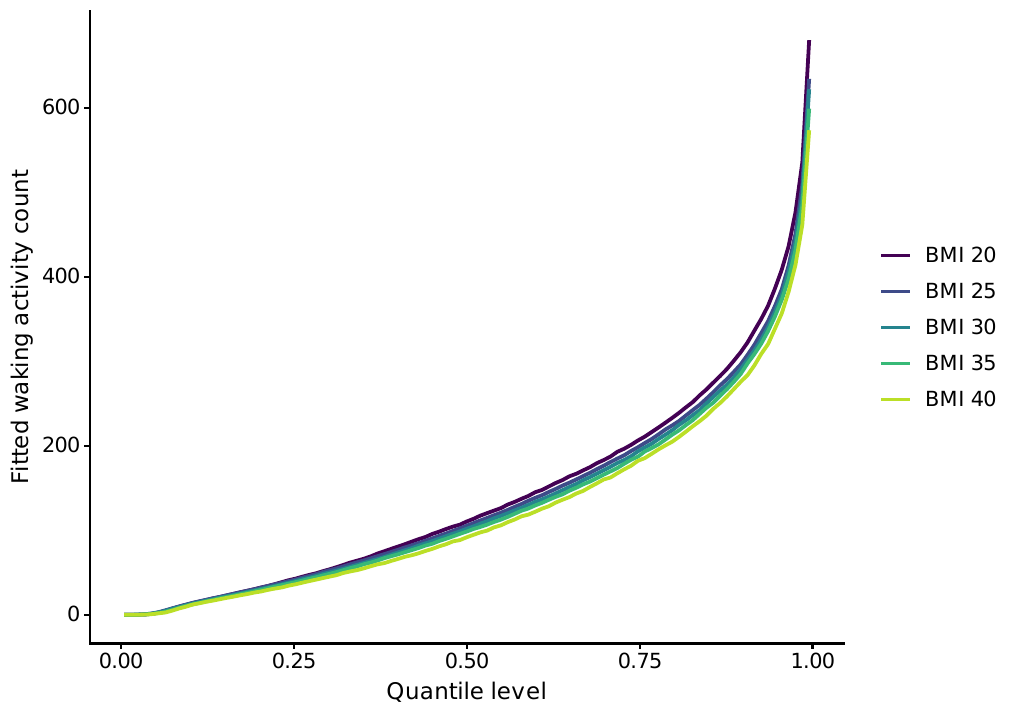}
\caption{
Fitted conditional quantile functions of waking activity counts at BMI values $20$, $25$, $30$, $35$, and $40$, obtained using RLFR with $B=19$. The curves are displayed on the original activity-count scale.
}
\label{fig:mesa_fitted}
\end{figure}

\section{Discussion}\label{sec:discussion}

We have established nonasymptotic pointwise high-probability bounds for robust local Fr\'echet regression under a local second-moment condition. A key ingredient is a nonasymptotic bound for weighted empirical Fr\'echet means that allows unequal weights and different distributions. After conditioning on the predictors, this result applies directly to the block estimators. Under either local second-order smoothness or local $1$-Wasserstein H\"older continuity, the optimized bounds attain the corresponding classical pointwise rates for nonparametric regression. Matching pointwise minimax lower bounds are obtained on embedded Gaussian regression subclasses.

The simulations illustrate that RLFR combines robustness to heavy-tailed responses with computational efficiency. In the real-data analysis, RLFR yields similar average prediction errors and modest improvements at the higher quantiles. These results illustrate its practical benefits without implying uniform superiority over the competing methods.

The present results concern pointwise local-constant Fr\'echet regression. Extending the present general metric-space theory to local-linear Fr\'echet regression would require treating signed weights, since local-linear equivalent weights need not be nonnegative. Uniform theory and data-driven bandwidth selection are also left for future work.

\section*{Acknowledgements}
Chang Jun Im was supported by the National Research Foundation of Korea grant funded by the Korea government (MSIT) (No.\ RS-2025-00515381). Daisuke Kurisu was partially supported by JSPS KAKENHI (No.\ 26K00323).

\bibliographystyle{bibliography_style}
\bibliography{Ref}

\ifcompilewithsupplement
\clearpage
\appendix
\begin{center}
{\Large\bfseries Supplementary Material for\par}
\vspace{2mm}
{\Large\bfseries \PaperTitle\par}
\end{center}
\vspace{0.5cm}

\section{Examples of response spaces}\label{sec:examples-scope}

\subsection{Hadamard spaces}\label{subsec:hadamard-examples}

A Hadamard space is a complete CAT(0) space. In a separable Hadamard space, every finite-second-moment distribution has a unique Fr\'echet mean, and the variance and quadruple inequalities hold with $\kappa_{\mathcal Y}=1$ and $C_{\mathcal Y}=2$; see Propositions 2.4, 4.3, and 4.4 of \cite{Sturm2003}. Together with Condition~\ref{ass:M3}, these results verify Conditions~\ref{ass:M1}--\ref{ass:M2} and \ref{ass:M4}. Condition~\ref{ass:M3} requires a distributional condition when $\mathcal Y$ is unbounded.

\begin{example}\leavevmode
Let $(\mathbb H,\|\cdot\|_{\mathbb H})$ be a nontrivial separable Hilbert space. Then $\mathbb H$ is a Hadamard space; see Proposition 3.5(i) of \cite{Sturm2003}. If $v\in\mathbb H$ satisfies $\|v\|_{\mathbb H}=1$, the map $t\mapsto tv$ is an isometric embedding of $\R$.
\end{example}

\begin{example}\leavevmode
Every complete separable metric tree is a Hadamard space; see Proposition 3.4 of \cite{Sturm2003}. Such a tree need not contain an isometric copy of $\R$. For example, $[0,1]$ with its usual distance is a complete metric tree and is bounded. The embedding condition used in the lower bounds must therefore be checked separately for metric trees.
\end{example}

\begin{example}\leavevmode
For $m\in\mathbb{N}$ with $m\ge 4$, let $\mathcal T_m$ denote the Billera--Holmes--Vogtmann (BHV) space of phylogenetic trees equipped with the BHV metric $d_{\mathrm{BHV}}$. The space $(\mathcal T_m,d_{\mathrm{BHV}})$ is a complete separable CAT(0) space and hence a Hadamard space; see \cite{BilleraHolmesVogtmann2001}. Consider the two incompatible splits
\begin{align*}
    S_1:=\{1,2\}\mid\{3,\ldots,m\}, \quad
    S_2:=\{1,3\}\mid\{2,4,\ldots,m\},
\end{align*}
which cannot both be induced by interior edges of the same tree. For $t\in(0,\infty)$ and $j=1,2$, let $T_j(t)$ denote the tree whose unique positive-length interior edge induces the split $S_j$ and has length $t$. At $t=0$, set $T_1(0)=T_2(0)=T_\star$, where $T_\star$ denotes the star tree on the labeled leaves $\{1,\ldots,m\}$. Define $\iota:\R\to\mathcal T_m$ by
\begin{align*}
    \iota(t):=
    \begin{cases}
        T_1(t), & t\ge0,\\
        T_2(-t), & t<0.
    \end{cases}
\end{align*}
Since the two splits are incompatible, the BHV geodesic between $T_1(s)$ and $T_2(t)$ passes through $T_\star$ for any $s,t\in[0,\infty)$. Hence, points on the same ray have BHV distance equal to the difference of their edge lengths, whereas points on different rays have BHV distance equal to the sum of their edge lengths. Therefore,
\begin{align*}
    d_{\mathrm{BHV}}\left(\iota(s),\iota(t)\right)=|s-t|, \quad s,t\in\R,
\end{align*}
so $\iota$ is an isometric embedding of $\R$ into $\mathcal T_m$.
\end{example}

\begin{example}\label{ex:spd-airm}\leavevmode
For $q\in\mathbb N$, let $\mathcal S_{++}^q$ be the space of $q\times q$ symmetric positive-definite matrices equipped with the affine-invariant metric
\begin{align*}
    d_{\mathrm{AI}}(A,B):=\left\|\log\left(A^{-1/2}BA^{-1/2}\right)\right\|_{\mathrm F}, \quad A,B\in\mathcal S_{++}^q,
\end{align*}
where $\|\cdot\|_{\mathrm F}$ is the Frobenius norm. Then, $\mathcal S_{++}^q$ is a finite-dimensional Hadamard space; see Proposition~5 of \cite{BhatiaHolbrook2006}. For fixed $A\in\mathcal S_{++}^q$, define $\iota:\R\to\mathcal S_{++}^q$ by
\begin{align*}
    \iota(t)=A^{1/2}\exp\left(\frac{t}{\sqrt{q}}I_q\right)A^{1/2}, \quad t\in\R,
\end{align*}
where $I_q$ is the $q\times q$ identity matrix. By the invariance of $d_{\mathrm{AI}}$ under congruence transformations,
\begin{align*}
    d_{\mathrm{AI}}(\iota(s),\iota(t))
    =d_{\mathrm{AI}}\left(\exp\left(\frac{s}{\sqrt{q}}I_q\right),\exp\left(\frac{t}{\sqrt{q}}I_q\right)\right)
    =\frac{|s-t|}{\sqrt{q}}\|I_q\|_{\mathrm F}=|s-t|, \quad s,t\in\R.
\end{align*}
Hence, $\iota$ is an isometric embedding of $\R$ into
$\mathcal S_{++}^q$.
\end{example}

\begin{example}\label{ex:Wass}\leavevmode
For $\mu\in\mathcal W_2(\R)$, let $F_\mu^{-1}$ denote its quantile function. The quantile map identifies $\mathcal W_2(\R)$ with the closed convex subset of $L^2(0,1)$ consisting of functions with a nondecreasing representative, since
\begin{align*}
    W_2^2(\mu,\nu)=\int_0^1\left(F_\mu^{-1}(u)-F_\nu^{-1}(u)\right)^2\,du, \quad \mu,\nu\in\mathcal W_2(\R).
\end{align*}
Thus $\mathcal W_2(\R)$ is a Hadamard space; see Proposition~2.2.8 and Section~2.3.2 of \cite{PanaretosZemel2020}. For $t\in\mathbb{R}$, let $\delta_t$ denote the Dirac measure at $t$. A map $t\mapsto\delta_t$ is an isometric embedding of $\R$, since $W_2(\delta_s,\delta_t)=|s-t|$ for any $s, t \in \mathbb{R}$. However, the response-space entropy condition used in \cite{PetersenMuller2019} does not hold in general on $\mathcal W_2(\R)$, as shown by the counterexample on $\mathcal W_2([0,1])$; see Proposition~2.1 of \cite{ImJeonPark2026Entropy}.
\end{example}

\begin{example}\label{ex:ccset}
For $q\in\mathbb N$, let $\mathcal{K}_c\left(\mathbb{R}^q\right)$ denote the collection of nonempty compact and convex sets in $\mathbb{R}^q$. For any $F\in\mathcal{K}_c\left(\mathbb{R}^q\right)$, define the support function by
\begin{align*}
    s_F(u):=\sup_{x\in F}\langle x,u\rangle, \quad u\in\mathbb{S}^{q-1},
\end{align*}
where $\langle\cdot,\cdot\rangle$ denotes the Euclidean inner product. Let $\sigma_{q-1}$ denote the surface area measure on the unit sphere $\mathbb{S}^{q-1}$. We define a metric $d_c$ on $\mathcal{K}_c\left(\mathbb{R}^q\right)$ by
\begin{align*}
    d_c^2(F,G):=\int_{\mathbb{S}^{q-1}}\left(s_F(u)-s_G(u)\right)^2\,d\sigma_{q-1}(u), \quad F,G\in\mathcal{K}_c\left(\mathbb{R}^q\right).
\end{align*}
Then $\left(\mathcal{K}_c\left(\mathbb{R}^q\right),d_c\right)$ is a Hadamard space; see \cite{kurisu2025random}. Moreover, the map $\iota:\mathbb{R}\to\mathcal{K}_c\left(\mathbb{R}^q\right)$ defined by
\begin{align*}
    \iota(t):=\left\{\left(\frac{q}{\sigma_{q-1}(\mathbb{S}^{q-1})}\right)^{1/2}t e_1\right\}, \quad t\in\mathbb{R},
\end{align*}
is an isometric embedding, where $e_1=(1,0,\ldots,0)\in\mathbb{R}^q$.
\end{example}

\begin{example}\leavevmode
Let $(\mathcal M,g)$ be a complete simply connected Riemannian manifold with nonpositive sectional curvature, and let $d_g$ be its geodesic distance. Then $(\mathcal M,d_g)$ is a Hadamard space; see Proposition 3.1 of \cite{Sturm2003}. Every complete unit-speed geodesic is an isometric embedding of $\R$. In particular, real hyperbolic spaces of positive dimension admit such an embedding.
\end{example}

\subsection{Non-Hadamard spaces}\label{subsec:non-hadamard-examples}

\begin{example}\label{ex:spherical-cap}\leavevmode
For $q\ge2$, let $\mathbb S^q$ be the $q$-dimensional unit sphere with its geodesic metric $d_{\mathbb S^q}$ defined by
\begin{align*}
    d_{\mathbb S^q}(y,\omega):=\arccos\langle y,\omega\rangle, \quad y,\omega\in\mathbb S^q.
\end{align*}
Fix $p\in\mathbb S^q$ and $\rho\in(0,\pi/4)$, and let $\mathcal Y=\overline B_{\mathbb S^q}(p,\rho)$ be a spherical cap. Note that $\mathcal Y$ is compact and geodesically convex, and its diameter is at most $2\rho$. Hence Condition~\ref{ass:M3} holds with $\sigma=2\rho$.

Fix $y\in\mathcal Y$ and define
\begin{align*}
    F_y(\omega):=d_{\mathbb S^q}^2(y,\omega), \quad \omega\in\mathcal Y.
\end{align*}
Since $d_{\mathbb S^q}(y,\omega)\le 2\rho<\pi/2$, the logarithm $\log_\omega (y)$ is well defined for any $\omega\in\mathcal Y$. By the first variation formula, the gradient of $F_y$ is
\begin{align} \label{eq:GradientofFy}
    \nabla F_y(\omega)=-2\log_\omega(y), \quad \omega\in\mathcal Y.
\end{align}
For $\omega\in\mathcal Y$ such that $\omega\ne y$, the Hessian of $F_y$ is
\begin{align}
\begin{split}
    \operatorname{Hess}F_y(\omega)[v,v]
    &=2d_{\mathbb S^q}(y,\omega)\cot\left(d_{\mathbb S^q}(y,\omega)\right)\|v\|^2\\
    &\quad
    +2\left(1-d_{\mathbb S^q}(y,\omega)\cot\left(d_{\mathbb S^q}(y,\omega)\right)\right)
    \left\langle
        \frac{\log_\omega(y)}{d_{\mathbb S^q}(y,\omega)},v
    \right\rangle^2,
    \quad v\in T_\omega\mathbb S^q,
\end{split} \label{eq:HessianofFy}
\end{align}
where $T_\omega\mathbb S^q$ is a tangent space of $\mathbb S^q$ at $\omega$. At $\omega=y$, the Hessian of $F_y$ is
\begin{align*}
    \operatorname{Hess}F_y(y)[v,v]=2\|v\|^2, \quad v\in T_y\mathbb S^q.
\end{align*}
This agrees with the continuous extension of \eqref{eq:HessianofFy} since $\lim_{t\to0+}t\cot t=1$ and $\|\log_\omega(y)\|=d_{\mathbb S^q}(y,\omega)$ for $\omega \in \mathcal{Y}\setminus\{y\}$. Since a map $t \mapsto t\cot t$ is decreasing on $(0,\pi/2)$ and $d_{\mathbb S^q}(y,\omega)\le 2\rho$, it follows that
\begin{align}\label{eq:hessianbound}
    \operatorname{Hess}F_y(\omega)[v,v]
    \ge 4\rho\cot(2\rho)\|v\|^2, \quad v\in T_\omega\mathbb S^q.
\end{align}

We first verify Conditions~\ref{ass:M1}(i) and \ref{ass:M2}(i). Let $h\in(0,h_0]$ be given. Since $\widetilde M_h(x_0,\cdot):\mathcal{Y}\to\mathbb{R}$ is continuous and $\mathcal Y$ is compact, the minimizer $\widetilde m_h(x_0)$ exists. Fix $\omega\in\mathcal Y$, and let $\gamma_{\omega}:[0,1]\to\mathcal Y$ be a constant-speed minimizing geodesic from $\widetilde m_h(x_0)$ to $\omega$. By the definition of $\widetilde M_h$ and \eqref{eq:hessianbound}, differentiation under the expectation gives
\begin{align*}
    \operatorname{Hess}\widetilde M_h(x_0,\gamma_{\omega}(t))
    [\dot\gamma_{\omega}(t),\dot\gamma_{\omega}(t)]
    \ge4\rho\cot(2\rho)\|\dot\gamma_{\omega}(t)\|^2,
    \quad t\in(0,1).
\end{align*}
By the definition of $\gamma_{\omega}$, $\|\dot\gamma_{\omega}(t)\|=d_{\mathcal Y}(\omega,\widetilde m_h(x_0))$ for $t\in(0, 1)$. Hence, the above yields
\begin{align*}
    \frac{d^2}{dt^2}\widetilde M_h(x_0,\gamma_{\omega}(t))
    &\ge4\rho\cot(2\rho)
    d_{\mathcal Y}^2\left(\omega,\widetilde m_h(x_0)\right), \quad t\in(0,1).
\end{align*}
Since $\widetilde m_h(x_0)$ is a minimizer and $\gamma_{\omega}(0)=\widetilde m_h(x_0)$,
\begin{align*}
    \left.\frac{d}{dt}\widetilde M_h(x_0,\gamma_{\omega}(t))\right|_{t=0}\ge0, \quad\omega\in\mathcal Y.
\end{align*}
Also, $\gamma_{\omega}(1)=\omega$, and Taylor's formula with integral remainder gives
\begin{align*}
    \widetilde M_h(x_0,\omega)
    -\widetilde M_h\left(x_0,\widetilde m_h(x_0)\right)
    =
    \left.\frac{d}{dt}\widetilde M_h(x_0,\gamma_{\omega}(t))\right|_{t=0}+\int_0^1(1-t)
    \frac{d^2}{dt^2}\widetilde M_h(x_0,\gamma_{\omega}(t))\,dt.
\end{align*}
Therefore,
\begin{align*}
    \widetilde M_h(x_0,\omega)
    -\widetilde M_h\left(x_0,\widetilde m_h(x_0)\right)
    &\ge
    4\rho\cot(2\rho)
    d_{\mathcal Y}^2\left(\omega,\widetilde m_h(x_0)\right)
    \int_0^1(1-t)\,dt\\
    &=2\rho\cot(2\rho)
    d_{\mathcal Y}^2\left(\omega,\widetilde m_h(x_0)\right).
\end{align*}
Since $2\rho\cot(2\rho)>0$, this inequality also implies uniqueness. Therefore, Conditions~\ref{ass:M1}(i) and \ref{ass:M2}(i) hold with $\kappa_{\mathcal Y}=2\rho\cot(2\rho)$. The same argument proves Conditions~\ref{ass:M1}(ii)--(iv) and \ref{ass:M2}(ii)--(iv) with the same constant.

For Condition~\ref{ass:M4}, fix $\omega\in\mathcal Y$ and consider the logarithm map $\log_\omega : \mathcal Y \to T_\omega\mathbb S^q$ defined by
\begin{align*}
    \log_\omega(y)
    =
    \begin{cases}
    \displaystyle
    \frac{d_{\mathbb S^q}(y,\omega)}
    {\sin(d_{\mathbb S^q}(y,\omega))}
    \left(y-\langle y,\omega\rangle\omega\right),
    & \text{if} \quad y\in\mathcal{Y}\setminus\{\omega\},\\
    0,
    & \text{if} \quad y=\omega.
    \end{cases}
\end{align*}
For $y\in\mathcal Y$ such that $y\ne\omega$, let $D_y\log_\omega(y):T_y\mathbb S^q\to T_\omega\mathbb S^q$ denote the differential of $\log_\omega$ with respect to $y$. Differentiating with respect to $y$ and using $\langle y,v\rangle=0$ for $v\in T_y\mathbb S^q$ gives
\begin{align*}
    \|D_y\log_\omega(y)[v]\|^2
    &=\frac{d_{\mathbb S^q}^2(y,\omega)}
    {\sin^2(d_{\mathbb S^q}(y,\omega))}\|v\|^2-\frac{d_{\mathbb S^q}^2(y,\omega)
    -\sin^2(d_{\mathbb S^q}(y,\omega))}
    {\sin^4(d_{\mathbb S^q}(y,\omega))}
    \langle\omega,v\rangle^2\\
    &\le\frac{d_{\mathbb S^q}^2(y,\omega)}
    {\sin^2(d_{\mathbb S^q}(y,\omega))}\|v\|^2, \quad v\in T_y\mathbb S^q.
\end{align*}
At $y=\omega$, the differential is the identity on $T_\omega\mathbb S^q$. Since a map $t \mapsto t/\sin t$ is increasing on $(0,\pi/2)$ and $d_{\mathbb S^q}(y,\omega)\le2\rho$, we obtain
\begin{align*}
    \|D_y\log_\omega(y)\|_{\mathrm{op}}
    \le\frac{2\rho}{\sin(2\rho)}, \quad y\in\mathcal Y.
\end{align*}
Integrating this bound along the minimizing geodesic from $y$ to $y'$, \eqref{eq:GradientofFy} yields
\begin{align*}
    \|\nabla F_y(\omega)-\nabla F_{y'}(\omega)\|
    &=2\|\log_\omega(y)-\log_\omega(y')\|\\
    &\le\frac{4\rho}{\sin(2\rho)}d_{\mathcal Y}(y,y'), \quad y,y'\in\mathcal Y.
\end{align*}
Now let $\eta_{\omega,\omega'}:[0,1]\to\mathcal Y$ be a constant-speed minimizing geodesic from $\omega$ to $\omega'$ for given $\omega, \omega' \in \mathcal{Y}$. Integrating along $\eta_{\omega,\omega'}$ yields
\begin{align} \label{eq:quadraticbound}
\begin{split}
    &\left|F_y(\omega)-F_y(\omega')-F_{y'}(\omega)+F_{y'}(\omega')\right|\\
    &=
    \left|\int_0^1
    \left\langle
        \nabla F_y\left(\eta_{\omega,\omega'}(t)\right)
        -\nabla F_{y'}\left(\eta_{\omega,\omega'}(t)\right),
        \dot\eta_{\omega,\omega'}(t)
    \right\rangle\,dt\right|\\
    &\le
    \frac{4\rho}{\sin(2\rho)}
    d_{\mathcal Y}(y,y')d_{\mathcal Y}(\omega,\omega'), \quad y,y'\in\mathcal Y.
\end{split}
\end{align}
This proves Condition~\ref{ass:M4} with $C_{\mathcal Y}=4\rho/\sin(2\rho)$.

However, $\mathcal Y$ has finite diameter, so it cannot contain an isometric copy of $\R$, whose diameter is infinite. Hence the embedding condition used in the lower bounds does not apply to this example.
\end{example}

\begin{example}\label{ex:fisher-rao}\leavevmode
For $q\in\mathbb N$, let $\mathcal P_q$ denote the space of probability densities on $\mathbb R^q$ with respect to Lebesgue measure, where densities that agree almost everywhere are identified. Equip $\mathcal P_q$ with the normalized Fisher--Rao metric
\begin{align*}
    d_{\mathrm{FR}}(y,\omega)
    :=\arccos\left(\int_{\mathbb R^q}\sqrt{y(z)\omega(z)}\,dz\right),
    \quad y,\omega\in\mathcal P_q.
\end{align*}
Let $H=L^2(\mathbb R^q)$, with inner product $\langle\cdot,\cdot\rangle_H$ and norm $\|\cdot\|_H$, and let
\begin{align*}
    \mathbb S_H:=\{\psi\in H:\|\psi\|_H=1\}
\end{align*}
be its unit sphere, equipped with the geodesic metric
\begin{align*}
    d_{\mathbb S_H}(\psi,\phi):=\arccos\langle\psi,\phi\rangle_H,
    \quad \psi,\phi\in\mathbb S_H.
\end{align*}
Define the square-root map $\Phi:\mathcal P_q\to\mathbb S_H$ by
\begin{align*}
    \Phi(y)(z):=\sqrt{y(z)}, \quad z\in\mathbb R^q,\quad y\in\mathcal P_q.
\end{align*}
Then,
\begin{align*}
    d_{\mathrm{FR}}(y,\omega)
    =d_{\mathbb S_H}\left(\Phi(y),\Phi(\omega)\right),
    \quad y,\omega\in\mathcal P_q.
\end{align*}
Thus $\Phi$ is an isometry onto the nonnegative part of $\mathbb S_H$; see \cite{Dai2022} and \cite{BruverisMichor2019} for background on this representation.

Fix $p\in\mathcal P_q$ and $\rho\in(0,\pi/4)$, and let $\mathcal Y=\overline B_{\mathrm{FR}}(p,\rho)$ be a closed ball. Then the image of $\mathcal Y$ under $\Phi$ is
\begin{align*}
    \mathcal C
    :=\left\{
    \psi\in H:
    \psi\ge0\ \text{almost everywhere},\
    \|\psi\|_H=1,\
    \langle\psi,\Phi(p)\rangle_H\ge\cos\rho
    \right\}.
\end{align*}
To see that $\mathcal Y$ is not a Hadamard space, choose a measurable partition $A_1,A_2,A_3$ of $\mathbb R^q$ such that
\begin{align*}
    \int_{A_j}p(z)\,dz=\frac13,
    \quad j=1,2,3.
\end{align*}
Such a partition exists because the first marginal distribution corresponding to $p$ has a continuous cumulative distribution function. Define
\begin{align*}
    e_j(z):=\sqrt{3p(z)}\,\mathbf 1_{A_j}(z),
    \quad z\in\mathbb R^q,\quad j=1,2,3.
\end{align*}
Then $e_1,e_2,e_3$ are nonnegative orthonormal elements of $H$, and $\Phi(p)=(e_1+e_2+e_3)/\sqrt3$. For $a=(a_1,a_2,a_3)^\top\in\mathbb S^2$ sufficiently close to $(1,1,1)^\top/\sqrt3$, $a_1, a_2, a_3>0$ and $\sum_{j=1}^3a_je_j\in\mathcal C$. Also,
\begin{align*}
    d_{\mathbb S_H}\left(
    \sum_{j=1}^3a_je_j,\sum_{j=1}^3b_je_j
    \right)
    =\arccos\langle a,b\rangle,
    \quad a,b\in\mathbb S^2.
\end{align*}
Hence $\mathcal C$ contains an isometric copy of a sufficiently small geodesically convex neighborhood of $(1,1,1)^\top/\sqrt3$ in $\mathbb S^2$. Since $\mathbb S^2$ has sectional curvature $1$, $\mathcal C$ is not CAT(0). Therefore, neither $\mathcal C$ nor $\mathcal Y$ is a Hadamard space.

We prove that $(\mathcal Y,d_{\mathrm{FR}})$ is complete and separable. Note that the nonnegative cone is closed in $H$, and both the norm and the inner product with $\Phi(p)$ are continuous. Hence $\mathcal C$ is closed in $H$. Moreover,
\begin{align*}
    \frac{2}{\pi}d_{\mathbb S_H}(\psi,\phi)
    \le\|\psi-\phi\|_H
    =2\sin\left(\frac{d_{\mathbb S_H}(\psi,\phi)}{2}\right)
    \le d_{\mathbb S_H}(\psi,\phi),
    \quad \psi,\phi\in\mathcal C.
\end{align*}
Therefore, the Hilbert and spherical metrics are uniformly equivalent on $\mathcal C$. Since $H$ is complete and separable, $\mathcal C$ is complete and separable under $d_{\mathbb S_H}$. The isometry $\Phi$ then implies that $(\mathcal Y,d_{\mathrm{FR}})$ is complete and separable.

We next show that $\mathcal C$ is geodesically convex. Let $\psi,\phi\in\mathcal C$ be distinct, and put $\theta=d_{\mathbb S_H}(\psi,\phi)$. Since $\theta\le2\rho<\pi/2$, the unique constant-speed minimizing geodesic from $\psi$ to $\phi$ is
\begin{align*}
    \gamma_{\psi,\phi}(t)
    =
    \frac{\sin((1-t)\theta)}{\sin\theta}\psi
    +
    \frac{\sin(t\theta)}{\sin\theta}\phi,
    \quad t\in[0,1].
\end{align*}
Both coefficients are nonnegative, so $\gamma_{\psi,\phi}(t)\ge0$ almost everywhere for all $t\in[0,1]$. Also,
\begin{align*}
    \left\langle\gamma_{\psi,\phi}(t),\Phi(p)\right\rangle_H
    &\ge
    \cos\rho\,
    \frac{\sin((1-t)\theta)+\sin(t\theta)}{\sin\theta}\\
    &=
    \cos\rho\,
    \frac{\cos((1/2-t)\theta)}{\cos(\theta/2)}\\
    &\ge\cos\rho,
    \quad t\in[0,1],
\end{align*}
where the last inequality follows from $|1/2-t|\le1/2$. Hence $\gamma_{\psi,\phi}(t)\in\mathcal C$ for all $t\in[0,1]$. Thus $\mathcal C$ is geodesically convex. Since $\Phi$ is an isometry, $\mathcal Y$ is also geodesically convex and has diameter at most $2\rho$. Hence Condition~\ref{ass:M3} holds with $\sigma=2\rho$.

We now verify Conditions~\ref{ass:M1} and \ref{ass:M2}. Fix $\psi\in\mathcal C$ and consider the squared-distance function
\begin{align*}
    F_\psi(\phi):=d_{\mathbb S_H}^2(\psi,\phi),
    \quad \phi\in\mathbb S_H.
\end{align*}
The same Hessian calculation applies to $\mathbb S_H$, since any two points and a tangent vector lie in a finite-dimensional subspace of $H$. By \eqref{eq:hessianbound}, we obtain
\begin{align}\label{eq:fr-hessian-lower}
    \operatorname{Hess}F_\psi(\phi)[v,v]
    \ge4\rho\cot(2\rho)\|v\|_H^2,
    \quad \psi,\phi\in\mathcal C,\quad v\in T_\phi\mathbb S_H,
\end{align}
when $\phi\ne\psi$. At $\phi=\psi$, we have $\operatorname{Hess}F_\psi(\psi)[v,v]=2\|v\|_H^2$. For $\omega,\omega'\in\mathcal Y$, let $\gamma_{\omega,\omega'}:[0,1]\to\mathcal Y$ be the constant-speed minimizing geodesic obtained by applying $\Phi^{-1}$ to the spherical geodesic from $\Phi(\omega)$ to $\Phi(\omega')$. By \eqref{eq:fr-hessian-lower}, we obtain
\begin{align*}
    \frac{d^2}{dt^2}
    d_{\mathrm{FR}}^2\left(y,\gamma_{\omega,\omega'}(t)\right)
    \ge
    4\rho\cot(2\rho)d_{\mathrm{FR}}^2(\omega,\omega'),
    \quad y,\omega,\omega'\in\mathcal Y,\quad t\in(0,1).
\end{align*}
Integrating this lower bound along the geodesic gives
\begin{align}\label{eq:fr-pointwise-convexity}
\begin{split}
    d_{\mathrm{FR}}^2\left(y,\gamma_{\omega,\omega'}(t)\right)
    &\le
    (1-t)d_{\mathrm{FR}}^2(y,\omega)
    +t\,d_{\mathrm{FR}}^2(y,\omega')\\
    &-2\rho\cot(2\rho)t(1-t)d_{\mathrm{FR}}^2(\omega,\omega'),
    \quad y,\omega,\omega'\in\mathcal Y,
    \quad t\in[0,1].
\end{split}
\end{align}
Let $P$ be any probability measure supported on $\mathcal Y$, and define
\begin{align*}
    F_P(\omega):=\int_{\mathcal Y}d_{\mathrm{FR}}^2(y,\omega)\,dP(y),
    \quad \omega\in\mathcal Y.
\end{align*}
Since $\mathcal Y$ has diameter at most $2\rho$, $F_P$ is finite. Furthermore, the triangle inequality gives
\begin{align*}
    |F_P(\omega)-F_P(\omega')|
    &\le
    \int_{\mathcal Y}
    \left|d_{\mathrm{FR}}^2(y,\omega)-d_{\mathrm{FR}}^2(y,\omega')\right|
    \,dP(y)\\
    &\le
    \int_{\mathcal Y}
    \left(d_{\mathrm{FR}}(y,\omega)+d_{\mathrm{FR}}(y,\omega')\right)
    d_{\mathrm{FR}}(\omega,\omega')\,dP(y)\\
    &\le4\rho\,d_{\mathrm{FR}}(\omega,\omega'),
    \quad \omega,\omega'\in\mathcal Y.
\end{align*}
Hence $F_P$ is continuous. Integrating \eqref{eq:fr-pointwise-convexity} with respect to $P$ yields
\begin{align}\label{eq:fr-integrated-convexity}
\begin{split}
    F_P\left(\gamma_{\omega,\omega'}(t)\right)
    &\le
    (1-t)F_P(\omega)+tF_P(\omega')\\
    &-2\rho\cot(2\rho)t(1-t)d_{\mathrm{FR}}^2(\omega,\omega'),
    \quad \omega,\omega'\in\mathcal Y,\quad t\in[0,1].
\end{split}
\end{align}
Unlike the finite-dimensional spherical cap in Example~\ref{ex:spherical-cap}, compactness is not available here. We instead use completeness and \eqref{eq:fr-integrated-convexity} to establish existence. Put $a_P:=\inf_{\omega\in\mathcal Y}F_P(\omega)$, and choose a sequence $(\omega_j)_{j\ge1}$ in $\mathcal Y$ such that $F_P(\omega_j)\to a_P$. By \eqref{eq:fr-integrated-convexity}, we obtain
\begin{align*}
    a_P
    \le
    F_P\left(\gamma_{\omega_j,\omega_\ell}\left(\frac{1}{2}\right)\right)
    \le
    \frac{F_P(\omega_j)+F_P(\omega_\ell)}{2}
    -\frac{\rho\cot(2\rho)}{2}d_{\mathrm{FR}}^2(\omega_j,\omega_\ell),
    \quad j,\ell\in\mathbb N.
\end{align*}
Therefore,
\begin{align*}
    d_{\mathrm{FR}}^2(\omega_j,\omega_\ell)
    \le
    \frac{1}{\rho\cot(2\rho)}
    \left\{F_P(\omega_j)+F_P(\omega_\ell)-2a_P\right\}
    \longrightarrow0
    \quad\text{as }j,\ell\to\infty.
\end{align*}
Thus $(\omega_j)_{j\ge1}$ is a Cauchy sequence. Since $\mathcal Y$ is complete, there exists $\omega_P\in\mathcal Y$ such that $\omega_j\to\omega_P$. By continuity, $F_P(\omega_P)=a_P$, so $\omega_P$ is a minimizer. Then \eqref{eq:fr-integrated-convexity} gives
\begin{align*}
    F_P(\omega_P)
    &\le
    F_P\left(\gamma_{\omega_P,\omega}(t)\right)\\
    &\le
    (1-t)F_P(\omega_P)+tF_P(\omega)
    -2\rho\cot(2\rho)t(1-t)d_{\mathrm{FR}}^2(\omega_P,\omega),
    \quad \omega\in\mathcal Y,
    \quad t\in(0,1).
\end{align*}
This implies
\begin{align*}
    F_P(\omega)-F_P(\omega_P)\ge \lim_{t\to0+}2\rho\cot(2\rho)(1-t)d_{\mathrm{FR}}^2(\omega,\omega_P)
    =2\rho\cot(2\rho)d_{\mathrm{FR}}^2(\omega,\omega_P),
    \quad \omega\in\mathcal Y.
\end{align*}
Since $2\rho\cot(2\rho)>0$, this inequality also implies uniqueness. Each criterion in Condition~\ref{ass:M1}(i)--(iv) can be written as $F_P$ for an appropriate probability measure $P$ supported on $\mathcal Y$. Therefore, Conditions~\ref{ass:M1} and \ref{ass:M2} hold with the common constant $\kappa_{\mathcal Y}=2\rho\cot(2\rho)$.

Finally, we verify Condition~\ref{ass:M4}. Fix $y,y',\omega,\omega'\in\mathcal Y$. Let $V$ be a finite-dimensional linear subspace of $H$ of dimension at least three containing $\Phi(p)$, $\Phi(y)$, $\Phi(y')$, $\Phi(\omega)$, and $\Phi(\omega')$. Such a subspace exists because $H$ is infinite-dimensional. Let
\begin{align*}
    \mathbb S_V:=\{\psi\in V:\|\psi\|_H=1\}.
\end{align*}
The restriction of $d_{\mathbb S_H}$ to $\mathbb S_V$ is the usual geodesic metric on this finite-dimensional unit sphere. Moreover, the four square-root densities belong to the spherical cap $\overline B_{\mathbb S_V}(\Phi(p),\rho)$. By \eqref{eq:quadraticbound}, we have
\begin{align*}
\begin{split}
    &\left|
    d_{\mathrm{FR}}^2(y,\omega)
    -d_{\mathrm{FR}}^2(y,\omega')
    -d_{\mathrm{FR}}^2(y',\omega)
    +d_{\mathrm{FR}}^2(y',\omega')
    \right|\\
    &\le
    \frac{4\rho}{\sin(2\rho)}
    d_{\mathbb S_H}\left(\Phi(y),\Phi(y')\right)
    d_{\mathbb S_H}\left(\Phi(\omega),\Phi(\omega')\right)\\
    &=
    \frac{4\rho}{\sin(2\rho)}
    d_{\mathrm{FR}}(y,y')d_{\mathrm{FR}}(\omega,\omega').
\end{split}
\end{align*}
This proves Condition~\ref{ass:M4} with $C_{\mathcal Y}=4\rho/\sin(2\rho)$.

However, since $\mathcal Y$ is bounded, it does not admit an isometric embedding of $\mathbb R$, and hence the embedded-Gaussian lower bounds do not apply to this example.
\end{example}

\begin{example}\label{ex:bures-wasserstein-spd}\leavevmode
For $q\in\mathbb N$ with $q\ge2$, let $\mathcal S_{++}^q$ denote the space of $q\times q$ symmetric positive-definite matrices. For fixed constants $m,M\in(0,\infty)$ with $m<M$, let
\begin{align*}
\mathcal Y:=\{A\in\mathcal S_{++}^q:mI_q\preceq A\preceq MI_q\},
\end{align*}
equipped with the Bures--Wasserstein metric
\begin{align*}
d_{\mathcal Y}^2(A,B):=\operatorname{tr}(A)+\operatorname{tr}(B)-2\operatorname{tr}(A^{1/2}BA^{1/2})^{1/2},\quad A,B\in\mathcal Y.
\end{align*}
In particular, $d_{\mathcal Y}(A,B)$ coincides with the $2$-Wasserstein distance between the centered Gaussian distributions $N(0,A)$ and $N(0,B)$ for $A,B\in\mathcal Y$; see Sections~2--4 of \cite{BhatiaJainLim2019}. The interior of $\mathcal Y$ has nonnegative sectional curvature, with strictly positive sectional curvature for some tangent planes; see Section~3.4 and Table~7 of \cite{ThanwerdasPennec2023}. Hence $(\mathcal Y,d_{\mathcal Y})$ is not a Hadamard space.

For $A,B\in\mathcal Y$, the factor representation of the Bures--Wasserstein metric gives
\begin{align}\label{eq:factor_representation}
d_{\mathcal Y}(A,B)=\min_{U\in O(q)}\|A^{1/2}-B^{1/2}U\|_{\mathrm F},\quad A,B\in\mathcal Y,
\end{align}
where $O(q)$ is the space of $q\times q$ orthogonal matrices. For more details, see Section~2 of \cite{BhatiaJainLim2019}. Taking $U=I_q$ gives
\begin{align*}
    d_{\mathcal Y}(A,B)\le\|A^{1/2}-B^{1/2}\|_{\mathrm F},\quad A,B\in\mathcal Y.
\end{align*}
By the Cauchy--Schwarz inequality and the symmetry,
\begin{align*}
\begin{split}
    &\|A^{1/2}-B^{1/2}\|_{\mathrm F}\|A-B\|_{\mathrm F}\\
    &\ge\left\langle A^{1/2}-B^{1/2},A-B\right\rangle_{\mathrm F}\\
    &=\left\langle A^{1/2}-B^{1/2},A^{1/2}(A^{1/2}-B^{1/2})+(A^{1/2}-B^{1/2})B^{1/2}\right\rangle_{\mathrm F}\\
    &=\left\langle A^{1/2}-B^{1/2},A^{1/2}(A^{1/2}-B^{1/2})\right\rangle_{\mathrm F}+\left\langle A^{1/2}-B^{1/2},(A^{1/2}-B^{1/2})B^{1/2}\right\rangle_{\mathrm F}\\
    &=\operatorname{tr}\left[(A^{1/2}-B^{1/2})A^{1/2}(A^{1/2}-B^{1/2})\right]+\operatorname{tr}\left[(A^{1/2}-B^{1/2})B^{1/2}(A^{1/2}-B^{1/2})\right]\\
    &\ge\operatorname{tr}\left[(A^{1/2}-B^{1/2})(\sqrt m I_q)(A^{1/2}-B^{1/2})\right]+\operatorname{tr}\left[(A^{1/2}-B^{1/2})(\sqrt m I_q)(A^{1/2}-B^{1/2})\right]\\
    &=2\sqrt m\,\operatorname{tr}\left[(A^{1/2}-B^{1/2})^2\right]\\
    &=2\sqrt m\,\|A^{1/2}-B^{1/2}\|_{\mathrm F}^2,\quad A,B\in\mathcal Y.
\end{split}
\end{align*}
where the last inequality follows from $A^{1/2},B^{1/2}\succeq\sqrt m I_q$. Therefore,
\begin{align}\label{eq:bw-upper-bound}
d_{\mathcal Y}(A,B)\le\frac{1}{2\sqrt m}\|A-B\|_{\mathrm F}\quad A,B\in\mathcal Y.
\end{align}
For the reverse inequality, let $A,B\in\mathcal Y$ be fixed. Let $U_{A,B}\in O(q)$ attain the minimum in \eqref{eq:factor_representation}. Then
\begin{align}\label{eq:factor_representation_equality}
d_{\mathcal Y}(A,B)=\|A^{1/2}-B^{1/2}U_{A,B}\|_{\mathrm F}.
\end{align}
Since $U_{A,B}\in O(q)$,
\begin{align}\label{eq:orthogonal_property}
B=(B^{1/2}U_{A,B})(B^{1/2}U_{A,B})^\top.
\end{align}
By \eqref{eq:factor_representation_equality}, \eqref{eq:orthogonal_property}, the triangle inequality, and the definition of the operator norm,
\begin{align*}
    \|A-B\|_{\mathrm F}
    &=\left\|A^{1/2}\left(A^{1/2}-B^{1/2}U_{A,B}\right)^\top+\left(A^{1/2}-B^{1/2}U_{A,B}\right)\left(B^{1/2}U_{A,B}\right)^\top\right\|_{\mathrm F}\\
    &\le\left\|A^{1/2}\left(A^{1/2}-B^{1/2}U_{A,B}\right)^\top\right\|_{\mathrm F}+\left\|\left(A^{1/2}-B^{1/2}U_{A,B}\right)\left(B^{1/2}U_{A,B}\right)^\top\right\|_{\mathrm F}\\
    &\le\left(\|A^{1/2}\|_{\mathrm{op}}+\|B^{1/2}U_{A,B}\|_{\mathrm{op}}\right)\|A^{1/2}-B^{1/2}U_{A,B}\|_{\mathrm F}\\
    &=\left(\|A^{1/2}\|_{\mathrm{op}}+\|B^{1/2}\|_{\mathrm{op}}\right)\|A^{1/2}-B^{1/2}U_{A,B}\|_{\mathrm F}\\
    &\le2\sqrt M\,d_{\mathcal Y}(A,B).
\end{align*}
where the last inequality follows from $A,B\preceq MI_q$. Since this holds for arbitrary $A,B\in\mathcal Y$,
\begin{align}\label{eq:bw-lower-bound}
d_{\mathcal Y}(A,B)\ge\frac{1}{2\sqrt M}\|A-B\|_{\mathrm F},\quad A,B\in\mathcal Y.
\end{align}
Combining \eqref{eq:bw-upper-bound} and \eqref{eq:bw-lower-bound}, we obtain
\begin{align}\label{eq:bw-frobenius-equivalence}
\frac{1}{2\sqrt M}\|A-B\|_{\mathrm F}\le d_{\mathcal Y}(A,B)\le\frac{1}{2\sqrt m}\|A-B\|_{\mathrm F},\quad A,B\in\mathcal Y.
\end{align}
Hence the Bures--Wasserstein and Frobenius metrics are uniformly equivalent on $\mathcal Y$. Since $\mathcal Y$ is compact in the Frobenius metric, it is also compact, complete, and separable under the Bures--Wasserstein metric.

We next verify Conditions~\ref{ass:M1} and \ref{ass:M2}. Let $A\in\mathcal Y$ be fixed, and define
\begin{align*}
    f_A(W):=d_{\mathcal Y}^2(A,W),\quad W\in\mathcal Y.
\end{align*}
Let $\mathrm{Sym}(q)$ be the space of $q\times q$ symmetric matrices. For $W\in\mathcal Y$ and $H\in\mathrm{Sym}(q)$, consider the second directional derivative of $f_A$ with respect to $W$ in the direction $H$. Write
\begin{align*}
\left(A^{1/2}WA^{1/2}\right)^{1/2}
=U_{A,W}\Lambda_{A,W}U_{A,W}^\top,
\quad
\Lambda_{A,W}
=\operatorname{diag}\left(\lambda_{A,W}^{(1)},\ldots,\lambda_{A,W}^{(q)}\right),
\quad W\in\mathcal Y,
\end{align*}
for the spectral decomposition of $\left(A^{1/2}WA^{1/2}\right)^{1/2}$, where $U_{A,W}\in O(q)$. Using the differential of the matrix square-root function, we obtain
\begin{align}\label{eq:bw-hessian}
D_W^2f_A(W)[H,H]
=\sum_{i,j=1}^q
\frac{\left[U_{A,W}^\top A^{1/2}HA^{1/2}U_{A,W}\right]_{ij}^2}
{\lambda_{A,W}^{(i)}\lambda_{A,W}^{(j)}\left(\lambda_{A,W}^{(i)}+\lambda_{A,W}^{(j)}\right)},\quad W\in\mathcal Y,\quad H\in\mathrm{Sym}(q).
\end{align}
For more details, see Section~6 of \cite{BhatiaJainLim2019}. Note that
\begin{align*}
m^2I_q\preceq mA\preceq A^{1/2}WA^{1/2}\preceq MA\preceq M^2I_q,\quad W\in\mathcal Y.
\end{align*}
By the operator monotonicity of the matrix square-root function,
\begin{align}\label{eq:matrix_square_root_bound}
mI_q\preceq\left(A^{1/2}WA^{1/2}\right)^{1/2}=U_{A,W}\Lambda_{A,W}U_{A,W}^\top\preceq MI_q,\quad W\in\mathcal Y.
\end{align}
For more details, see Section~1.5 of \cite{Bhatia2007}. Therefore,
\begin{align}\label{eq:spectral_eigen_value_bound}
    m\le\lambda_{A,W}^{(i)}\le M,\quad i=1,\cdots,q,\quad W\in\mathcal Y.
\end{align}
Moreover, $A\in\mathcal Y$ and the definition of the operator norm give
\begin{align}\label{eq:frobenius_operator_norm_bound}
\begin{split}
\|H\|_{\mathrm F}&=\left\|A^{-1/2}(A^{1/2}HA^{1/2})A^{-1/2}\right\|_{\mathrm F}\\
&\le\|A^{-1/2}\|_{\mathrm{op}}^2\|A^{1/2}HA^{1/2}\|_{\mathrm F}\\
&\le\frac1m\|A^{1/2}HA^{1/2}\|_{\mathrm F},\quad H\in\mathrm{Sym}(q).
\end{split}
\end{align}
Combining \eqref{eq:bw-hessian}, \eqref{eq:spectral_eigen_value_bound}, and \eqref{eq:frobenius_operator_norm_bound}, and using the orthogonal invariance of the Frobenius norm, we obtain
\begin{align}\label{eq:bw-hessian-lower}
\begin{split}
D_W^2f_A(W)[H,H]
&\ge\frac{1}{2M^3}\sum_{i,j=1}^q
\left[U_{A,W}^\top A^{1/2}HA^{1/2}U_{A,W}\right]_{ij}^2\\
&=\frac{1}{2M^3}\left\|U_{A,W}^\top A^{1/2}HA^{1/2}U_{A,W}\right\|_{\mathrm F}^2\\
&=\frac{1}{2M^3}\|A^{1/2}HA^{1/2}\|_{\mathrm F}^2\\
&\ge\frac{m^2}{2M^3}\|H\|_{\mathrm F}^2,\quad A,W\in\mathcal Y,\quad H\in\mathrm{Sym}(q).
\end{split}
\end{align}
Now let $P$ be any probability measure supported on $\mathcal Y$ and define
\begin{align*}
F_P(W):=\int_{\mathcal Y}f_A(W)\,dP(A),\quad W\in\mathcal Y.
\end{align*}
Since $\mathcal Y$ is compact and $F_P$ is continuous, there exists a minimizer $W_P\in\mathcal Y$. Also, since $\mathcal Y$ is convex in the Euclidean space of symmetric matrices, $W_P+t(W-W_P)\in\mathcal Y$ for every $W\in\mathcal Y$ and $t\in[0,1]$. For any given $W\in\mathcal Y$, define
\begin{align*}
g_W(t):=F_P\left(W_P+t(W-W_P)\right),\quad t\in[0,1].
\end{align*}
Since $g_W$ is minimized at $t=0$, $g_W'(0+)\ge0$ implies
\begin{align}\label{eq:bw-first-order-optimality}
\left\langle\nabla F_P(W_P),W-W_P\right\rangle_{\mathrm F}\ge0,\quad W\in\mathcal Y.
\end{align}
Integrating \eqref{eq:bw-hessian-lower} with respect to $P$ yields
\begin{align}\label{eq:bw-population-hessian-lower}
D_W^2F_P(W)[H,H]\ge\frac{m^2}{2M^3}\|H\|_{\mathrm F}^2,\quad W\in\mathcal Y,\quad H\in\mathrm{Sym}(q).
\end{align}
By \eqref{eq:bw-frobenius-equivalence}, \eqref{eq:bw-first-order-optimality}, and \eqref{eq:bw-population-hessian-lower}, Taylor's formula with integral remainder gives
\begin{align*}
\begin{split}
F_P(W)-F_P(W_P)
&=g_W(1)-g_W(0)\\
&=g_W'(0+)+\int_0^1(1-t)g_W''(t)\,dt\\
&=\left\langle\nabla F_P(W_P),W-W_P\right\rangle_{\mathrm F}\\
&\quad+\int_0^1(1-t)D_W^2F_P\left(W_P+t(W-W_P)\right)[W-W_P,W-W_P]\,dt\\
&\ge\frac{m^2}{2M^3}\|W-W_P\|_{\mathrm F}^2\int_0^1(1-t)\,dt\\
&=\frac{m^2}{4M^3}\|W-W_P\|_{\mathrm F}^2\\
&\ge\left(\frac{m}{M}\right)^3d_{\mathcal Y}^2(W,W_P),\quad W\in\mathcal Y.
\end{split}
\end{align*}
Each criterion in Condition~\ref{ass:M1}(i)--(iv) can be written as $F_P$ for an appropriate probability measure $P$ supported on $\mathcal Y$. Since $m/M>0$, the above inequality also implies that $W_P$ is unique. Therefore, Conditions~\ref{ass:M1} and \ref{ass:M2} hold with the common constant $\kappa_{\mathcal Y}=\left(m/M\right)^3$.

Now we verify Condition~\ref{ass:M3}. Since $\sqrt m I_q\preceq A^{1/2},B^{1/2}\preceq\sqrt M I_q$,
\begin{align}\label{eq:sqrt_root_subtraction_bound}
(\sqrt m-\sqrt M) I_q \preceq A^{1/2}-B^{1/2}\preceq (\sqrt M-\sqrt m) I_q,\quad A,B\in\mathcal Y.
\end{align}
By \eqref{eq:factor_representation} and \eqref{eq:sqrt_root_subtraction_bound}, taking $U=I_q$ in \eqref{eq:factor_representation} gives
\begin{align*}
d_{\mathcal Y}(A,B)\le\|A^{1/2}-B^{1/2}\|_{\mathrm F}\le\sqrt{q}\|A^{1/2}-B^{1/2}\|_{\mathrm{OP}}\le\sqrt q\,(\sqrt M-\sqrt m),\quad A,B\in\mathcal Y.
\end{align*}
Since $Y,m_\oplus(x_0)\in\mathcal Y$,
\begin{align*}
M_\oplus\left(x,m_\oplus(x_0)\right)=\E\left[d_{\mathcal Y}^2\left(Y,m_\oplus(x_0)\right)\mid X=x\right]\le q(\sqrt M-\sqrt m)^2.
\end{align*}
Therefore, Condition~\ref{ass:M3} holds with $\sigma=\sqrt q\,(\sqrt M-\sqrt m)$.

Finally, we verify Condition~\ref{ass:M4}. The gradient of $f_A(W)$ with respect to $W$ is
\begin{align}\label{eq:bw-gradient}
\nabla_Wf_A(W)=I_q-W^{-1/2}(W^{1/2}AW^{1/2})^{1/2}W^{-1/2}.
\end{align}
For more details, see Section~6 of \cite{BhatiaJainLim2019}. For $A,W\in\mathcal Y$ and $H\in\mathrm{Sym}(q)$, the standard differential formula for the matrix square-root function gives that the derivative of $(W^{1/2}AW^{1/2})^{1/2}$ with respect to $A$ in the direction $H$ is the unique solution $L_{A,W}(H)$ of
\begin{align}\label{eq:bw-sylvester-cross}
(W^{1/2}AW^{1/2})^{1/2}L_{A,W}(H)+L_{A,W}(H)(W^{1/2}AW^{1/2})^{1/2}=W^{1/2}HW^{1/2}.
\end{align}
It follows from \eqref{eq:bw-gradient} and \eqref{eq:bw-sylvester-cross} that
\begin{align}\label{eq:bw-gradient-cross-derivative}
D_A\nabla_Wf_A(W)[H]=-W^{-1/2}L_{A,W}(H)W^{-1/2}.
\end{align}
By \eqref{eq:matrix_square_root_bound}, with $A$ and $W$ interchanged,
\begin{align}\label{eq:bw-sylvester-bound-1}
(W^{1/2}AW^{1/2})^{1/2}\succeq mI_q,\quad A,W\in\mathcal Y.
\end{align}
Also, since $W\preceq MI_q$, the definition of the operator norm gives
\begin{align}\label{eq:bw-sylvester-bound-2}
\|W^{1/2}HW^{1/2}\|_{\mathrm F}
\le\|W^{1/2}\|_{\mathrm{op}}^2\|H\|_{\mathrm F}
\le M\|H\|_{\mathrm F},\quad W\in\mathcal Y,\quad H\in\mathrm{Sym}(q).
\end{align}
By \eqref{eq:bw-sylvester-cross}, \eqref{eq:bw-sylvester-bound-1}, and the Cauchy--Schwarz inequality,
\begin{align}\label{eq:bw-sylvester-bound-4}
\begin{split}
\|L_{A,W}(H)\|_{\mathrm F}\|W^{1/2}HW^{1/2}\|_{\mathrm F}
&\ge \left\langle L_{A,W}(H),W^{1/2}HW^{1/2}\right\rangle_{\mathrm F}\\
&=\left\langle L_{A,W}(H),(W^{1/2}AW^{1/2})^{1/2}L_{A,W}(H)\right\rangle_{\mathrm F}\\
&\quad+\left\langle L_{A,W}(H),L_{A,W}(H)(W^{1/2}AW^{1/2})^{1/2}\right\rangle_{\mathrm F}\\
&\ge m\|L_{A,W}(H)\|_{\mathrm F}^2+m\|L_{A,W}(H)\|_{\mathrm F}^2\\
&=2m\|L_{A,W}(H)\|_{\mathrm F}^2,\quad A,W\in\mathcal Y,\quad H\in\mathrm{Sym}(q).
\end{split}
\end{align}
Therefore, by \eqref{eq:bw-sylvester-bound-2} and \eqref{eq:bw-sylvester-bound-4},
\begin{align}\label{eq:bw-sylvester-bound-3}
\|L_{A,W}(H)\|_{\mathrm F}
\le\frac{1}{2m}\|W^{1/2}HW^{1/2}\|_{\mathrm F}
\le\frac{M}{2m}\|H\|_{\mathrm F},
\quad A,W\in\mathcal Y,\quad H\in\mathrm{Sym}(q).
\end{align}
Combining \eqref{eq:bw-gradient-cross-derivative} and \eqref{eq:bw-sylvester-bound-3}, the definition of the operator norm gives
\begin{align}\label{eq:bw-cross-hessian}
\begin{split}
\|D_A\nabla_Wf_A(W)[H]\|_{\mathrm F}
&\le\|W^{-1/2}\|_{\mathrm{op}}^2\|L_{A,W}(H)\|_{\mathrm F}\\
&\le\frac{M}{2m^2}\|H\|_{\mathrm F},
\quad A,W\in\mathcal Y,\quad H\in\mathrm{Sym}(q).
\end{split}
\end{align}
Let $A,A',W,W'\in\mathcal Y$ be fixed. Define
\begin{align*}
A_s:=A'+s(A-A'),\qquad W_t:=W'+t(W-W'),\quad s,t\in[0,1].
\end{align*}
Since $\mathcal Y$ is convex, $A_s,W_t\in\mathcal Y$ for all $s,t\in[0,1]$. By the chain rule,
\begin{align*}
\frac{\partial}{\partial t}f_{A_s}(W_t)&=\left\langle\left.\nabla_W f_{A_s}(W)\right|_{W=W_t},W-W'\right\rangle_{\mathrm F},\\
\frac{\partial^2}{\partial s\,\partial t}f_{A_s}(W_t)
&=
\left\langle
\left.
D_A\nabla_Wf_A(W)[A-A']
\right|_{(A,W)=(A_s,W_t)},
W-W'
\right\rangle_{\mathrm F},\quad s,t\in[0,1].
\end{align*}
Therefore, applying the fundamental theorem of calculus twice,
\begin{align} \label{eq:fundamental_calculus}
\begin{split}
&d_{\mathcal Y}^2(A,W)-d_{\mathcal Y}^2(A,W')-d_{\mathcal Y}^2(A',W)+d_{\mathcal Y}^2(A',W')\\
&=\int_0^1\left(\left\langle\left.\nabla_W f_{A}(W)\right|_{W=W_t},W-W'\right\rangle_{\mathrm F}-\left\langle\left.\nabla_W f_{A'}(W)\right|_{W=W_t},W-W'\right\rangle_{\mathrm F}\right)\,dt\\
&=\int_0^1\int_0^1\left\langle\left.D_A\nabla_Wf_A(W)[A-A']\right|_{(A,W)=(A_s,W_t)},W-W'\right\rangle_{\mathrm F}\,ds\,dt.
\end{split}
\end{align}
Hence, by \eqref{eq:bw-frobenius-equivalence}, \eqref{eq:bw-cross-hessian}, \eqref{eq:fundamental_calculus}, and the Cauchy--Schwarz inequality,
\begin{align*}
\begin{split}
&\left|
d_{\mathcal Y}^2(A,W)-d_{\mathcal Y}^2(A,W')
-d_{\mathcal Y}^2(A',W)+d_{\mathcal Y}^2(A',W')
\right|\\
&\le
\int_0^1\int_0^1
\left\|
\left.
D_A\nabla_Wf_A(W)[A-A']
\right|_{(A,W)=(A_s,W_t)}
\right\|_{\mathrm F}
\|W-W'\|_{\mathrm F}
\,ds\,dt\\
&\le
\frac{M}{2m^2}
\|A-A'\|_{\mathrm F}
\|W-W'\|_{\mathrm F}\\
&\le
2\left(\frac{M}{m}\right)^2
d_{\mathcal Y}(A,A')
d_{\mathcal Y}(W,W'),
\end{split}
\end{align*}
Therefore, Condition~\ref{ass:M4} holds with $C_{\mathcal Y}=2\left(M/m\right)^2$.

However, since $\mathcal Y$ is bounded, it does not admit an isometric embedding of $\mathbb R$, and hence the embedded-Gaussian lower bounds do not apply to this example.
\end{example}

\section{Proofs for Section \ref{subsec:general-framework}}\label{sec:supp-weighted}\label{sec:supp-r2}

In this section, we prove Theorem~\ref{thm:r2-weighted-barycenter} and provide the corresponding conditional bounds for the block estimators. For the proof of the theorem, fix $N\in\mathbb N$, let $\bm\lambda\in\mathcal S_N$ be deterministic, and let $\widetilde{\bm Y}=(\widetilde Y_1,\ldots,\widetilde Y_N)$ be as in Theorem~\ref{thm:r2-weighted-barycenter}. Let $\widetilde{\bm Y}'=(\widetilde Y_1',\ldots,\widetilde Y_N')$ be an independent copy of $\widetilde{\bm Y}$. For $i=1,\ldots,N$, define
\begin{align*}
    \widetilde{\bm Y}^{(i)}
    :=\left(\widetilde Y_1,\ldots,\widetilde Y_{i-1},\widetilde Y_i',\widetilde Y_{i+1},\ldots,\widetilde Y_N\right).
\end{align*}

\begin{lemma}\label{lem:r2-integrability}
Under the conditions of Theorem~\ref{thm:r2-weighted-barycenter},
\begin{align}
    \kappa_{\mathcal Y}d_{\mathcal Y}^2\left(\overline b_{\oplus,\bm\lambda}(\widetilde{\bm Y}),\overline q_{\oplus,\bm\lambda}(\widetilde{\bm Y})\right)
    &\le\overline B_{\oplus,\bm\lambda}\left(\widetilde{\bm Y},\overline q_{\oplus,\bm\lambda}(\widetilde{\bm Y})\right), \label{eq:r2-integrability-part1}\\
    \kappa_{\mathcal Y}d_{\mathcal Y}^2\left(\overline b_{\oplus,\bm\lambda}(\widetilde{\bm Y}^{(i)}),\overline q_{\oplus,\bm\lambda}(\widetilde{\bm Y})\right)
    &\le\overline B_{\oplus,\bm\lambda}\left(\widetilde{\bm Y}^{(i)},\overline q_{\oplus,\bm\lambda}(\widetilde{\bm Y})\right),
    \quad i=1,\ldots,N. \label{eq:r2-integrability-part2}
\end{align}
In particular, for every $i=1,\ldots,N$,
\begin{align}
    \E\left[d_{\mathcal Y}^2\left(\overline b_{\oplus,\bm\lambda}(\widetilde{\bm Y}),\overline q_{\oplus,\bm\lambda}(\widetilde{\bm Y})\right)\right]
    +\E\left[d_{\mathcal Y}^2\left(\overline b_{\oplus,\bm\lambda}(\widetilde{\bm Y}^{(i)}),\overline q_{\oplus,\bm\lambda}(\widetilde{\bm Y})\right)\right]
    &<\infty, \label{eq:r2-integrability-part3}\\
    \E\left[d_{\mathcal Y}(\widetilde Y_i,\widetilde Y_i')d_{\mathcal Y}\left(\overline b_{\oplus,\bm\lambda}(\widetilde{\bm Y}),\overline b_{\oplus,\bm\lambda}(\widetilde{\bm Y}^{(i)})\right)\right]
    &<\infty. \label{eq:r2-integrability-part4}
\end{align}
\end{lemma}

\begin{proof}[Proof of \Cref{lem:r2-integrability}]
By Condition~\ref{ass:M2}(iii) and the nonnegativity of $\overline B_{\oplus,\bm\lambda}$,
\begin{align*}
    \kappa_{\mathcal Y}d_{\mathcal Y}^2\left(\overline b_{\oplus,\bm\lambda}(\widetilde{\bm Y}),\overline q_{\oplus,\bm\lambda}(\widetilde{\bm Y})\right)
    &\le\overline B_{\oplus,\bm\lambda}\left(\widetilde{\bm Y},\overline q_{\oplus,\bm\lambda}(\widetilde{\bm Y})\right)
    -\overline B_{\oplus,\bm\lambda}\left(\widetilde{\bm Y},\overline b_{\oplus,\bm\lambda}(\widetilde{\bm Y})\right)\\
    &\le\overline B_{\oplus,\bm\lambda}\left(\widetilde{\bm Y},\overline q_{\oplus,\bm\lambda}(\widetilde{\bm Y})\right).
\end{align*}
This proves \eqref{eq:r2-integrability-part1}. The same argument with $\widetilde{\bm Y}^{(1)}, \ldots, \widetilde{\bm Y}^{(N)}$ proves \eqref{eq:r2-integrability-part2}.

Note that $\overline q_{\oplus,\bm\lambda}(\widetilde{\bm Y})$ is deterministic since it depends only on the distributions of $\widetilde Y_1,\ldots,\widetilde Y_N$ and $\bm\lambda$. Since $\widetilde Y_1,\ldots,\widetilde Y_N\in\mathcal L_2(\mathcal Y)$,
\begin{align*}
    \E\left[\overline B_{\oplus,\bm\lambda}\left(\widetilde{\bm Y},\overline q_{\oplus,\bm\lambda}(\widetilde{\bm Y})\right)\right]
    =\sum_{i=1}^N\lambda_i\E\left[d_{\mathcal Y}^2\left(\widetilde Y_i,\overline q_{\oplus,\bm\lambda}(\widetilde{\bm Y})\right)\right]<\infty.
\end{align*}
The same conclusion holds with $\widetilde{\bm Y}^{(i)}$ in place of $\widetilde{\bm Y}$ for every $i=1,\ldots,N$. Therefore, the preceding bound, \eqref{eq:r2-integrability-part1}, and \eqref{eq:r2-integrability-part2} imply \eqref{eq:r2-integrability-part3}.

Moreover, by the triangle inequality and the Cauchy--Schwarz inequality, for every $i=1,\ldots,N$,
\begin{align*}
    d_{\mathcal Y}^2\left(\overline b_{\oplus,\bm\lambda}(\widetilde{\bm Y}),\overline b_{\oplus,\bm\lambda}(\widetilde{\bm Y}^{(i)})\right)
    &\le2d_{\mathcal Y}^2\left(\overline b_{\oplus,\bm\lambda}(\widetilde{\bm Y}),\overline q_{\oplus,\bm\lambda}(\widetilde{\bm Y})\right)
    +2d_{\mathcal Y}^2\left(\overline b_{\oplus,\bm\lambda}(\widetilde{\bm Y}^{(i)}),\overline q_{\oplus,\bm\lambda}(\widetilde{\bm Y})\right).
\end{align*}
Thus, the above and \eqref{eq:r2-integrability-part3} imply, for every $i=1,\ldots,N$,
\begin{align}\label{eq:r2-integrability-proof-part1}
    \E\left[d_{\mathcal Y}^2\left(\overline b_{\oplus,\bm\lambda}(\widetilde{\bm Y}),\overline b_{\oplus,\bm\lambda}(\widetilde{\bm Y}^{(i)})\right)\right]<\infty.
\end{align}
For every $i=1,\ldots,N$, $\widetilde Y_i'$ has the same distribution as $\widetilde Y_i$. Therefore, by the triangle inequality and the Cauchy--Schwarz inequality,
\begin{align}\label{eq:r2-integrability-proof-part2}
    \E\left[d_{\mathcal Y}^2(\widetilde Y_i,\widetilde Y_i')\right]
    &\le2\E\left[d_{\mathcal Y}^2\left(\widetilde Y_i,\overline q_{\oplus,\bm\lambda}(\widetilde{\bm Y})\right)\right]
    +2\E\left[d_{\mathcal Y}^2\left(\widetilde Y_i',\overline q_{\oplus,\bm\lambda}(\widetilde{\bm Y})\right)\right]<\infty.
\end{align}
Thus, \eqref{eq:r2-integrability-proof-part1}, \eqref{eq:r2-integrability-proof-part2}, and the Cauchy--Schwarz inequality prove \eqref{eq:r2-integrability-part4}.
\end{proof}

\begin{lemma}\label{lem:r2-replace}
Under the conditions of Theorem~\ref{thm:r2-weighted-barycenter},
\begin{align*}
    d_{\mathcal Y}\left(\overline b_{\oplus,\bm\lambda}(\widetilde{\bm Y}),\overline b_{\oplus,\bm\lambda}(\widetilde{\bm Y}^{(i)})\right)
    \le\frac{C_{\mathcal Y}\lambda_i}{2\kappa_{\mathcal Y}}d_{\mathcal Y}(\widetilde Y_i,\widetilde Y_i'), \quad i=1,\ldots,N.
\end{align*}
\end{lemma}

\begin{proof}[Proof of \Cref{lem:r2-replace}]
Let $i=1,\ldots,N$ be given. Applying Condition~\ref{ass:M2}(iii) to $\widetilde{\bm Y}$ and $\widetilde{\bm Y}^{(i)}$ gives
\begin{align*}
    \kappa_{\mathcal Y}d_{\mathcal Y}^2\left(\overline b_{\oplus,\bm\lambda}(\widetilde{\bm Y}),\overline b_{\oplus,\bm\lambda}(\widetilde{\bm Y}^{(i)})\right)
    &\le\overline B_{\oplus,\bm\lambda}\left(\widetilde{\bm Y},\overline b_{\oplus,\bm\lambda}(\widetilde{\bm Y}^{(i)})\right)
    -\overline B_{\oplus,\bm\lambda}\left(\widetilde{\bm Y},\overline b_{\oplus,\bm\lambda}(\widetilde{\bm Y})\right),\\
    \kappa_{\mathcal Y}d_{\mathcal Y}^2\left(\overline b_{\oplus,\bm\lambda}(\widetilde{\bm Y}),\overline b_{\oplus,\bm\lambda}(\widetilde{\bm Y}^{(i)})\right)
    &\le\overline B_{\oplus,\bm\lambda}\left(\widetilde{\bm Y}^{(i)},\overline b_{\oplus,\bm\lambda}(\widetilde{\bm Y})\right)
    -\overline B_{\oplus,\bm\lambda}\left(\widetilde{\bm Y}^{(i)},\overline b_{\oplus,\bm\lambda}(\widetilde{\bm Y}^{(i)})\right).
\end{align*}
Adding the two inequalities above gives
\begin{align*}
    &2\kappa_{\mathcal Y}d_{\mathcal Y}^2\left(\overline b_{\oplus,\bm\lambda}(\widetilde{\bm Y}),\overline b_{\oplus,\bm\lambda}(\widetilde{\bm Y}^{(i)})\right) \\
    &\le\lambda_i\left(d_{\mathcal Y}^2\left(\widetilde Y_i,\overline b_{\oplus,\bm\lambda}(\widetilde{\bm Y}^{(i)})\right)
    -d_{\mathcal Y}^2\left(\widetilde Y_i,\overline b_{\oplus,\bm\lambda}(\widetilde{\bm Y})\right)-d_{\mathcal Y}^2\left(\widetilde Y_i',\overline b_{\oplus,\bm\lambda}(\widetilde{\bm Y}^{(i)})\right)
    +d_{\mathcal Y}^2\left(\widetilde Y_i',\overline b_{\oplus,\bm\lambda}(\widetilde{\bm Y})\right)\right)\\
    &\le\lambda_i\left|d_{\mathcal Y}^2\left(\widetilde Y_i,\overline b_{\oplus,\bm\lambda}(\widetilde{\bm Y}^{(i)})\right)
    -d_{\mathcal Y}^2\left(\widetilde Y_i,\overline b_{\oplus,\bm\lambda}(\widetilde{\bm Y})\right)-d_{\mathcal Y}^2\left(\widetilde Y_i',\overline b_{\oplus,\bm\lambda}(\widetilde{\bm Y}^{(i)})\right)
    +d_{\mathcal Y}^2\left(\widetilde Y_i',\overline b_{\oplus,\bm\lambda}(\widetilde{\bm Y})\right)\right|\\
    &\le C_{\mathcal Y}\lambda_i d_{\mathcal Y}(\widetilde Y_i,\widetilde Y_i')
    d_{\mathcal Y}\left(\overline b_{\oplus,\bm\lambda}(\widetilde{\bm Y}),\overline b_{\oplus,\bm\lambda}(\widetilde{\bm Y}^{(i)})\right),
\end{align*}
where the last inequality follows from Condition~\ref{ass:M4}. Since all quantities are nonnegative, this proves the lemma.
\end{proof}

\begin{lemma}\label{lem:r2-gap}
Under the conditions of Theorem~\ref{thm:r2-weighted-barycenter},
\begin{align*}
    \E\left[\overline Q_{\oplus,\bm\lambda}\left(\widetilde{\bm Y},\overline b_{\oplus,\bm\lambda}(\widetilde{\bm Y})\right)
    -\overline B_{\oplus,\bm\lambda}\left(\widetilde{\bm Y},\overline b_{\oplus,\bm\lambda}(\widetilde{\bm Y})\right)\right]
    \le\frac{C_{\mathcal Y}^2}{4\kappa_{\mathcal Y}}
    \sum_{i=1}^N\lambda_i^2\E\left[d_{\mathcal Y}^2(\widetilde Y_i,\widetilde Y_i')\right].
\end{align*}
\end{lemma}

\begin{proof}[Proof of \Cref{lem:r2-gap}]
By Lemma~\ref{lem:r2-integrability} and the triangle inequality, all expectations appearing below are finite, so the subtractions and rearrangements of the expectation terms are well-defined. By the definitions of $\overline Q_{\oplus,\bm\lambda}$ and $\overline B_{\oplus,\bm\lambda}$ and the independence of $\widetilde{\bm Y}'$ and $\widetilde{\bm Y}$,
\begin{align}
\begin{split}
    &\E\left[\overline Q_{\oplus,\bm\lambda}\left(\widetilde{\bm Y},\overline b_{\oplus,\bm\lambda}(\widetilde{\bm Y})\right)
    -\overline B_{\oplus,\bm\lambda}\left(\widetilde{\bm Y},\overline b_{\oplus,\bm\lambda}(\widetilde{\bm Y})\right)\right]\\
    &=\sum_{i=1}^N\lambda_i\E\left[
    d_{\mathcal Y}^2\left(\widetilde Y_i',\overline b_{\oplus,\bm\lambda}(\widetilde{\bm Y})\right)
    -d_{\mathcal Y}^2\left(\widetilde Y_i,\overline b_{\oplus,\bm\lambda}(\widetilde{\bm Y})\right)\right].
\end{split} \label{eq:r2-gap-proof-part1}
\end{align}
For every $i=1,\ldots,N$, interchanging $\widetilde Y_i$ and $\widetilde Y_i'$ preserves their joint distribution and interchanges $\overline b_{\oplus,\bm\lambda}(\widetilde{\bm Y})$ and $\overline b_{\oplus,\bm\lambda}(\widetilde{\bm Y}^{(i)})$. Therefore,
\begin{align}
\begin{split} 
    &\E\left[
    d_{\mathcal Y}^2\left(\widetilde Y_i',\overline b_{\oplus,\bm\lambda}(\widetilde{\bm Y})\right)
    -d_{\mathcal Y}^2\left(\widetilde Y_i,\overline b_{\oplus,\bm\lambda}(\widetilde{\bm Y})\right)\right]\\
    &=\frac12\E\left[
    d_{\mathcal Y}^2\left(\widetilde Y_i',\overline b_{\oplus,\bm\lambda}(\widetilde{\bm Y})\right)
    -d_{\mathcal Y}^2\left(\widetilde Y_i,\overline b_{\oplus,\bm\lambda}(\widetilde{\bm Y})\right)\right.\\
    &\quad\left.+d_{\mathcal Y}^2\left(\widetilde Y_i,\overline b_{\oplus,\bm\lambda}(\widetilde{\bm Y}^{(i)})\right)
    -d_{\mathcal Y}^2\left(\widetilde Y_i',\overline b_{\oplus,\bm\lambda}(\widetilde{\bm Y}^{(i)})\right)\right]\\
    &\le\frac{C_{\mathcal Y}}{2}\E\left[
    d_{\mathcal Y}(\widetilde Y_i,\widetilde Y_i')
    d_{\mathcal Y}\left(\overline b_{\oplus,\bm\lambda}(\widetilde{\bm Y}),\overline b_{\oplus,\bm\lambda}(\widetilde{\bm Y}^{(i)})\right)\right]\\
    &\le\frac{C_{\mathcal Y}^2\lambda_i}{4\kappa_{\mathcal Y}}
    \E\left[d_{\mathcal Y}^2(\widetilde Y_i,\widetilde Y_i')\right], \quad i=1, \ldots, N,
\end{split} \label{eq:r2-gap-proof-part2}
\end{align}
where the inequalities follow from Condition~\ref{ass:M4} and Lemma~\ref{lem:r2-replace}, respectively. Multiplying \eqref{eq:r2-gap-proof-part2} by $\lambda_i$, summing over $i=1,\ldots,N$, and using \eqref{eq:r2-gap-proof-part1} prove the lemma.
\end{proof}

\subsection{Proof of \Cref{thm:r2-weighted-barycenter}}
By Conditions~\ref{ass:M2}(iii)--(iv),
\begin{align}
\begin{split}
    \kappa_{\mathcal Y}d_{\mathcal Y}^2\left(\overline b_{\oplus,\bm\lambda}(\widetilde{\bm Y}),\overline q_{\oplus,\bm\lambda}(\widetilde{\bm Y})\right)
    &\le\overline Q_{\oplus,\bm\lambda}\left(\widetilde{\bm Y},\overline b_{\oplus,\bm\lambda}(\widetilde{\bm Y})\right)-\overline Q_{\oplus,\bm\lambda}\left(\widetilde{\bm Y},\overline q_{\oplus,\bm\lambda}(\widetilde{\bm Y})\right). \\ 
    \kappa_{\mathcal Y}d_{\mathcal Y}^2\left(\overline b_{\oplus,\bm\lambda}(\widetilde{\bm Y}),\overline q_{\oplus,\bm\lambda}(\widetilde{\bm Y})\right)
    &\le\overline B_{\oplus,\bm\lambda}\left(\widetilde{\bm Y},\overline q_{\oplus,\bm\lambda}(\widetilde{\bm Y})\right)-\overline B_{\oplus,\bm\lambda}\left(\widetilde{\bm Y},\overline b_{\oplus,\bm\lambda}(\widetilde{\bm Y})\right).
\end{split}\label{eq:r2-weighted-proof-part1}
\end{align}
Since $\overline q_{\oplus,\bm\lambda}(\widetilde{\bm Y})$ is deterministic, it holds that
\begin{align}
    \E\left[\overline B_{\oplus,\bm\lambda}\left(\widetilde{\bm Y},\overline q_{\oplus,\bm\lambda}(\widetilde{\bm Y})\right)\right]
    =\overline Q_{\oplus,\bm\lambda}\left(\widetilde{\bm Y},\overline q_{\oplus,\bm\lambda}(\widetilde{\bm Y})\right).
    \label{eq:r2-weighted-proof-part2}
\end{align}
By \eqref{eq:r2-weighted-proof-part1} and \eqref{eq:r2-weighted-proof-part2}, we have
\begin{align*}
    2\kappa_{\mathcal Y}\E\left[d_{\mathcal Y}^2\left(\overline b_{\oplus,\bm\lambda}(\widetilde{\bm Y}),\overline q_{\oplus,\bm\lambda}(\widetilde{\bm Y})\right)\right]
    &\le\E\left[\overline Q_{\oplus,\bm\lambda}\left(\widetilde{\bm Y},\overline b_{\oplus,\bm\lambda}(\widetilde{\bm Y})\right)-\overline B_{\oplus,\bm\lambda}\left(\widetilde{\bm Y},\overline b_{\oplus,\bm\lambda}(\widetilde{\bm Y})\right)\right]\\
    &\quad+\E\left[\overline B_{\oplus,\bm\lambda}\left(\widetilde{\bm Y},\overline q_{\oplus,\bm\lambda}(\widetilde{\bm Y})\right)\right]-\overline Q_{\oplus,\bm\lambda}\left(\widetilde{\bm Y},\overline q_{\oplus,\bm\lambda}(\widetilde{\bm Y})\right)\\
    &=\E\left[\overline Q_{\oplus,\bm\lambda}\left(\widetilde{\bm Y},\overline b_{\oplus,\bm\lambda}(\widetilde{\bm Y})\right)-\overline B_{\oplus,\bm\lambda}\left(\widetilde{\bm Y},\overline b_{\oplus,\bm\lambda}(\widetilde{\bm Y})\right)\right]\\
    &\le\frac{C_{\mathcal Y}^2}{4\kappa_{\mathcal Y}}\sum_{i=1}^N\lambda_i^2\E\left[d_{\mathcal Y}^2(\widetilde Y_i,\widetilde Y_i')\right],
\end{align*}
where the last inequality follows from Lemma~\ref{lem:r2-gap}. This proves Theorem~\ref{thm:r2-weighted-barycenter}.

\begin{corollary}\label{cor:r2-moment}
Under the conditions of Theorem~\ref{thm:r2-weighted-barycenter},
\begin{align} \label{eq:r2-moment-part1}
    \E\left[d_{\mathcal Y}^2\left(\overline b_{\oplus,\bm\lambda}(\widetilde{\bm Y}),\overline q_{\oplus,\bm\lambda}(\widetilde{\bm Y})\right)\right]
    \le\frac{C_{\mathcal Y}^2}{2\kappa_{\mathcal Y}^2}\sum_{i=1}^N\lambda_i^2\E\left[d_{\mathcal Y}^2\left(\widetilde Y_i,\overline q_{\oplus,\bm\lambda}(\widetilde{\bm Y})\right)\right].
\end{align}
Suppose that, for some $s\in[0,\infty)$,
\begin{align} \label{eq:r2-moment-assumption}
    \E\left[d_{\mathcal Y}^2\left(\widetilde Y_i,\overline q_{\oplus,\bm\lambda}(\widetilde{\bm Y})\right)\right]\le s^2, \quad i=1,\ldots,N.
\end{align}
Then
\begin{align} \label{eq:r2-moment-part2}
    \Pbb\left(d_{\mathcal Y}\left(\overline b_{\oplus,\bm\lambda}(\widetilde{\bm Y}),\overline q_{\oplus,\bm\lambda}(\widetilde{\bm Y})\right)>\frac{\sqrt{2}C_{\mathcal Y}s}{\kappa_{\mathcal Y}}\sqrt{\sum_{i=1}^N\lambda_i^2}\right)\le\frac14.
\end{align}
\end{corollary}

\begin{proof}[Proof of \Cref{cor:r2-moment}]
Note that $\widetilde Y_i'$ has the same distribution as $\widetilde Y_i$ for every $i=1,\ldots,N$. Therefore, by the triangle inequality and the Cauchy--Schwarz inequality, we have
\begin{align*}
    \E\left[d_{\mathcal Y}^2(\widetilde Y_i,\widetilde Y_i')\right]
    &\le2\E\left[d_{\mathcal Y}^2\left(\widetilde Y_i,\overline q_{\oplus,\bm\lambda}(\widetilde{\bm Y})\right)\right]
    +2\E\left[d_{\mathcal Y}^2\left(\widetilde Y_i',\overline q_{\oplus,\bm\lambda}(\widetilde{\bm Y})\right)\right]\\
    &=4\E\left[d_{\mathcal Y}^2\left(\widetilde Y_i,\overline q_{\oplus,\bm\lambda}(\widetilde{\bm Y})\right)\right].
\end{align*}
Combining the above with Theorem~\ref{thm:r2-weighted-barycenter} proves \eqref{eq:r2-moment-part1}.

Suppose \eqref{eq:r2-moment-assumption} holds for some $s\in[0,\infty)$. Trivially, we may assume $s>0$. By \eqref{eq:r2-moment-part1} and Markov's inequality, we have
\begin{align*}
    &\Pbb\left(d_{\mathcal Y}\left(\overline b_{\oplus,\bm\lambda}(\widetilde{\bm Y}),\overline q_{\oplus,\bm\lambda}(\widetilde{\bm Y})\right)>\frac{\sqrt{2}C_{\mathcal Y}s}{\kappa_{\mathcal Y}}\sqrt{\sum_{i=1}^N\lambda_i^2}\right)\le\frac{\kappa_{\mathcal Y}^2\E\left[d_{\mathcal Y}^2\left(\overline b_{\oplus,\bm\lambda}(\widetilde{\bm Y}),\overline q_{\oplus,\bm\lambda}(\widetilde{\bm Y})\right)\right]}{2C_{\mathcal Y}^2s^2\sum_{i=1}^N\lambda_i^2}\le\frac14.
\end{align*}
This proves \eqref{eq:r2-moment-part2}.
\end{proof}

For $h\in(0,h_0]$ and $b=1,\ldots,B$, define
\begin{align*}
    S_b&:=\sum_{i\in I_b}K_h(X_i,x_0), \quad N_b:=|I_b|,\\
    \bm X_b&:=(X_i:i\in I_b), \quad \bm Y_b:=(Y_i:i\in I_b),
\end{align*}
with indices in increasing order. On $S_b>0$, define
\begin{align*}
    \lambda_{bi}&:=\frac{K_h(X_i,x_0)}{S_b}, \quad i\in I_b,\\
    \bm\lambda_b&:=(\lambda_{bi}:i\in I_b).
\end{align*}
Then $\widehat m_{h,b}(x_0)=\overline b_{\oplus,\bm\lambda_b}(\bm Y_b)$. By Condition~\ref{ass:K1}, $\lambda_{bi}>0$ implies $X_i\in\overline B(x_0,h)\subset\overline B(x_0,h_0)$ for $i\in I_b$. Therefore, whenever Conditions~\ref{ass:M3}, \ref{ass:S1}, or \ref{ass:S2} are applied to the block observations in the Supplementary Material, we restrict to indices with $\lambda_{bi}>0$ for $i\in I_b$.

\begin{corollary}\label{cor:r2-conditional}
Suppose Conditions~\ref{ass:K1} and \ref{ass:M1}--\ref{ass:M4} hold. Fix $h\in(0,h_0]$ and $b=1, \ldots, B$. Then, almost surely on $S_b>0$,
\begin{align} 
\begin{split} 
    &\E\left[d_{\mathcal Y}^2\left(\widehat m_{h,b}(x_0),\overline m_{\oplus,\bm\lambda_b}(\bm X_b)\right)\,\middle|\,X_1,\ldots,X_n\right]\\
    &\le\frac{C_{\mathcal Y}^2}{2\kappa_{\mathcal Y}^2}\sum_{i\in I_b}\lambda_{bi}^2
    \E\left[d_{\mathcal Y}^2\left(Y_i,\overline m_{\oplus,\bm\lambda_b}(\bm X_b)\right)\,\middle|\,X_1,\ldots,X_n\right],
\end{split}\label{eq:r2-conditional-mean-square}
\end{align}
where we omit the indices with $\lambda_{bi}=0$. Moreover,
\begin{align} \label{eq:r2-conditional-probability}
    \Pbb\left(
        d_{\mathcal Y}\left(\widehat m_{h,b}(x_0),\overline m_{\oplus,\bm\lambda_b}(\bm X_b)\right)
        >
        \frac{2C_{\mathcal Y}\sigma}{\kappa_{\mathcal Y}}
        \sqrt{\left(1+\frac{1}{\kappa_{\mathcal Y}}\right)\sum_{i\in I_b}\lambda_{bi}^2}
        \,\middle|\,X_1,\ldots,X_n
    \right)
    \le\frac14.
\end{align}
\end{corollary}

\begin{proof}[Proof of \Cref{cor:r2-conditional}]
Conditionally on $X_1,\ldots,X_n$, the weights $\lambda_{bi}$ for $i\in I_b$ are deterministic and the responses $Y_i$ for $i\in I_b$ are independent with conditional distributions $P_{Y\mid X=X_i}$ for $i\in I_b$. Moreover, the weighted population and empirical Fr\'echet means are $\overline m_{\oplus,\bm\lambda_b}(\bm X_b)$ and $\widehat m_{h,b}(x_0)$, respectively. By Lemma~\ref{lem:block-second-moment} below, for every $i\in I_b$ with $\lambda_{bi}>0$, the conditional distribution of $Y_i$ has a finite second moment about $\overline m_{\oplus,\bm\lambda_b}(\bm X_b)$. Therefore, applying Corollary~\ref{cor:r2-moment} conditionally on $X_1,\ldots,X_n$ yields \eqref{eq:r2-conditional-mean-square}.

By Lemma~\ref{lem:block-second-moment} below, for every $i\in I_b$ with $\lambda_{bi}>0$,
\begin{align*}
    \E\left[d_{\mathcal Y}^2\left(Y_i,\overline m_{\oplus,\bm\lambda_b}(\bm X_b)\right)\,\middle|\,X_1,\ldots,X_n\right]
    \le2\left(1+\frac{1}{\kappa_{\mathcal Y}}\right)\sigma^2.
\end{align*}
By the above and \eqref{eq:r2-conditional-mean-square}, we have
\begin{align*}
    \E\left[d_{\mathcal Y}^2\left(\widehat m_{h,b}(x_0),\overline m_{\oplus,\bm\lambda_b}(\bm X_b)\right)\,\middle|\,X_1,\ldots,X_n\right]
    \le\frac{C_{\mathcal Y}^2\sigma^2}{\kappa_{\mathcal Y}^2}\left(1+\frac{1}{\kappa_{\mathcal Y}}\right)\sum_{i\in I_b}\lambda_{bi}^2.
\end{align*}
Therefore, the above and Markov's inequality applied conditionally on $X_1,\ldots,X_n$ prove \eqref{eq:r2-conditional-probability}.
\end{proof}

\section{Proofs for Section \ref{subsec:second-order-theory}}\label{sec:supp-second-order}

In this section, we provide proofs of Theorems~\ref{thm:second-order-kernel-regression}, \ref{thm:minimax-lower-second-order} and related results. Auxiliary results used in this section are collected in Section \ref{sec:supp-auxiliary}. 

\subsection{Sufficient conditions for Condition~\ref{ass:S1}}

We provide two sufficient conditions for Condition~\ref{ass:S1}. The first considers Hilbert-valued responses under smoothness of the conditional Fr\'echet mean, while the second uses smoothness of the conditional density ratio for a general metric response space. Similar regularity conditions on the conditional density ratio were also used in \cite{ImEtAl2025} and \cite{ImJeon2026}.

\begin{proposition}\label{prop:s1-hilbert}
Suppose that $\mathcal Y=\mathbb H$ is a separable Hilbert space and that $m_\oplus:\mathcal X\to\mathbb H$ is twice continuously differentiable on $B(x_0,h_0)$ with bounded first and second derivatives. Then Condition~\ref{ass:S1} holds with any $L_S\in(0,\infty)$ satisfying
\begin{align*}
    L_S\ge
    2\max\left\{
        \sup_{x\in B(x_0,h_0)}\|Dm_\oplus(x)\|_{\mathrm{op}},
        \sup_{x\in B(x_0,h_0)}\|D^2m_\oplus(x)\|_{\mathrm{op}}
    \right\}.
\end{align*}
\end{proposition}

\begin{proof}[Proof of \Cref{prop:s1-hilbert}]
Since $\mathcal Y=\mathbb H$, note that
\begin{align*}
    m_\oplus(x)=\E[Y\mid X=x], \quad x\in B(x_0,h_0),
\end{align*}
and the definition of $S_\omega$ in \eqref{def:criterion-contrast} gives
\begin{align*}
    S_\omega(x)=\|\omega\|_{\mathbb H}^2-\|m_\oplus(x_0)\|_{\mathbb H}^2-2\langle m_\oplus(x),\omega-m_\oplus(x_0)\rangle_{\mathbb H}, \quad x\in B(x_0,h_0).
\end{align*}
Let $Dm_\oplus(x):\R^d\to\mathbb H$ and $D^2m_\oplus(x):\R^d\times\R^d\to\mathbb H$ denote the first and second derivatives of $m_\oplus$ at $x\in B(x_0,h_0)$. For $x\in B(x_0,h_0)$, we have
\begin{align*}
    DS_\omega(x)[v]&=-2\langle Dm_\oplus(x)[v],\omega-m_\oplus(x_0)\rangle_{\mathbb H},\\
    D^2S_\omega(x)[v,w]&=-2\langle D^2m_\oplus(x)[v,w],\omega-m_\oplus(x_0)\rangle_{\mathbb H}, \quad v,w\in\R^d.
\end{align*}
Therefore, by the above and the Cauchy--Schwarz inequality, for $x\in B(x_0,h_0)$,
\begin{align*}
    \|\nabla S_\omega(x)\|&\le2\|Dm_\oplus(x)\|_{\mathrm{op}}\|\omega-m_\oplus(x_0)\|_{\mathbb H},\\
    \|\nabla^2S_\omega(x)\|_{\mathrm{op}}&\le2\|D^2m_\oplus(x)\|_{\mathrm{op}}\|\omega-m_\oplus(x_0)\|_{\mathbb H}.
\end{align*}
Since the above holds for arbitrary $x\in B(x_0,h_0)$, this proves the proposition.
\end{proof}

\begin{proposition}\label{prop:pm-sufficient}
Suppose Conditions~\ref{ass:M1}--\ref{ass:M4} hold. Let $P_Y$ denote the marginal distribution of $Y$, and suppose that $P_{Y\mid X=x}$ is absolutely continuous with respect to $P_Y$ for every $x\in B(x_0,h_0)$. Suppose that the Radon--Nikodym derivatives can be chosen such that, for $y\in\mathcal Y$,
\begin{align*}
    g_y(x):=\frac{dP_{Y\mid X=x}}{dP_Y}(y), \quad x\in B(x_0,h_0).
\end{align*}
Also, for every $y\in\mathcal Y$, suppose that the map $x\mapsto g_y(x)$ is twice continuously differentiable on $B(x_0,h_0)$. Also, suppose that
\begin{align} \label{eq:pm-sufficient-assumption1}
    \E\left[d_{\mathcal Y}\left(Y,m_\oplus(x_0)\right)\right]<\infty
\end{align}
and that there exists $L_g\in(0,\infty)$ such that
\begin{align} \label{eq:pm-sufficient-assumption2}
    \sup_{y\in\mathcal Y}\sup_{x\in B(x_0,h_0)}\|\nabla g_y(x)\|\le L_g,
    \quad
    \sup_{y\in\mathcal Y}\sup_{x\in B(x_0,h_0)}\|\nabla^2g_y(x)\|_{\mathrm{op}}\le L_g.
\end{align}
Then Condition~\ref{ass:S1} holds with any $L_S\in(0,\infty)$ satisfying
\begin{align*}
    L_S\ge C_{\mathcal Y}L_g\E\left[d_{\mathcal Y}\left(Y,m_\oplus(x_0)\right)\right].
\end{align*}
\end{proposition}

\begin{proof}[Proof of \Cref{prop:pm-sufficient}]
Fix $\omega\in\overline B_{\mathcal Y}\left(m_\oplus(x_0),\sigma/\sqrt{\kappa_{\mathcal Y}}\right)$ and define
\begin{align*}
    D_\omega(y):=d_{\mathcal Y}^2(y,\omega)-d_{\mathcal Y}^2\left(y,m_\oplus(x_0)\right)-d_{\mathcal Y}^2\left(m_\oplus(x_0),\omega\right).
\end{align*}
By Condition~\ref{ass:M4}, we have
\begin{align}\label{eq:pm-sufficient-part1}
    |D_\omega(y)|\le C_{\mathcal Y}d_{\mathcal Y}\left(\omega,m_\oplus(x_0)\right)d_{\mathcal Y}\left(y,m_\oplus(x_0)\right), \quad y\in\mathcal Y.
\end{align}
By the definitions of $S_\omega$ and $g_y$, it holds that
\begin{align}
\begin{split}
    S_\omega(x)
    &=\E\left[d_{\mathcal Y}^2(Y,\omega)-d_{\mathcal Y}^2\left(Y,m_\oplus(x_0)\right)\,\middle|\,X=x\right]\\
    &=\int_{\mathcal Y}\left(d_{\mathcal Y}^2(y,\omega)-d_{\mathcal Y}^2\left(y,m_\oplus(x_0)\right)\right)g_y(x)\,dP_Y(y)\\
    &=\int_{\mathcal Y}\left(D_\omega(y)+d_{\mathcal Y}^2\left(m_\oplus(x_0),\omega\right)\right)g_y(x)\,dP_Y(y)\\
    &=\E\left[D_\omega(Y)g_Y(x)\right]+d_{\mathcal Y}^2\left(m_\oplus(x_0),\omega\right),\quad x\in B(x_0,h_0),
\end{split} \label{eq:pm-sufficient-part2}    
\end{align}
where the last expectation is with respect to $P_Y$. By \eqref{eq:pm-sufficient-assumption1}, \eqref{eq:pm-sufficient-assumption2}, \eqref{eq:pm-sufficient-part1}, and \eqref{eq:pm-sufficient-part2}, the dominated convergence theorem gives
\begin{align*}
    \nabla S_\omega(x)=\E\left[D_\omega(Y)\nabla g_Y(x)\right],\quad \nabla^2S_\omega(x)=\E\left[D_\omega(Y)\nabla^2g_Y(x)\right],\quad x\in B(x_0,h_0),
\end{align*}
and these derivatives are continuous on $B(x_0,h_0)$. Therefore, by the above, we have
\begin{align*}
    \|\nabla S_\omega(x_0)\|
    &\le C_{\mathcal Y}L_g\E\left[d_{\mathcal Y}\left(Y,m_\oplus(x_0)\right)\right]d_{\mathcal Y}\left(\omega,m_\oplus(x_0)\right),\\
    \sup_{x\in B(x_0,h_0)}\|\nabla^2S_\omega(x)\|_{\mathrm{op}}
    &\le C_{\mathcal Y}L_g\E\left[d_{\mathcal Y}\left(Y,m_\oplus(x_0)\right)\right]d_{\mathcal Y}\left(\omega,m_\oplus(x_0)\right).
\end{align*}
Since the above holds for arbitrary $\omega\in\overline B_{\mathcal Y}\left(m_\oplus(x_0),\sigma/\sqrt{\kappa_{\mathcal Y}}\right)$, this proves the proposition.
\end{proof}

\subsection{Bias and kernel weights}\label{subsec:bias-second-order}

Define
\begin{align*}
    \mu_{2,K}:=\int_{\R^d}\|u\|^2K(u)\,du, \quad
    \nu_{2,K}:=\int_{\R^d}\|u\|^2K(u)^2\,du,
\end{align*}
and let
\begin{align*}
    M_f:=\sup_{x\in\overline B(x_0,h_0)}f_X(x), \quad
    L_f:=\sup_{x\in\overline B(x_0,h_0)}\|\nabla f_X(x)\|.
\end{align*}
Note that these constants are finite by Conditions~\ref{ass:K1} and \ref{ass:D2}.

\begin{lemma}\label{lem:bias-second-order}
Suppose Conditions~\ref{ass:K1}--\ref{ass:K2}, \ref{ass:D1}--\ref{ass:D2}, \ref{ass:M1}--\ref{ass:M3}, and \ref{ass:S1} hold. For $h\in(0,h_0)$,
\begin{align*}
    d_{\mathcal Y}(\widetilde m_h(x_0),m_\oplus(x_0))
    \le\frac{2\mu_{2,K}L_S}{\kappa_{\mathcal Y}f_X(x_0)}\left(L_f+\frac{M_f}{2}\right)h^2.
\end{align*}
\end{lemma}

\begin{proof}[Proof of \Cref{lem:bias-second-order}]
Let $h\in(0,h_0)$ be given, and fix $\omega\in\overline B_{\mathcal Y}(m_\oplus(x_0),\sigma/\sqrt{\kappa_{\mathcal Y}})$. Condition~\ref{ass:S1} and Taylor's formula give
\begin{align} \label{eq:bias-second-order-proof1}
    \left|S_\omega(x_0+hu)-S_\omega(x_0)-h\langle\nabla S_\omega(x_0),u\rangle\right|
    \le\frac{L_S}{2}d_{\mathcal Y}(\omega,m_\oplus(x_0))h^2\|u\|^2, \quad u\in\overline B(0,1).
\end{align}
By Conditions~\ref{ass:K2} and \ref{ass:D2},
\begin{align} 
\begin{split} 
    \left\|h\int_{\overline B(0,1)}uK(u)f_X(x_0+hu)\,du\right\|
    &=\left\|h\int_{\overline B(0,1)}uK(u)\left(f_X(x_0+hu)-f_X(x_0)\right)\,du\right\|\\
    &\le L_f\mu_{2,K}h^2.
\end{split} \label{eq:bias-second-order-proof2}
\end{align}
Also, by the change of variables, we have
\begin{align}
\begin{split}
    &\E[K_h(X,x_0)S_\omega(X)]-S_\omega(x_0)\E[K_h(X,x_0)]\\
    &=\int_{\overline B(0,1)}K(u)f_X(x_0+hu)\left(S_\omega(x_0+hu)-S_\omega(x_0)\right)\,du\\
    &=\left\langle\nabla S_\omega(x_0),h\int_{\overline B(0,1)}uK(u)f_X(x_0+hu)\,du\right\rangle\\
    &\quad+\int_{\overline B(0,1)}K(u)f_X(x_0+hu)
    \left(S_\omega(x_0+hu)-S_\omega(x_0)-h\langle\nabla S_\omega(x_0),u\rangle\right)\,du.
\end{split} \label{eq:bias-second-order-proof3}
\end{align}
Therefore, by \eqref{eq:bias-second-order-proof1}, \eqref{eq:bias-second-order-proof2}, \eqref{eq:bias-second-order-proof3}, and Condition~\ref{ass:S1}, we have
\begin{align*}
    &\left|\E[K_h(X,x_0)S_\omega(X)]-S_\omega(x_0)\E[K_h(X,x_0)]\right|\\
    &\le L_Sd_{\mathcal Y}(\omega,m_\oplus(x_0))L_f\mu_{2,K}h^2
    +\frac{L_S}{2}d_{\mathcal Y}(\omega,m_\oplus(x_0))M_f\mu_{2,K}h^2\\
    &=\mu_{2,K}L_S\left(L_f+\frac{M_f}{2}\right)h^2d_{\mathcal Y}(\omega,m_\oplus(x_0)).
\end{align*}
Since $\E[K_h(X,x_0)]>0$ by Lemma~\ref{lem:kernel-target-positive} below, the above implies
\begin{align} \label{eq:bias-second-order-proof4}
    \left|\frac{\E[K_h(X,x_0)S_\omega(X)]}{\E[K_h(X,x_0)]}-S_\omega(x_0)\right|
    \le\frac{\mu_{2,K}L_S}{\E[K_h(X,x_0)]}\left(L_f+\frac{M_f}{2}\right)h^2d_{\mathcal Y}(\omega,m_\oplus(x_0)).
\end{align}
On the other hand, by the definition of $\widetilde m_h(x_0)$, 
\begin{align} \label{eq:bias-second-order-proof5}
    \frac{\E\left[K_h(X,x_0)S_{\widetilde m_h(x_0)}(X)\right]}{\E[K_h(X,x_0)]}
    =\widetilde M_h(x_0,\widetilde m_h(x_0))-\widetilde M_h(x_0,m_\oplus(x_0))\le0.
\end{align}
Also, Condition~\ref{ass:M2}(ii) with $\bm\lambda=(1)$ and $\bm x=(x_0)$ gives
\begin{align} \label{eq:bias-second-order-proof6}
    \kappa_{\mathcal Y}d_{\mathcal Y}^2(\widetilde m_h(x_0),m_\oplus(x_0))
    &\le S_{\widetilde m_h(x_0)}(x_0)
\end{align} 
Since we can apply \eqref{eq:bias-second-order-proof4} for $\omega=\widetilde m_h(x_0)$ by Lemmas~\ref{lem:block-second-moment} below, \eqref{eq:bias-second-order-proof4}, \eqref{eq:bias-second-order-proof5}, \eqref{eq:bias-second-order-proof6}, and Lemma~\ref{lem:kernel-target-positive} below imply 
\begin{align*}
    \kappa_{\mathcal Y}d_{\mathcal Y}^2(\widetilde m_h(x_0),m_\oplus(x_0))
    &\le S_{\widetilde m_h(x_0)}(x_0)\\
    &\le S_{\widetilde m_h(x_0)}(x_0)
    -\frac{\E[K_h(X,x_0)S_{\widetilde m_h(x_0)}(X)]}{\E[K_h(X,x_0)]}\\
    &\le\frac{\mu_{2,K}L_S}{\E[K_h(X,x_0)]}
    \left(L_f+\frac{M_f}{2}\right)h^2
    d_{\mathcal Y}(\widetilde m_h(x_0),m_\oplus(x_0))\\
    &\le\frac{2\mu_{2,K}L_S}{f_X(x_0)}
    \left(L_f+\frac{M_f}{2}\right)h^2
    d_{\mathcal Y}(\widetilde m_h(x_0),m_\oplus(x_0)).
\end{align*}
Therefore, the above and Lemma~\ref{lem:kernel-target-positive} prove the lemma. 
\end{proof}

\begin{lemma}\label{lem:self-normalized-first-moment}
Suppose Conditions~\ref{ass:K1}--\ref{ass:K2} and \ref{ass:D1}--\ref{ass:D2} hold. For $h\in(0,h_0)$ and $b=1,\ldots,B$, if 
\begin{align} \label{eq:self-normalized-first-moment-ass1}
    N_bh^df_X(x_0)\ge\frac{56\|K\|_\infty}{3}\log 24,
\end{align}
then, with probability at least $11/12$,
\begin{align}
\begin{split}
    S_b&\ge\frac{N_b\E[K_h(X,x_0)]}{2}>0,\\
    \left\|\sum_{i\in I_b}\lambda_{bi}(X_i-x_0)\right\|
    &\le\frac{4L_f\mu_{2,K}}{f_X(x_0)}h^2
    +4\sqrt{\frac{24M_f\nu_{2,K}}{f_X(x_0)}}
    \frac{h}{\sqrt{N_bh^df_X(x_0)}}.
\end{split} \label{eq:self-normalized-first-moment}
\end{align}
\end{lemma}

\begin{proof}[Proof of \Cref{lem:self-normalized-first-moment}]
Let $h\in(0,h_0)$ and $b=1,\ldots,B$ be given. By the change of variables and \eqref{eq:bias-second-order-proof2}, we have
\begin{align}
    \left\|\E\left[K_h(X,x_0)(X-x_0)\right]\right\|
    &=\left\|h\int_{\overline B(0,1)}uK(u)\left(f_X(x_0+hu)-f_X(x_0)\right)\,du\right\| \notag\\
    &\le L_f\mu_{2,K}h^2. \label{eq:self-normalized-first-moment-proof1}
\end{align}
Similarly, by the change of variables,
\begin{align}
    \E\left[\|K_h(X,x_0)(X-x_0)\|^2\right]
    &=h^{2-d}\int_{\overline B(0,1)}\|u\|^2K(u)^2f_X(x_0+hu)\,du \notag\\
    &\le M_f\nu_{2,K}h^{2-d}. \label{eq:self-normalized-first-moment-proof2}
\end{align}
Therefore, by independence of $X_i$ for $i\in I_b$, and \eqref{eq:self-normalized-first-moment-proof2},
\begin{align*}
    &\E\left[\left\|\frac1{N_b}\sum_{i\in I_b}
    \left(K_h(X_i,x_0)(X_i-x_0)-\E[K_h(X,x_0)(X-x_0)]\right)\right\|^2\right]\\
    &=\frac1{N_b^2}\sum_{i\in I_b}
    \E\left[\left\|K_h(X_i,x_0)(X_i-x_0)-\E[K_h(X,x_0)(X-x_0)]\right\|^2\right]\\
    &\le\frac1{N_b}\E\left[\|K_h(X,x_0)(X-x_0)\|^2\right]
    \le\frac{M_f\nu_{2,K}}{N_bh^{d-2}}.
\end{align*}
Thus, applying Markov's inequality with the above, we have
\begin{align} 
\begin{split}
    &\Pbb\left(
    \left\|\frac1{N_b}\sum_{i\in I_b}
    \left(K_h(X_i,x_0)(X_i-x_0)-\E[K_h(X,x_0)(X-x_0)]\right)\right\|
    >\sqrt{\frac{24M_f\nu_{2,K}}{N_bh^{d-2}}}
    \right)\\
    &\le\frac{N_bh^{d-2}}{24M_f\nu_{2,K}}\E\left[\left\|\frac1{N_b}\sum_{i\in I_b}
    \left(K_h(X_i,x_0)(X_i-x_0)-\E[K_h(X,x_0)(X-x_0)]\right)\right\|^2\right]\\
    &\le\frac1{24}.
\end{split} \label{eq:self-normalized-first-moment-proof3}
\end{align}
Suppose \eqref{eq:self-normalized-first-moment-ass1} holds. Then by Lemmas~\ref{lem:kernel-target-positive} and \ref{lem:kernel-mass} below,
\begin{align} \label{eq:self-normalized-first-moment-proof4}
    \Pbb\left(S_b<\frac{N_b\E[K_h(X,x_0)]}{2}\right)
    \le\exp\left(-\frac{3N_bh^df_X(x_0)}{56\|K\|_\infty}\right)
    \le\frac1{24}.
\end{align}
Therefore, by \eqref{eq:self-normalized-first-moment-proof3}, \eqref{eq:self-normalized-first-moment-proof4}, and the union bound, with probability at least $11/12$, we have
\begin{align}
\begin{split}
    S_b\ge\frac{N_b\E[K_h(X,x_0)]}{2}&>0,\\
    \left\|\frac1{N_b}\sum_{i\in I_b}
    \left(K_h(X_i,x_0)(X_i-x_0)-\E[K_h(X,x_0)(X-x_0)]\right)\right\|
    &\le\sqrt{\frac{24M_f\nu_{2,K}}{N_bh^{d-2}}}.
\end{split} \label{eq:self-normalized-first-moment-proof5}
\end{align}
On the event in \eqref{eq:self-normalized-first-moment-proof5}, Lemma~\ref{lem:kernel-target-positive} gives
\begin{align}
    \frac{S_b}{N_b}
    \ge\frac{\E[K_h(X,x_0)]}{2}
    \ge\frac{f_X(x_0)}{4}.
    \label{eq:self-normalized-first-moment-proof6}
\end{align}
Also, by the triangle inequality, \eqref{eq:self-normalized-first-moment-proof1}, and \eqref{eq:self-normalized-first-moment-proof5}, we have
\begin{align}
\begin{split}
    &\left\|\frac1{N_b}\sum_{i\in I_b}K_h(X_i,x_0)(X_i-x_0)\right\| \\
    &\le
    \left\|\frac1{N_b}\sum_{i\in I_b}
    \left(K_h(X_i,x_0)(X_i-x_0)-\E[K_h(X,x_0)(X-x_0)]\right)\right\| \\
    &\quad+\left\|\E[K_h(X,x_0)(X-x_0)]\right\| \\
    &\le
    \sqrt{\frac{24M_f\nu_{2,K}}{N_bh^{d-2}}}
    +L_f\mu_{2,K}h^2,
\end{split} \label{eq:self-normalized-first-moment-proof7}
\end{align}
on the event in \eqref{eq:self-normalized-first-moment-proof5}. Therefore, by \eqref{eq:self-normalized-first-moment-proof6}, \eqref{eq:self-normalized-first-moment-proof7}, and the definitions of $S_b$ and $\lambda_{bi}$ for $i\in I_b$, we have
\begin{align}
\begin{split}
    \left\|\sum_{i\in I_b}\lambda_{bi}(X_i-x_0)\right\|
    &=\frac{\left\|N_b^{-1}\sum_{i\in I_b}K_h(X_i,x_0)(X_i-x_0)\right\|}{S_b/N_b}\\
    &\le\frac{4L_f\mu_{2,K}}{f_X(x_0)}h^2
    +4\sqrt{\frac{24M_f\nu_{2,K}}{f_X(x_0)}}
    \frac{h}{\sqrt{N_bh^df_X(x_0)}},
\end{split} \label{eq:self-normalized-first-moment-proof8}
\end{align}
on the event in \eqref{eq:self-normalized-first-moment-proof5}. \eqref{eq:self-normalized-first-moment-proof5} and \eqref{eq:self-normalized-first-moment-proof8} prove the lemma.
\end{proof}

\begin{lemma}\label{lem:conditional-bias-second-order}
Suppose Conditions~\ref{ass:K1}, \ref{ass:M1}--\ref{ass:M3}, and \ref{ass:S1} hold. For $h\in(0,h_0)$ and $b=1,\ldots,B$, if $S_b>0$, then
\begin{align*}
    d_{\mathcal Y}\left(\overline m_{\oplus,\bm\lambda_b}(\bm X_b),m_\oplus(x_0)\right)
    \le\frac{L_S}{\kappa_{\mathcal Y}}\left(\left\|\sum_{i\in I_b}\lambda_{bi}(X_i-x_0)\right\|+\frac{h^2}{2}\right).
\end{align*}
\end{lemma}

\begin{proof}[Proof of \Cref{lem:conditional-bias-second-order}]
Let $h\in(0,h_0)$ and $b=1,\ldots,B$ be given, and suppose $S_b>0$. Fix $\omega\in\overline B_{\mathcal Y}(m_\oplus(x_0),\sigma/\sqrt{\kappa_{\mathcal Y}})$. By Taylor's formula and Condition~\ref{ass:S1}, we have
\begin{align} 
\begin{split}
    &-\sum_{i\in I_b}\lambda_{bi}S_\omega(X_i)+S_\omega(x_0)+\left\langle \nabla S_\omega(x_0),\sum_{i\in I_b}\lambda_{bi}(X_i-x_0)\right\rangle\\
    &\le\left|\sum_{i\in I_b}\lambda_{bi}S_\omega(X_i)-S_\omega(x_0)-\left\langle \nabla S_\omega(x_0),\sum_{i\in I_b}\lambda_{bi}(X_i-x_0)\right\rangle \right|\\
    &\le\frac{L_S}{2}d_{\mathcal Y}(\omega,m_\oplus(x_0))\sum_{i\in I_b}\lambda_{bi}\|X_i-x_0\|^2\\
    &\le\frac{L_S}{2}d_{\mathcal Y}(\omega,m_\oplus(x_0))h^2,
\end{split} \label{eq:conditional-bias-second-order-proof1}
\end{align}
where the last inequality follows since $\lambda_{bi}>0$ implies $X_i\in\overline B(x_0,h)$ for $i\in I_b$. By Lemma~\ref{lem:block-second-moment} below, we can apply the above at $\omega=\overline m_{\oplus,\bm\lambda_b}(\bm X_b)$. Also, by the definition of $\overline m_{\oplus,\bm\lambda_b}(\bm X_b)$, we have
\begin{align} \label{eq:conditional-bias-second-order-proof2}
    \sum_{i\in I_b}\lambda_{bi}S_{\overline m_{\oplus,\bm\lambda_b}(\bm X_b)}(X_i)
    =\overline M_{\oplus,\bm\lambda_b}\left(\bm X_b,\overline m_{\oplus,\bm\lambda_b}(\bm X_b)\right)
    -\overline M_{\oplus,\bm\lambda_b}(\bm X_b,m_\oplus(x_0))
    \le0.
\end{align}
On the other hand, Condition~\ref{ass:M2}(ii) with $N=1$, $\bm\lambda=(1)$, and $\bm x=(x_0)$ gives
\begin{align} \label{eq:conditional-bias-second-order-proof3}
    \kappa_{\mathcal Y}d_{\mathcal Y}^2\left(\overline m_{\oplus,\bm\lambda_b}(\bm X_b),m_\oplus(x_0)\right)
    \le S_{\overline m_{\oplus,\bm\lambda_b}(\bm X_b)}(x_0).
\end{align}
Therefore, by \eqref{eq:conditional-bias-second-order-proof1}, \eqref{eq:conditional-bias-second-order-proof2}, \eqref{eq:conditional-bias-second-order-proof3}, and Condition~\ref{ass:S1}, we have
\begin{align*}
    &\kappa_{\mathcal Y}d_{\mathcal Y}^2\left(\overline m_{\oplus,\bm\lambda_b}(\bm X_b),m_\oplus(x_0)\right)\\
    &\le S_{\overline m_{\oplus,\bm\lambda_b}(\bm X_b)}(x_0)\\
    &\le \sum_{i\in I_b}\lambda_{bi}
    S_{\overline m_{\oplus,\bm\lambda_b}(\bm X_b)}(X_i)
    -\left\langle
    \nabla S_{\overline m_{\oplus,\bm\lambda_b}(\bm X_b)}(x_0),
    \sum_{i\in I_b}\lambda_{bi}(X_i-x_0)
    \right\rangle \\
    &\quad+\frac{L_S}{2}
    d_{\mathcal Y}\left(\overline m_{\oplus,\bm\lambda_b}(\bm X_b),m_\oplus(x_0)\right)h^2\\
    &\le
    \left\|
    \nabla S_{\overline m_{\oplus,\bm\lambda_b}(\bm X_b)}(x_0)
    \right\|
    \left\|\sum_{i\in I_b}\lambda_{bi}(X_i-x_0)\right\|+\frac{L_S}{2}
    d_{\mathcal Y}\left(\overline m_{\oplus,\bm\lambda_b}(\bm X_b),m_\oplus(x_0)\right)h^2\\
    &\le
    L_Sd_{\mathcal Y}\left(\overline m_{\oplus,\bm\lambda_b}(\bm X_b),m_\oplus(x_0)\right)
    \left(
    \left\|\sum_{i\in I_b}\lambda_{bi}(X_i-x_0)\right\|
    +\frac{h^2}{2}
    \right).
\end{align*}
This proves the lemma.
\end{proof}

\subsection{Proof of \Cref{thm:second-order-kernel-regression}}
We first prepare the following lemma. 
\begin{lemma}\label{lem:one-kernel-block-second-order}
Under the conditions of Theorem~\ref{thm:second-order-kernel-regression}, there exist constants $C_{\mathrm{blk}},C_{\mathrm{bias}},C_{\mathrm{des}},C_{\mathrm{resp}}\in(0,\infty)$ depending only on the fixed constants in those conditions such that the following holds. For $h\in(0,h_0)$ and $b=1,\ldots,B$, if
\begin{align*}
    N_bh^df_X(x_0)\ge C_{\mathrm{blk}},
\end{align*}
then
\begin{align*}
    \Pbb\left(
    d_{\mathcal Y}(\widehat m_{h,b}(x_0),m_\oplus(x_0))
    >
    C_{\mathrm{bias}}h^2
    +\frac{C_{\mathrm{resp}}\sigma+C_{\mathrm{des}}h}
    {\sqrt{N_bh^df_X(x_0)}}
    \right)
    \le\frac5{16}.
\end{align*}
\end{lemma}

\begin{proof}[Proof of \Cref{lem:one-kernel-block-second-order}]
Let $h\in(0,h_0)$ and $b=1,\ldots,B$ be given, and define
\begin{align}
\begin{split}\label{def:second-order-theorem-constant1}
    C_{\mathrm{blk}}:=\frac{56\|K\|_\infty}{3}\log 24,\quad
    &C_{\mathrm{bias}}
    :=\frac{L_S}{\kappa_{\mathcal Y}}
    \left(\frac{4L_f\mu_{2,K}}{f_X(x_0)}+\frac12\right),\\
    C_{\mathrm{des}}
    :=\frac{4L_S}{\kappa_{\mathcal Y}}
    \sqrt{\frac{24M_f\nu_{2,K}}{f_X(x_0)}},\quad
    &C_{\mathrm{resp}}
    :=\frac{4C_{\mathcal Y}}{\kappa_{\mathcal Y}}
    \sqrt{\left(1+\frac{1}{\kappa_{\mathcal Y}}\right)\|K\|_\infty}.
\end{split}
\end{align}
Suppose that $N_bh^df_X(x_0)\ge C_{\mathrm{blk}}$. Let $\mathcal A_b$ be the event on which the two inequalities in Lemma~\ref{lem:self-normalized-first-moment} hold. By Lemma~\ref{lem:self-normalized-first-moment},
\begin{align} \label{eq:one-kernel-block-second-order-proof1}
    \Pbb(\mathcal A_b)\ge\frac{11}{12}.
\end{align}
On $\mathcal A_b$, Lemmas~\ref{lem:self-normalized-first-moment} and \ref{lem:conditional-bias-second-order} imply
\begin{align}
\begin{split}
    &d_{\mathcal Y}\left(\overline m_{\oplus,\bm\lambda_b}(\bm X_b),m_\oplus(x_0)\right)\\
    &\le\frac{L_S}{\kappa_{\mathcal Y}}
    \left[
    \left(\frac{4L_f\mu_{2,K}}{f_X(x_0)}+\frac12\right)h^2
    +4\sqrt{\frac{24M_f\nu_{2,K}}{f_X(x_0)}}
    \frac{h}{\sqrt{N_bh^df_X(x_0)}}
    \right]\\
    &=C_{\mathrm{bias}}h^2
    +\frac{C_{\mathrm{des}}h}{\sqrt{N_bh^df_X(x_0)}}.
\end{split}\label{eq:one-kernel-block-second-order-proof2}
\end{align}
By the above, Lemma~\ref{lem:weight-square} below, and Corollary~\ref{cor:r2-conditional}, on $\mathcal A_b$, we have
\begin{align*}
    &\Pbb\left(
    d_{\mathcal Y}\left(\widehat m_{h,b}(x_0),\overline m_{\oplus,\bm\lambda_b}(\bm X_b)\right)
    >
    \frac{C_{\mathrm{resp}}\sigma}{\sqrt{N_bh^df_X(x_0)}}
    \,\middle|\,X_1,\ldots,X_n
    \right)\\
    &\le
    \Pbb\left(
    d_{\mathcal Y}\left(\widehat m_{h,b}(x_0),\overline m_{\oplus,\bm\lambda_b}(\bm X_b)\right)
    >
    \frac{2C_{\mathcal Y}\sigma}{\kappa_{\mathcal Y}}
    \sqrt{\left(1+\frac{1}{\kappa_{\mathcal Y}}\right)\sum_{i\in I_b}\lambda_{bi}^2}
    \,\middle|\,X_1,\ldots,X_n
    \right)\\
    &\le\frac14.
\end{align*}
Since $\mathcal A_b$ is determined only by the predictors, the above implies
\begin{align}
\begin{split}
    &\Pbb\left(
    \mathcal A_b\cap
    \left\{
    d_{\mathcal Y}\left(\widehat m_{h,b}(x_0),
    \overline m_{\oplus,\bm\lambda_b}(\bm X_b)\right)
    \le
    \frac{C_{\mathrm{resp}}\sigma}{\sqrt{N_bh^df_X(x_0)}}
    \right\}
    \right)\\
    &=\E\left[
    \mathds{1}_{\mathcal A_b}
    \Pbb\left(
    d_{\mathcal Y}\left(\widehat m_{h,b}(x_0),
    \overline m_{\oplus,\bm\lambda_b}(\bm X_b)\right)
    \le
    \frac{C_{\mathrm{resp}}\sigma}{\sqrt{N_bh^df_X(x_0)}}
    \,\middle|\,X_1,\ldots,X_n
    \right)
    \right]\\
    &\ge\frac34\Pbb(\mathcal A_b).
\end{split}\label{eq:one-kernel-block-second-order-proof3}
\end{align}
On the event in \eqref{eq:one-kernel-block-second-order-proof3}, \eqref{eq:one-kernel-block-second-order-proof2} and the triangle inequality give
\begin{align}
\begin{split}
    d_{\mathcal Y}\left(\widehat m_{h,b}(x_0),m_\oplus(x_0)\right)
    &\le d_{\mathcal Y}\left(\widehat m_{h,b}(x_0),
    \overline m_{\oplus,\bm\lambda_b}(\bm X_b)\right)
    +d_{\mathcal Y}\left(\overline m_{\oplus,\bm\lambda_b}(\bm X_b),m_\oplus(x_0)\right)\\
    &\le C_{\mathrm{bias}}h^2
    +\frac{C_{\mathrm{resp}}\sigma+C_{\mathrm{des}}h}
    {\sqrt{N_bh^df_X(x_0)}}.
\end{split}\label{eq:one-kernel-block-second-order-proof4}
\end{align}
Therefore, by \eqref{eq:one-kernel-block-second-order-proof1}, \eqref{eq:one-kernel-block-second-order-proof3}, and \eqref{eq:one-kernel-block-second-order-proof4}, we have
\begin{align*}
    \Pbb\left(
    d_{\mathcal Y}\left(\widehat m_{h,b}(x_0),m_\oplus(x_0)\right)
    >
    C_{\mathrm{bias}}h^2
    +\frac{C_{\mathrm{resp}}\sigma+C_{\mathrm{des}}h}
    {\sqrt{N_bh^df_X(x_0)}}
    \right)
    \le1-\frac34\Pbb(\mathcal A_b)
    \le\frac5{16}.
\end{align*}
This proves the lemma.
\end{proof}

Now we prove Theorem \ref{thm:second-order-kernel-regression}.  By Lemma~\ref{lem:binomial-calibration} below, there exists a universal constant $C_B\in(0,\infty)$ such that
\begin{align} \label{eq:second-order-kernel-regression-proof1}
    B_\delta\le C_B\log(1/\delta), \quad \delta\in\left(0,\frac{1}{2}\right].
\end{align}
Let $\delta\in(0,1/2]$ and $h\in(0,h_0)$ be given. Define
\begin{align} \label{def:second-order-theorem-constant2}
    C_{S,U}':=2C_{\mathrm{blk}}C_B, \quad C_{S,U}:=\max\left\{3C_{\mathrm{bias}},3\sqrt{2C_B}\left(C_{\mathrm{resp}}+C_{\mathrm{des}}\right)\right\},
\end{align}
where $C_{\mathrm{blk}},C_{\mathrm{bias}},C_{\mathrm{des}},C_{\mathrm{resp}}\in(0,\infty)$ are defined by \eqref{def:second-order-theorem-constant1}. Suppose that 
\begin{align} \label{eq:second-order-kernel-regression-proof2}
    B_\delta\le\lfloor n/2\rfloor, \quad nh^df_X(x_0)\ge C_{S,U}'\log(1/\delta).
\end{align}
Since the block sizes differ by at most one,
\begin{align} \label{eq:second-order-kernel-regression-proof3}
    N_b\ge\left\lfloor\frac{n}{B_\delta}\right\rfloor\ge\frac{n}{2B_\delta}, \quad b=1,\ldots,B_\delta.
\end{align}
By \eqref{eq:second-order-kernel-regression-proof1}, \eqref{eq:second-order-kernel-regression-proof2}, and \eqref{eq:second-order-kernel-regression-proof3}, we have
\begin{align} \label{eq:second-order-kernel-regression-proof4}
    N_bh^df_X(x_0)\ge \frac{nh^df_X(x_0)}{2B_\delta} \ge \frac{C_{\mathrm{blk}}C_B}{B_\delta}\log(1/\delta) \ge C_{\mathrm{blk}}, \quad b=1,\ldots,B_\delta.
\end{align}
Therefore, \eqref{eq:second-order-kernel-regression-proof3}, \eqref{eq:second-order-kernel-regression-proof4}, and Lemma~\ref{lem:one-kernel-block-second-order} give
\begin{align}
\begin{split}
    &\Pbb\left(d_{\mathcal Y}\left(\widehat m_{h,b}(x_0),m_\oplus(x_0)\right)>C_{\mathrm{bias}}h^2+\left(C_{\mathrm{resp}}\sigma+C_{\mathrm{des}}h\right)\sqrt{\frac{2B_\delta}{nh^df_X(x_0)}}\right)\\
    &\le\Pbb\left(d_{\mathcal Y}(\widehat m_{h,b}(x_0),m_\oplus(x_0))> C_{\mathrm{bias}}h^2 +\frac{C_{\mathrm{resp}}\sigma+C_{\mathrm{des}}h} {\sqrt{N_bh^df_X(x_0)}}\right)\\
    &\le\frac5{16}, \quad b=1,\ldots,B_\delta.
\end{split} \label{eq:second-order-kernel-regression-proof5}
\end{align}
Since the block estimators are independent, \eqref{eq:second-order-kernel-regression-proof5} and Lemma~\ref{lem:mr-probability} below with $p=11/16$ give
\begin{align}
\begin{split}
    &\Pbb\left(d_{\mathcal Y}\left(\widehat m^{\mathrm{RLFR}}_{h,B_\delta}(x_0),m_\oplus(x_0)\right)>3C_{\mathrm{bias}}h^2+3\left(C_{\mathrm{resp}}\sigma+C_{\mathrm{des}}h\right)\sqrt{\frac{2B_\delta}{nh^df_X(x_0)}}\right)\\
    &\le Q_{B_\delta}\left(\frac{11}{16}\right) \\
    &\le\delta,
\end{split} \label{eq:second-order-kernel-regression-proof6}
\end{align}
where the last inequality holds by the definition of $B_\delta$. By \eqref{eq:second-order-kernel-regression-proof1} and \eqref{def:second-order-theorem-constant2}, we have
\begin{align} \label{eq:second-order-kernel-regression-proof7}
    3C_{\mathrm{bias}}h^2+3\left(C_{\mathrm{resp}}\sigma+C_{\mathrm{des}}h\right)\sqrt{\frac{2B_\delta}{nh^df_X(x_0)}} \le C_{S,U}\left(h^2+\sigma\sqrt{\frac{\log(1/\delta)}{nh^df_X(x_0)}}\right)
\end{align}
since $h<h_0\le\sigma$. Therefore, \eqref{eq:second-order-kernel-regression-proof6} and \eqref{eq:second-order-kernel-regression-proof7} prove the theorem.

\begin{corollary}\label{cor:second-order-optimized-rate}
Under the conditions of Theorem~\ref{thm:second-order-kernel-regression}, for fixed $\delta\in(0,1/2]$, define
\begin{align*}
    h_{n,\delta}^{\mathrm{opt}}:=\left(\frac{\sigma^2\log(1/\delta)}{nf_X(x_0)}\right)^{1/(d+4)}, \quad n\in\mathbb{N}.
\end{align*}
Then, for sufficiently large $n\in\mathbb N$,
\begin{align*}
    \Pbb\left(d_{\mathcal Y}\left(\widehat m^{\mathrm{RLFR}}_{h_{n,\delta}^{\mathrm{opt}},B_\delta}(x_0),m_\oplus(x_0)\right)>2C_{S,U}\left(\frac{\sigma^2\log(1/\delta)}{nf_X(x_0)}\right)^{2/(d+4)}\right)\le\delta.
\end{align*}
\end{corollary}

\begin{proof}[Proof of \Cref{cor:second-order-optimized-rate}]
Let $\delta\in(0,1/2]$ be fixed. Note that $B_\delta$ is also fixed. By the definition of $h_{n,\delta}^{\mathrm{opt}}$, we have
\begin{align*}
    \lim_{n\to\infty}h_{n,\delta}^{\mathrm{opt}}=0, \quad \lim_{n\to\infty} n\left(h_{n,\delta}^{\mathrm{opt}}\right)^df_X(x_0)=\infty.
\end{align*}
Thus, for sufficiently large $n\in\mathbb N$, the above implies
\begin{align*}
    h_{n,\delta}^{\mathrm{opt}}\in(0,h_0), \quad B_\delta\le\lfloor n/2\rfloor, \quad n\left(h_{n,\delta}^{\mathrm{opt}}\right)^df_X(x_0)\ge C_{S,U}'\log(1/\delta).
\end{align*}
Therefore, the above, the definition of $h_{n,\delta}^{\mathrm{opt}}$, and Theorem~\ref{thm:second-order-kernel-regression} prove the corollary.
\end{proof}

\subsection{Proof of \Cref{thm:minimax-lower-second-order}}
Define the smooth radial function $\varphi:\R^d\to[0,1]$ by
\begin{align*}
    \varphi(u):=
    \begin{cases}
        \displaystyle \exp\left(1-\frac{1}{1-\|u\|^2}\right), & \text{if} \quad \|u\|<1,\\
        0, & \text{if} \quad \|u\|\ge1,
    \end{cases}\quad u\in\R^d.
\end{align*}
Then $\varphi(0)=1$, $\nabla\varphi(0)=0$, and $\operatorname{supp}(\varphi)\subset\overline B(0,1)$. Define
\begin{align*}
    H_\varphi:=\max\left\{1,\sup_{u\in\R^d}\|\nabla^2\varphi(u)\|_{\mathrm{op}}\right\}, \quad
    J_\varphi:=\int_{\R^d}\varphi(u)^2\,du.
\end{align*}
Then $H_\varphi, J_\varphi\in(0,\infty)$. By Condition~\ref{ass:D1}, there exists $r_0\in(0,h_0]$ such that
\begin{align} \label{eq:minimax-lower-second-order-proof1}
    f_X(x)\le2f_X(x_0), \quad x\in\overline B(x_0,r_0).
\end{align}
Define
\begin{align*}
    C_{S,L}':=\frac{H_\varphi^2}{J_\varphi r_0^{d+4}}, \quad
    C_{S,L}:=\frac{1}{2H_\varphi}\left(\frac{H_\varphi^2}{2J_\varphi}\right)^{2/(d+4)}.
\end{align*}
Let $\delta\in(0,1/16]$ and $L\in(0,\infty)$ be given, and suppose
\begin{align} \label{eq:minimax-lower-second-order-proof2}
    L\left(2h_0+h_0^2\right)\le2\tau, \quad
    nf_X(x_0)L^2\ge C_{S,L}'\tau^2\log(1/\delta).
\end{align}
Define
\begin{align*}
    h_{\delta, L}:=\left(
        \frac{H_\varphi^2\tau^2\log(1/(4\delta))}
        {J_\varphi nf_X(x_0)L^2}
    \right)^{1/(d+4)}.
\end{align*}
By \eqref{eq:minimax-lower-second-order-proof2} and the definition of $C_{S,L}'$, we have
\begin{align} \label{eq:minimax-lower-second-order-proof3}
    h_{\delta, L}^{d+4}
    =\frac{H_\varphi^2\tau^2\log(1/(4\delta))}
    {J_\varphi nf_X(x_0)L^2}
    \le\frac{H_\varphi^2\tau^2\log(1/\delta)}
    {J_\varphi nf_X(x_0)L^2}
    \le\frac{H_\varphi^2}{J_\varphi C_{S,L}'}
    =r_0^{d+4},
\end{align}
which implies $h_{\delta, L}\le r_0\le h_0$. We now define $g_1, g_2:\R^d\to\R$ by
\begin{align*}
    g_1(x):=0, \quad
    g_2(x):=\frac{Lh_{\delta, L}^2}{H_\varphi}\varphi\left(\frac{x-x_0}{h_{\delta, L}}\right),
    \quad x\in\R^d.
\end{align*}
By the definitions of $\varphi$ and $g_2$, we have
\begin{align} \label{eq:minimax-lower-second-order-proof4}
    \nabla g_2(x_0)
    =\frac{Lh_{\delta,L}}{H_\varphi}\nabla\varphi(0)=0,\quad
    \sup_{x\in\R^d}\|\nabla^2g_2(x)\|_{\mathrm{op}}
    =\frac{L}{H_\varphi}\sup_{u\in\R^d}\|\nabla^2\varphi(u)\|_{\mathrm{op}}
    \le L.
\end{align}
Also, since $\operatorname{supp}(\varphi)\subset\overline B(0,1)$ and $h_{\delta,L}\le h_0$, we have
\begin{align} \label{eq:minimax-lower-second-order-proof5}
    \operatorname{supp}(g_2)\subset\overline B(x_0,h_{\delta,L})\subset\overline B(x_0,h_0).
\end{align}
Therefore, by \eqref{eq:minimax-lower-second-order-proof4}, \eqref{eq:minimax-lower-second-order-proof5}, and the definition of $g_1$, we have $g_1,g_2\in\mathcal G_L$. Moreover, by the isometry of $\iota$ and $\varphi(0)=1$,
\begin{align} \label{eq:minimax-lower-second-order-proof6}
    d_{\mathcal Y}\left(\iota(g_1(x_0)),\iota(g_2(x_0))\right)
    =|g_1(x_0)-g_2(x_0)|
    =\frac{Lh_{\delta, L}^2}{H_\varphi}. 
\end{align}
Also, by Lemma~\ref{lem:gaussian-product-kl} below, \eqref{eq:minimax-lower-second-order-proof1}, \eqref{eq:minimax-lower-second-order-proof3}, the definition of $h_{\delta,L}$, and the change of variables, we have
\begin{align}
\begin{split}
    \operatorname{KL}\left(P_{g_1,\tau}^{\otimes n},P_{g_2,\tau}^{\otimes n}\right)
    &=\frac{n}{2\tau^2}\int_{\mathcal X}g_2(x)^2f_X(x)\,dx\\
    &\le\frac{nf_X(x_0)L^2h_{\delta,L}^4}{H_\varphi^2\tau^2}
    \int_{\overline B\left(x_0,h_{\delta,L}\right)}
    \varphi\left(\frac{x-x_0}{h_{\delta,L}}\right)^2\,dx\\
    &=\frac{nf_X(x_0)J_\varphi L^2h_{\delta,L}^{d+4}}{H_\varphi^2\tau^2}\\
    &=\log(1/(4\delta)).
\end{split}
\label{eq:minimax-lower-second-order-proof7}
\end{align}
Therefore, by \eqref{eq:minimax-lower-second-order-proof6}, \eqref{eq:minimax-lower-second-order-proof7}, and Lemma~\ref{lem:two-point-estimation} below, for every measurable map $\mathcal T_n:(\mathcal X\times\mathcal Y)^n\to\mathcal Y$,
\begin{align}
\begin{split}
    \max_{j\in\{1,2\}}P_{g_j,\tau}^{\otimes n}\left(
    d_{\mathcal Y}\left(\mathcal T_n(\mathcal D_n),\iota(g_j(x_0))\right)
    \ge\frac{Lh_{\delta,L}^2}{2H_\varphi}
    \right)&\ge\frac14
    \exp\left(
        -\operatorname{KL}\left(P_{g_1,\tau}^{\otimes n},P_{g_2,\tau}^{\otimes n}\right)
    \right)\\
    &\ge\frac14\exp\left(-\log(1/(4\delta))\right)\\
    &=\delta.
\end{split}
\label{eq:minimax-lower-second-order-proof8}
\end{align}
Since $\delta\in(0,1/16]$, we have
\begin{align*}
    \log(1/(4\delta))
    =\log(1/\delta)-\log4
    \ge\frac12\log(1/\delta).
\end{align*}
Then, by the above and the definition of $h_{\delta, L}$, we have
\begin{align*}
    \frac{Lh_{\delta,L}^2}{2H_\varphi}
    &=\frac{1}{2H_\varphi}
    \left(\frac{H_\varphi^2}{J_\varphi}\right)^{2/(d+4)}
    L^{d/(d+4)}
    \left(
        \frac{\tau^2\log(1/(4\delta))}
        {nf_X(x_0)}
    \right)^{2/(d+4)}\\
    &\ge\frac{1}{2H_\varphi}
    \left(\frac{H_\varphi^2}{2J_\varphi}\right)^{2/(d+4)}
    L^{d/(d+4)}
    \left(
        \frac{\tau^2\log(1/\delta)}
        {nf_X(x_0)}
    \right)^{2/(d+4)}\\
    &=C_{S,L}L^{d/(d+4)}
    \left(
        \frac{\tau^2\log(1/\delta)}
        {nf_X(x_0)}
    \right)^{2/(d+4)}.
\end{align*}
Since $g_1,g_2\in\mathcal G_L$, the above implies
\begin{align}
\begin{split}
    &\sup_{g\in\mathcal G_L}
    P_{g,\tau}^{\otimes n}\left(
    d_{\mathcal Y}\left(\mathcal T_n(\mathcal D_n),\iota(g(x_0))\right)
    \ge C_{S,L}L^{d/(d+4)}
    \left(
        \frac{\tau^2}
        {nf_X(x_0)} \log\left(\frac{1}{\delta} \right) 
    \right)^{2/(d+4)}
    \right)\\
    &\ge
    \max_{j\in\{1,2\}}
    P_{g_j,\tau}^{\otimes n}\left(
    d_{\mathcal Y}\left(\mathcal T_n(\mathcal D_n),\iota(g_j(x_0))\right)
    \ge\frac{Lh_{\delta,L}^2}{2H_\varphi}
    \right).
\end{split} \label{eq:minimax-lower-second-order-proof9}
\end{align}
Therefore, \eqref{eq:minimax-lower-second-order-proof8} and \eqref{eq:minimax-lower-second-order-proof9} prove the theorem.

\begin{proposition}\label{prop:gaussian-subclass-upper}
Suppose Conditions~\ref{ass:K1}--\ref{ass:K2}, \ref{ass:D1}--\ref{ass:D2}, \ref{ass:M1}(iv), \ref{ass:M2}(iv), and \ref{ass:M4} hold. For $L,\tau\in(0,\infty)$, if $L\left(2h_0+h_0^2\right)\le2\tau$, then, uniformly over $g\in\mathcal G_L$, the model \eqref{eq:embeddedgaussian} satisfies Conditions~\ref{ass:M1}--\ref{ass:M4} and \ref{ass:S1}, with $\sigma=\sqrt2\tau$ and $L_S=3C_{\mathcal Y}L$.
\end{proposition}

\begin{proof}[Proof of \Cref{prop:gaussian-subclass-upper}]
Let $L,\tau\in(0,\infty)$ be given such that $L\left(2h_0+h_0^2\right)\le2\tau$. By the definition of $\mathcal G_L$, Taylor's formula gives
\begin{align} \label{eq:gaussian-subclass-upper-proof1}
    |g(x)-g(x_0)|
    \le L\|x-x_0\|+\frac{L}{2}\|x-x_0\|^2
    \le L\left(h_0+\frac{h_0^2}{2}\right)
    \le\tau, \quad x\in\overline B(x_0,h_0),
\end{align}
for every $g\in\mathcal G_L$. Therefore, by the definition of $M_\oplus$, \eqref{eq:embeddedgaussian}, the isometry of $\iota$, and Lemma~\ref{lem:embedded-gaussian-mean} below, for every $g\in\mathcal G_L$, we have
\begin{align}
\begin{split}
    M_\oplus(x,m_\oplus(x_0))
    &=\E\left[d_{\mathcal Y}^2\left(Y,m_\oplus(x_0)\right)\,\middle|\,X=x\right]\\
    &=\E\left[d_{\mathcal Y}^2\left(\iota(g(x)+\varepsilon),\iota(g(x_0))\right)\right]\\
    &=\E\left[(g(x)-g(x_0)+\varepsilon)^2\right]\\
    &=(g(x)-g(x_0))^2
    +2(g(x)-g(x_0))\E[\varepsilon]
    +\E[\varepsilon^2]\\
    &=(g(x)-g(x_0))^2+\tau^2\\
    &\le2\tau^2, \quad x\in\overline B(x_0,h_0),
\end{split}
\label{eq:gaussian-subclass-upper-proof2}
\end{align}
where the last inequality follows from \eqref{eq:gaussian-subclass-upper-proof1}. Thus Condition~\ref{ass:M3} holds with $\sigma=\sqrt2\tau$.

For the remainder of the proof, we fix $g\in\mathcal G_L$. To verify Conditions~\ref{ass:M1}--\ref{ass:M2}, first note that $\E[K_h(X,x_0)]>0$ for $h\in(0,h_0]$ by Lemma~\ref{lem:kernel-target-positive} below. For $h\in(0,h_0]$, define the probability measure $\widetilde P_h:\mathcal B(\mathcal Y)\to[0,1]$ by
\begin{align*}
    \widetilde P_h(A)
    :=\frac{1}{\E[K_h(X,x_0)]}
    \int_{\overline B(x_0,h_0)}
    K_h(x,x_0)P_{Y\mid X=x}(A)\,dP_X(x),
    \quad A\in\mathcal B(\mathcal Y),
\end{align*}
where $\mathcal B(\mathcal Y)$ denotes the Borel $\sigma$-field on $\mathcal Y$. By Condition~\ref{ass:K1} and the definition of $M_\oplus$,
\begin{align*}
    \int_{\mathcal Y}d_{\mathcal Y}^2(y,\omega)\,d\widetilde P_h(y)
    &=\frac{\E\left[K_h(X,x_0)d_{\mathcal Y}^2(Y,\omega)\right]}
    {\E[K_h(X,x_0)]}
    =\widetilde M_h(x_0,\omega),
    \quad \omega\in\mathcal Y.
\end{align*}
Let $\widetilde Y_h$ be a $\mathcal Y$-valued random variable with distribution $\widetilde P_h$. Then
\begin{align}\label{eq:gaussian-subclass-upper-proof3}
    \E\left[d_{\mathcal Y}^2\left(\widetilde Y_h,\omega\right)\right]
    =\widetilde M_h(x_0,\omega),
    \quad \omega\in\mathcal Y.
\end{align}
By \eqref{eq:gaussian-subclass-upper-proof2} and \eqref{eq:gaussian-subclass-upper-proof3}, we have
\begin{align*}
    \E\left[d_{\mathcal Y}^2\left(\widetilde Y_h,m_\oplus(x_0)\right)\right]
    =\widetilde M_h(x_0,m_\oplus(x_0))=\frac{\E\left[K_h(X,x_0)M_\oplus(X,m_\oplus(x_0))\right]}
    {\E[K_h(X,x_0)]}
    \le2\tau^2,
\end{align*}
which also implies $\widetilde Y_h\in\mathcal L_2(\mathcal Y)$. Therefore, by the above, Conditions~\ref{ass:M1}(iv) and \ref{ass:M2}(iv), applied with $N=1$, $\widetilde{\bm Y}=(\widetilde Y_h)$, and $\bm\lambda=(1)$, give Conditions~\ref{ass:M1}(i) and \ref{ass:M2}(i). Similarly, for $N\in\mathbb N$, $\bm\lambda\in\mathcal S_N$, and $\bm x\in\overline B(x_0,h_0)^N$, let $\widetilde Y_i\sim P_{Y\mid X=x_i}$ for $i=1,\ldots,N$. By \eqref{eq:gaussian-subclass-upper-proof2}, $\widetilde Y_i\in\mathcal L_2(\mathcal Y)$ for every $i=1,\ldots,N$, and
\begin{align*}
    \overline Q_{\oplus,\bm\lambda}(\widetilde{\bm Y},\omega)
    =\sum_{i=1}^N\lambda_i
    \E\left[d_{\mathcal Y}^2(\widetilde Y_i,\omega)\right]
    =\sum_{i=1}^N\lambda_iM_\oplus(x_i,\omega)=\overline M_{\oplus,\bm\lambda}(\bm x,\omega),
    \quad \omega\in\mathcal Y.
\end{align*}
By the above, Conditions~\ref{ass:M1}(iv) and \ref{ass:M2}(iv) give Conditions~\ref{ass:M1}(ii) and \ref{ass:M2}(ii). Also, for $\bm y=(y_1,\ldots,y_N)\in\mathcal Y^N$, take $\widetilde Y_i=y_i$ for $i=1,\ldots,N$. Then
\begin{align*}
    \overline Q_{\oplus,\bm\lambda}(\widetilde{\bm Y},\omega)
    =\sum_{i=1}^N\lambda_i d_{\mathcal Y}^2(y_i,\omega)
    =\overline B_{\oplus,\bm\lambda}(\bm y,\omega),
    \quad \omega\in\mathcal Y.
\end{align*}
By the above, Conditions~\ref{ass:M1}(iv) and \ref{ass:M2}(iv) give Conditions~\ref{ass:M1}(iii) and \ref{ass:M2}(iii).

It remains to verify Condition~\ref{ass:S1}. Fix $\omega\in\mathcal Y$ and define
\begin{align*}
    \Delta_\omega(t)
    :=d_{\mathcal Y}^2\left(\iota(t),\omega\right)
    -d_{\mathcal Y}^2\left(\iota(t),m_\oplus(x_0)\right),
    \quad t\in\R.
\end{align*}
By Condition~\ref{ass:M4} and the isometry of $\iota$, we have
\begin{align*} 
    |\Delta_\omega(t)-\Delta_\omega(s)|
    \le C_{\mathcal Y}d_{\mathcal Y}(\omega,m_\oplus(x_0))|t-s|,
    \quad s,t\in\R,
\end{align*}
which implies that $\Delta_\omega$ is Lipschitz. Thus, $\Delta_\omega$ is differentiable almost everywhere and
\begin{align} \label{eq:gaussian-subclass-upper-proof5}
    |\Delta_\omega'(t)|
    \le C_{\mathcal Y}d_{\mathcal Y}(\omega,m_\oplus(x_0))
\end{align}
for almost every $t\in\R$. Moreover, since $\Delta_\omega$ is Lipschitz, its convolution with the Gaussian density is infinitely differentiable; see Proposition~9.10 of \cite{Folland1999}. Therefore, integration by parts gives
\begin{align} \label{eq:gaussian-subclass-upper-proof6}
    \frac{d}{du}\E[\Delta_\omega(u+\varepsilon)]
    =\E[\Delta_\omega'(u+\varepsilon)],\quad
    \frac{d^2}{du^2}\E[\Delta_\omega(u+\varepsilon)]
    =\frac{1}{\tau^2}\E[\varepsilon\Delta_\omega'(u+\varepsilon)].
\end{align}
Since $g\in C^2(\R^d)$ and $\operatorname{supp}(g)\subset\overline B(x_0,h_0)$, we have
\begin{align*}
    \nabla g(x)=0,\quad x\in\overline B(x_0,h_0)\setminus B(x_0,h_0).
\end{align*}
Let $x\in B(x_0,h_0)$ be fixed. Then choose $\widetilde x \in\overline B(x_0,h_0)\setminus B(x_0,h_0)$ such that $\|x-\widetilde x\|\le h_0$. By the fundamental theorem of calculus and the definition of $\mathcal G_L$, the above implies
\begin{align*}
    \|\nabla g(x)\|=\left\|\nabla g(x)-\nabla g\left(\widetilde x\right)\right\| \le\sup_{z\in B(x_0,h_0)} \|\nabla^2g(z)\|_{\mathrm{op}}\|x-\widetilde x\| \le Lh_0.
\end{align*}
Since the above holds for every $x\in B(x_0,h_0)$, we have
\begin{align} \label{eq:gaussian-subclass-upper-proof7}
    \sup_{x\in B(x_0,h_0)}\|\nabla g(x)\|\le Lh_0.
\end{align}
Also, by the definition of $S_\omega$ in \eqref{def:criterion-contrast},
\begin{align*}
    S_\omega(x)=\E[\Delta_\omega(g(x)+\varepsilon)],\quad x\in \overline B(x_0,h_0).
\end{align*}
Therefore, by the above, \eqref{eq:gaussian-subclass-upper-proof6}, and the chain rule, for $x\in B(x_0,h_0)$, we have
\begin{align} \label{eq:gaussian-subclass-upper-proof8}
\begin{split}
    \nabla S_\omega(x)
    &=\E[\Delta_\omega'(g(x)+\varepsilon)]\nabla g(x),\\
    \nabla^2S_\omega(x)
    &=\frac{1}{\tau^2}
    \E[\varepsilon\Delta_\omega'(g(x)+\varepsilon)]
    \nabla g(x)\nabla g(x)^\top +\E[\Delta_\omega'(g(x)+\varepsilon)]\nabla^2g(x).
\end{split}
\end{align}
By \eqref{eq:gaussian-subclass-upper-proof5}, \eqref{eq:gaussian-subclass-upper-proof7}, \eqref{eq:gaussian-subclass-upper-proof8}, the definition of $\mathcal G_L$, and the fact that $\E[|\varepsilon|]=\tau\sqrt{2/\pi}$, we have
\begin{align}
\begin{split}
    \|\nabla S_\omega(x_0)\|
    &\le \E\left[\left|\Delta_\omega'(g(x_0)+\varepsilon)\right|\right]
    \|\nabla g(x_0)\|
    \le C_{\mathcal Y}Ld_{\mathcal Y}(\omega,m_\oplus(x_0)),\\
    \sup_{x\in B(x_0,h_0)}
    \|\nabla^2S_\omega(x)\|_{\mathrm{op}}
    &\le
    \frac{1}{\tau^2}
    \sup_{x\in B(x_0,h_0)}
    \E\left[
        |\varepsilon|
        \left|\Delta_\omega'(g(x)+\varepsilon)\right|
    \right]
    \|\nabla g(x)\|^2\\
    &\quad+
    \sup_{x\in B(x_0,h_0)}
    \E\left[
        \left|\Delta_\omega'(g(x)+\varepsilon)\right|
    \right]
    \|\nabla^2g(x)\|_{\mathrm{op}}\\
    &\le
    C_{\mathcal Y}
    \left(\sqrt{\frac{2}{\pi}}\frac{L^2h_0^2}{\tau}+L\right)d_{\mathcal Y}(\omega,m_\oplus(x_0)).
\end{split} \label{eq:gaussian-subclass-upper-proof9}
\end{align}
Since $L\left(2h_0+h_0^2\right)\le2\tau$, we have
\begin{align} \label{eq:gaussian-subclass-upper-proof10}
    \sqrt{\frac{2}{\pi}}\frac{L^2h_0^2}{\tau}+L \le 2\sqrt{\frac{2}{\pi}}\frac{L h_0}{2+h_0}+L
    <\left(2\sqrt{\frac{2}{\pi}}+1\right)L<3L.
\end{align}
Therefore, by \eqref{eq:gaussian-subclass-upper-proof9} and \eqref{eq:gaussian-subclass-upper-proof10}, Condition~\ref{ass:S1} holds with $L_S=3C_{\mathcal Y}L$. Since $g\in\mathcal G_L$ was arbitrary and all the above bounds are independent of $g$, the conclusions hold uniformly over $g\in\mathcal G_L$. This proves the proposition.
\end{proof}

\section{Proofs for Section \ref{subsec:w1-main-summary}}\label{sec:supp-w1}

In this section, we provide proofs of Theorems~\ref{thm:w1-kernel-regression}, \ref{thm:minimax-lower-w1} and related results. Auxiliary results used in this section are collected in Section \ref{sec:supp-auxiliary}. 

\subsection{Bias and kernel weights}\label{subsec:block-target}

\begin{proposition}\label{prop:bias-w1}
Suppose Conditions~\ref{ass:K1}, \ref{ass:D1}, \ref{ass:M1}--\ref{ass:M4}, and \ref{ass:S2} hold. For $h\in(0,h_0)$,
\begin{align*}
    d_{\mathcal Y}(\widetilde m_h(x_0),m_\oplus(x_0))
    \le\frac{C_{\mathcal Y}}{\kappa_{\mathcal Y}}L_Wh^\beta.
\end{align*}
\end{proposition}

\begin{proof}[Proof of \Cref{prop:bias-w1}]
Let $h\in(0,h_0)$ be given. By the definition of $M_\oplus$, the definition of $S_\omega$ in \eqref{def:criterion-contrast}, and Condition~\ref{ass:M4}, for $\omega\in\mathcal Y$ and $x\in\overline B(x_0,h_0)$, we have
\begin{align*}
    &\left|S_\omega(x)-S_\omega(x_0)\right|\\
    &=\left|M_\oplus(x,\omega)-M_\oplus(x,m_\oplus(x_0))-M_\oplus(x_0,\omega)+M_\oplus(x_0,m_\oplus(x_0))\right|\\
    &\le\int_{\mathcal Y\times\mathcal Y}\left|d_{\mathcal Y}^2(y_1,\omega)-d_{\mathcal Y}^2(y_1,m_\oplus(x_0))-d_{\mathcal Y}^2(y_2,\omega)+d_{\mathcal Y}^2(y_2,m_\oplus(x_0))\right|\,d\gamma(y_1,y_2)\\
    &\le C_{\mathcal Y}d_{\mathcal Y}(\omega,m_\oplus(x_0))
    \int_{\mathcal Y\times\mathcal Y}d_{\mathcal Y}(y_1,y_2)\,d\gamma(y_1,y_2), \quad \gamma\in\Pi\left(P_{Y\mid X=x},P_{Y\mid X=x_0}\right).
\end{align*}
All four values of $M_\oplus$ are finite by Condition~\ref{ass:M3} and the triangle inequality and the Cauchy--Schwarz inequality. Therefore, taking the infimum over $\gamma\in\Pi(P_{Y\mid X=x},P_{Y\mid X=x_0})$ and applying Condition~\ref{ass:S2}, for $\omega\in\mathcal Y$ and $x\in\overline B(x_0,h_0)$, the above gives
\begin{align}
\begin{split}
    \left|S_\omega(x)-S_\omega(x_0)\right|&\le C_{\mathcal Y}d_{\mathcal Y}(\omega,m_\oplus(x_0))W_1(P_{Y\mid X=x},P_{Y\mid X=x_0})\\
    &\le C_{\mathcal Y}L_Wd_{\mathcal Y}(\omega,m_\oplus(x_0))\|x-x_0\|^\beta.
\end{split}\label{eq:bias-w1-proof1}
\end{align}
Since $\E[K_h(X,x_0)]>0$ by Lemma~\ref{lem:kernel-target-positive} below, the definitions of $M_\oplus$, $\widetilde M_h$, $\widetilde m_h(x_0)$, and $S_{\widetilde m_h(x_0)}$ give
\begin{align}
\begin{split}
    \frac{\E\left[K_h(X,x_0)S_{\widetilde m_h(x_0)}(X)\right]}{\E[K_h(X,x_0)]}
    &=\frac{\E\left[K_h(X,x_0)\left(M_\oplus(X,\widetilde m_h(x_0))-M_\oplus(X,m_\oplus(x_0))\right)\right]}{\E[K_h(X,x_0)]}\\
    &=\frac{\E\left[K_h(X,x_0)\left(d_{\mathcal Y}^2(Y,\widetilde m_h(x_0))-d_{\mathcal Y}^2(Y,m_\oplus(x_0))\right)\right]}{\E[K_h(X,x_0)]}\\
    &=\widetilde M_h(x_0,\widetilde m_h(x_0))-\widetilde M_h(x_0,m_\oplus(x_0))\\
    &\le0.
\end{split}\label{eq:bias-w1-proof2}
\end{align}
Also, Condition~\ref{ass:M2}(ii) with $\bm\lambda=(1)$ and $\bm x=(x_0)$ gives
\begin{align} \label{eq:bias-w1-proof3}
    \kappa_{\mathcal Y}d_{\mathcal Y}^2(\omega,m_\oplus(x_0))
    \le S_\omega(x_0), \quad \omega\in\mathcal Y.
\end{align} 
Therefore, by \eqref{eq:bias-w1-proof1}, \eqref{eq:bias-w1-proof2}, \eqref{eq:bias-w1-proof3}, and Condition~\ref{ass:K1}, we have
\begin{align*}
\kappa_{\mathcal Y}d_{\mathcal Y}^2(\widetilde m_h(x_0),m_\oplus(x_0))
&\le S_{\widetilde m_h(x_0)}(x_0)\\
&\le S_{\widetilde m_h(x_0)}(x_0)
-\frac{\E\left[K_h(X,x_0)S_{\widetilde m_h(x_0)}(X)\right]}{\E[K_h(X,x_0)]}\\
&=\frac{\E\left[K_h(X,x_0)\left(
S_{\widetilde m_h(x_0)}(x_0)-S_{\widetilde m_h(x_0)}(X)
\right)\right]}{\E[K_h(X,x_0)]}\\
&\le C_{\mathcal Y}L_Wd_{\mathcal Y}(\widetilde m_h(x_0),m_\oplus(x_0))
\frac{\E\left[K_h(X,x_0)\|X-x_0\|^\beta\right]}{\E[K_h(X,x_0)]}\\
&\le C_{\mathcal Y}L_Wh^\beta
d_{\mathcal Y}(\widetilde m_h(x_0),m_\oplus(x_0)).
\end{align*}
This proves the proposition.
\end{proof}

\begin{lemma}\label{lem:conditional-bias}
Under the conditions of Proposition~\ref{prop:bias-w1}, for $h\in(0,h_0)$ and $b=1,\ldots,B$, if $S_b>0$, then
\begin{align*}
    d_{\mathcal Y}\left(\overline m_{\oplus,\bm\lambda_b}(\bm X_b),m_\oplus(x_0)\right)
    \le\frac{C_{\mathcal Y}}{\kappa_{\mathcal Y}}L_Wh^\beta.
\end{align*}
\end{lemma}

\begin{proof}[Proof of \Cref{lem:conditional-bias}]
Let $h\in(0,h_0)$ and $b=1,\ldots,B$ be given, and suppose $S_b>0$. By the definitions of $\overline m_{\oplus,\bm\lambda_b}(\bm X_b)$ and $S_\omega$, we have
\begin{align*}
    \sum_{i\in I_b}\lambda_{bi}S_{\overline m_{\oplus,\bm\lambda_b}(\bm X_b)}(X_i)
    &=\sum_{i\in I_b}\lambda_{bi}\left(
    M_\oplus\left(X_i,\overline m_{\oplus,\bm\lambda_b}(\bm X_b)\right)
    -M_\oplus(X_i,m_\oplus(x_0))
    \right)\\
    &=\overline M_{\oplus,\bm\lambda_b}\left(\bm X_b,\overline m_{\oplus,\bm\lambda_b}(\bm X_b)\right)
    -\overline M_{\oplus,\bm\lambda_b}\left(\bm X_b,m_\oplus(x_0)\right)\\
    &\le0.
\end{align*}
Therefore, by the above, \eqref{eq:bias-w1-proof1}, \eqref{eq:bias-w1-proof3}, and Condition~\ref{ass:K1}, we have
\begin{align*}
    \kappa_{\mathcal Y}d_{\mathcal Y}^2\left(\overline m_{\oplus,\bm\lambda_b}(\bm X_b),m_\oplus(x_0)\right)
    &\le S_{\overline m_{\oplus,\bm\lambda_b}(\bm X_b)}(x_0)\\
    &\le S_{\overline m_{\oplus,\bm\lambda_b}(\bm X_b)}(x_0)
    -\sum_{i\in I_b}\lambda_{bi}S_{\overline m_{\oplus,\bm\lambda_b}(\bm X_b)}(X_i)\\
    &=\sum_{i\in I_b}\lambda_{bi}\left(S_{\overline m_{\oplus,\bm\lambda_b}(\bm X_b)}(x_0)-S_{\overline m_{\oplus,\bm\lambda_b}(\bm X_b)}(X_i)\right)\\
    &\le C_{\mathcal Y}L_Wd_{\mathcal Y}\left(\overline m_{\oplus,\bm\lambda_b}(\bm X_b),m_\oplus(x_0)\right)
    \sum_{i\in I_b}\lambda_{bi}\|X_i-x_0\|^\beta\\
    &\le C_{\mathcal Y}L_Wh^\beta d_{\mathcal Y}\left(\overline m_{\oplus,\bm\lambda_b}(\bm X_b),m_\oplus(x_0)\right),
\end{align*}
where the last inequality follows since $\lambda_{bi}>0$ implies $X_i\in\overline B(x_0,h)$ for $i\in I_b$. This proves the lemma.
\end{proof}

\subsection{Proof of \Cref{thm:w1-kernel-regression}}

We first prepare the following lemma. 
\begin{lemma}\label{lem:one-kernel-block}
Under the conditions of Theorem~\ref{thm:w1-kernel-regression}, there exist constants $C_{\mathrm{blk}},C_{\mathrm{resp}}\in(0,\infty)$ depending only on the fixed constants in Conditions~\ref{ass:K1}, \ref{ass:D1}, and \ref{ass:M1}--\ref{ass:M4} such that the following holds. For $h\in(0,h_0)$ and $b=1,\ldots,B$, if
\begin{align*}
    N_bh^df_X(x_0)\ge C_{\mathrm{blk}},
\end{align*}
then
\begin{align*}
    \Pbb\left(
    d_{\mathcal Y}(\widehat m_{h,b}(x_0),m_\oplus(x_0))
    >
    \frac{C_{\mathcal Y}}{\kappa_{\mathcal Y}}L_Wh^\beta
    +\frac{C_{\mathrm{resp}}\sigma}{\sqrt{N_bh^df_X(x_0)}}
    \right)
    \le\frac{9}{32}.
\end{align*}
\end{lemma}

\begin{proof}[Proof of \Cref{lem:one-kernel-block}]
Let $h\in(0,h_0)$ and $b=1,\ldots,B$ be given, and let $C_{\mathrm{blk}}$ and $C_{\mathrm{resp}}$ be the constants defined in \eqref{def:second-order-theorem-constant1}. Suppose that $N_bh^df_X(x_0)\ge C_{\mathrm{blk}}$. Define
\begin{align*}
    \mathcal A_b:=\left\{S_b\ge\frac{N_b\E[K_h(X,x_0)]}{2}\right\}.
\end{align*}
By Lemmas~\ref{lem:kernel-target-positive} and \ref{lem:kernel-mass} below,
\begin{align} \label{eq:one-kernel-block-w1-proof1}
    \Pbb\left(\mathcal A_b^c\right) = \Pbb\left(S_b<\frac{N_b\E[K_h(X,x_0)]}{2}\right)
    \le\exp\left(-\frac{3N_bh^df_X(x_0)}{56\|K\|_\infty}\right)
    \le\frac{1}{24}.
\end{align}
By the above, Lemma~\ref{lem:weight-square} below, and Corollary~\ref{cor:r2-conditional}, on $\mathcal A_b$, we have
\begin{align*}
    &\Pbb\left(
    d_{\mathcal Y}\left(\widehat m_{h,b}(x_0),\overline m_{\oplus,\bm\lambda_b}(\bm X_b)\right)
    >
    \frac{C_{\mathrm{resp}}\sigma}{\sqrt{N_bh^df_X(x_0)}}
    \,\middle|\,X_1,\ldots,X_n
    \right)\\
    &\le\Pbb\left(
    d_{\mathcal Y}\left(\widehat m_{h,b}(x_0),\overline m_{\oplus,\bm\lambda_b}(\bm X_b)\right)
    >
    \frac{2C_{\mathcal Y}\sigma}{\kappa_{\mathcal Y}}
    \sqrt{\left(1+\frac{1}{\kappa_{\mathcal Y}}\right)\sum_{i\in I_b}\lambda_{bi}^2}
    \,\middle|\,X_1,\ldots,X_n
    \right)
    \le\frac14.
\end{align*}
Since $\mathcal A_b$ is determined only by the predictors, the above implies
\begin{align}
\begin{split}
    &\Pbb\left(
    \mathcal A_b\cap
    \left\{
    d_{\mathcal Y}\left(\widehat m_{h,b}(x_0),\overline m_{\oplus,\bm\lambda_b}(\bm X_b)\right)
    \le\frac{C_{\mathrm{resp}}\sigma}{\sqrt{N_bh^df_X(x_0)}}
    \right\}
    \right)\\
    &=\E\left[
    \mathds{1}_{\mathcal A_b}
    \Pbb\left(
    d_{\mathcal Y}\left(\widehat m_{h,b}(x_0),\overline m_{\oplus,\bm\lambda_b}(\bm X_b)\right)
    \le\frac{C_{\mathrm{resp}}\sigma}{\sqrt{N_bh^df_X(x_0)}}
    \,\middle|\,X_1,\ldots,X_n
    \right)
    \right]\\
    &\ge\frac34\Pbb(\mathcal A_b).
\end{split}
\label{eq:one-kernel-block-w1-proof3}
\end{align}
On the event in \eqref{eq:one-kernel-block-w1-proof3}, Lemma~\ref{lem:conditional-bias} and the triangle inequality give
\begin{align}
\begin{split}
    d_{\mathcal Y}(\widehat m_{h,b}(x_0),m_\oplus(x_0))
    &\le d_{\mathcal Y}\left(\widehat m_{h,b}(x_0),\overline m_{\oplus,\bm\lambda_b}(\bm X_b)\right)
    +d_{\mathcal Y}\left(\overline m_{\oplus,\bm\lambda_b}(\bm X_b),m_\oplus(x_0)\right)\\
    &\le\frac{C_{\mathcal Y}}{\kappa_{\mathcal Y}}L_Wh^\beta
    +\frac{C_{\mathrm{resp}}\sigma}{\sqrt{N_bh^df_X(x_0)}}.
\end{split}
\label{eq:one-kernel-block-w1-proof4}
\end{align}
Therefore, by \eqref{eq:one-kernel-block-w1-proof1}, \eqref{eq:one-kernel-block-w1-proof3}, and \eqref{eq:one-kernel-block-w1-proof4},
\begin{align*}
    \Pbb\left(
    d_{\mathcal Y}(\widehat m_{h,b}(x_0),m_\oplus(x_0))
    >
    \frac{C_{\mathcal Y}}{\kappa_{\mathcal Y}}L_Wh^\beta
    +\frac{C_{\mathrm{resp}}\sigma}{\sqrt{N_bh^df_X(x_0)}}
    \right)
    \le1-\frac34\Pbb(\mathcal A_b)
    \le\frac{9}{32}.
\end{align*}
This proves the lemma.
\end{proof}

Now we prove Theorem \ref{thm:w1-kernel-regression}. 
By Lemma~\ref{lem:binomial-calibration} below, there exists a universal constant $C_B\in(0,\infty)$ such that
\begin{align} \label{eq:w1-kernel-regression-proof1}
    B_\delta\le C_B\log(1/\delta), \quad \delta\in\left(0,\frac12\right].
\end{align}
Let $\delta\in(0,1/2]$ and $h\in(0,h_0)$ be given. Define
\begin{align} \label{def:w1-theorem-constant1}
    C_{W,U}':=2C_{\mathrm{blk}}C_B, \quad
    C_{W,U}:=\max\left\{\frac{3C_{\mathcal Y}}{\kappa_{\mathcal Y}},3\sqrt{2C_B}C_{\mathrm{resp}}\right\},
\end{align}
where $C_{\mathrm{blk}},C_{\mathrm{resp}}\in(0,\infty)$ are defined by \eqref{def:second-order-theorem-constant1}. Suppose that
\begin{align} \label{eq:w1-kernel-regression-proof2}
    B_\delta\le\lfloor n/2\rfloor, \quad
    nh^df_X(x_0)\ge C_{W,U}'\log(1/\delta).
\end{align}
Since the block sizes differ by at most one,
\begin{align} \label{eq:w1-kernel-regression-proof3}
    N_b\ge\left\lfloor\frac{n}{B_\delta}\right\rfloor
    \ge\frac{n}{2B_\delta}, \quad b=1,\ldots,B_\delta.
\end{align}
By \eqref{eq:w1-kernel-regression-proof1}, \eqref{eq:w1-kernel-regression-proof2}, and \eqref{eq:w1-kernel-regression-proof3},
\begin{align} \label{eq:w1-kernel-regression-proof4}
    N_bh^df_X(x_0)
    \ge\frac{nh^df_X(x_0)}{2B_\delta}
    \ge\frac{C_{\mathrm{blk}}C_B}{B_\delta}\log(1/\delta)
    \ge C_{\mathrm{blk}}, \quad b=1,\ldots,B_\delta.
\end{align}
Therefore, \eqref{eq:w1-kernel-regression-proof3}, \eqref{eq:w1-kernel-regression-proof4}, and Lemma~\ref{lem:one-kernel-block} give
\begin{align}
\begin{split}
    &\Pbb\left(
    d_{\mathcal Y}\left(\widehat m_{h,b}(x_0),m_\oplus(x_0)\right)
    >
    \frac{C_{\mathcal Y}}{\kappa_{\mathcal Y}}L_Wh^\beta
    +C_{\mathrm{resp}}\sigma\sqrt{\frac{2B_\delta}{nh^df_X(x_0)}}
    \right)\\
    &\le\Pbb\left(
    d_{\mathcal Y}\left(\widehat m_{h,b}(x_0),m_\oplus(x_0)\right)
    >
    \frac{C_{\mathcal Y}}{\kappa_{\mathcal Y}}L_Wh^\beta
    +\frac{C_{\mathrm{resp}}\sigma}{\sqrt{N_bh^df_X(x_0)}}
    \right)\\
    &\le\frac{9}{32}\\
    &<\frac{5}{16}, \quad b=1,\ldots,B_\delta.
\end{split}
\label{eq:w1-kernel-regression-proof5}
\end{align}
Since the block estimators are independent, \eqref{eq:w1-kernel-regression-proof5} and Lemma~\ref{lem:mr-probability} below with $p=11/16$ give
\begin{align}
\begin{split}
    \Pbb\left(
    d_{\mathcal Y}\left(\widehat m^{\mathrm{RLFR}}_{h,B_\delta}(x_0),m_\oplus(x_0)\right)
    >
    \frac{3C_{\mathcal Y}}{\kappa_{\mathcal Y}}L_Wh^\beta
    +3C_{\mathrm{resp}}\sigma\sqrt{\frac{2B_\delta}{nh^df_X(x_0)}}
    \right)&\le Q_{B_\delta}\left(\frac{11}{16}\right)\\
    &\le\delta,
\end{split}
\label{eq:w1-kernel-regression-proof6}
\end{align}
where the last inequality holds by the definition of $B_\delta$. By \eqref{eq:w1-kernel-regression-proof1} and \eqref{def:w1-theorem-constant1},
\begin{align} \label{eq:w1-kernel-regression-proof7}
    \frac{3C_{\mathcal Y}}{\kappa_{\mathcal Y}}L_Wh^\beta
    +3C_{\mathrm{resp}}\sigma\sqrt{\frac{2B_\delta}{nh^df_X(x_0)}}
    \le C_{W,U}\left(
    L_Wh^\beta+\sigma\sqrt{\frac{\log(1/\delta)}{nh^df_X(x_0)}}
    \right).
\end{align}
Therefore, \eqref{eq:w1-kernel-regression-proof6} and \eqref{eq:w1-kernel-regression-proof7} prove the theorem.

\begin{corollary}\label{cor:w1-optimized-rate}
Under the conditions of Theorem~\ref{thm:w1-kernel-regression}, for fixed $\delta\in(0,1/2]$, define
\begin{align*}
    h_{n,\delta}^{\mathrm{opt}}:=\left(\frac{\sigma^2\log(1/\delta)}{L_W^2nf_X(x_0)}\right)^{1/(2\beta+d)}, \quad n\in\mathbb N.
\end{align*}
Then, for sufficiently large $n\in\mathbb N$,
\begin{align*}
    \Pbb\left(d_{\mathcal Y}\left(\widehat m^{\mathrm{RLFR}}_{h_{n,\delta}^{\mathrm{opt}},B_\delta}(x_0),m_\oplus(x_0)\right)>2C_{W,U}L_W^{d/(2\beta+d)}\left(\frac{\sigma^2\log(1/\delta)}{nf_X(x_0)}\right)^{\beta/(2\beta+d)}\right)\le\delta.
\end{align*}
\end{corollary}

\begin{proof}[Proof of \Cref{cor:w1-optimized-rate}]
Let $\delta\in(0,1/2]$ be fixed. Note that $B_\delta$ is also fixed. By the definition of $h_{n,\delta}^{\mathrm{opt}}$, we have
\begin{align*}
    \lim_{n\to\infty}h_{n,\delta}^{\mathrm{opt}}=0, \quad \lim_{n\to\infty}n\left(h_{n,\delta}^{\mathrm{opt}}\right)^df_X(x_0)=\infty.
\end{align*}
Thus, for sufficiently large $n\in\mathbb N$, the above implies
\begin{align*}
    h_{n,\delta}^{\mathrm{opt}}\in(0,h_0), \quad B_\delta\le\lfloor n/2\rfloor, \quad n\left(h_{n,\delta}^{\mathrm{opt}}\right)^df_X(x_0)\ge C_{W,U}'\log(1/\delta).
\end{align*}
Therefore, the above, the definition of $h_{n,\delta}^{\mathrm{opt}}$, and Theorem~\ref{thm:w1-kernel-regression} prove the corollary.
\end{proof}

\subsection{Proof of \Cref{thm:minimax-lower-w1}}
For $\beta\in(0,1]$, define
\begin{align*}
    \varphi_\beta(u):=(1-\|u\|)_+^\beta, \quad u\in\R^d,
\end{align*}
and let
\begin{align*}
    J_\beta:=\int_{\R^d}\varphi_\beta(u)^2\,du.
\end{align*}
Since the map $u\mapsto(1-\|u\|)_+$ is Lipschitz with constant one and 
\begin{align*}
    |s^\beta-t^\beta|\le|s-t|^\beta, \quad s,t\in[0,\infty)
\end{align*}
for $\beta\in(0,1]$, we have
\begin{align} \label{eq:minimax-lower-w1-proof1}
    |\varphi_\beta(u)-\varphi_\beta(v)|
    \le\|u-v\|^\beta, \quad u,v\in\R^d.
\end{align}
Let $V_d$ denote the volume of the $d$-dimensional unit Euclidean ball $\overline B(0,1)$. Since
\begin{align*}
    0\le\varphi_\beta(u)\le1, \quad \varphi_1(u)\le\varphi_\beta(u), \quad u\in\R^d,
\end{align*}
and $\operatorname{supp}(\varphi_\beta)\subset\overline B(0,1)$, we have, for $\beta\in(0,1]$,
\begin{align} \label{eq:minimax-lower-w1-proof-add1}
    0<J_1\le J_\beta\le V_d.
\end{align}
By Condition~\ref{ass:D1}, there exists $r_0\in\left(0,\min\{h_0, 1\}\right)$ such that
\begin{align} \label{eq:minimax-lower-w1-proof2}
    f_X(x)\le2f_X(x_0), \quad x\in\overline B(x_0,r_0).
\end{align}
Define
\begin{align*}
    C_{W,L}':=\frac{1}{J_1r_0^{d+2}}, \quad
    C_{W,L}:=\frac12\min\left\{1,(2V_d)^{-1/(d+2)}\right\}.
\end{align*}
Let $\delta\in(0,1/16]$, $\beta\in(0,1]$, and $L\in(0,\infty)$ be given, and suppose
\begin{align} \label{eq:minimax-lower-w1-proof3}
    Lh_0^\beta\le\tau, \quad
    nf_X(x_0)L^2\ge C_{W,L}'\tau^2\log(1/\delta).
\end{align}
Define
\begin{align*}
    h_{\delta,\beta,L}
    :=\left(
    \frac{\tau^2\log(1/(4\delta))}{J_\beta nf_X(x_0)L^2}
    \right)^{1/(2\beta+d)}.
\end{align*}
By \eqref{eq:minimax-lower-w1-proof-add1}, \eqref{eq:minimax-lower-w1-proof3}, and the definition of $C_{W,L}'$, we have
\begin{align}\label{eq:minimax-lower-w1-proof4}
    h_{\delta,\beta,L}^{d+2\beta}=\frac{\tau^2\log(1/(4\delta))}{J_\beta nf_X(x_0)L^2}
    \le\frac{\tau^2\log(1/\delta)}{J_\beta nf_X(x_0)L^2}
    \le\frac{1}{J_\beta C_{W,L}'}
    =\frac{J_1}{J_\beta}r_0^{d+2}\le r_0^{d+2}
    \le r_0^{d+2\beta}.
\end{align}
Thus, $h_{\delta,\beta,L}\le r_0\le h_0$. We now define $g_1,g_2:\R^d\to\R$ by
\begin{align*}
    g_1(x):=0, \quad
    g_2(x):=Lh_{\delta,\beta,L}^\beta
    \varphi_\beta\left(\frac{x-x_0}{h_{\delta,\beta,L}}\right), \quad x\in\R^d.
\end{align*}
By the definition of $\varphi_\beta$ and \eqref{eq:minimax-lower-w1-proof4},
\begin{align} \label{eq:minimax-lower-w1-proof-add2}
    \operatorname{supp}(g_2)
    \subset\overline B(x_0,h_{\delta,\beta,L})
    \subset\overline B(x_0,h_0).
\end{align}
Also, by \eqref{eq:minimax-lower-w1-proof1}, we have
\begin{align} \label{eq:minimax-lower-w1-proof-add3}
    |g_2(x_1)-g_2(x_2)|
    \le Lh_{\delta,\beta,L}^\beta
    \left\|\frac{x_1-x_2}{h_{\delta,\beta,L}}\right\|^\beta
    =L\|x_1-x_2\|^\beta, \quad x_1,x_2\in\R^d.
\end{align}
Therefore, by \eqref{eq:minimax-lower-w1-proof-add2}, \eqref{eq:minimax-lower-w1-proof-add3}, and the definition of $g_1$, we have $g_1,g_2\in\mathcal G_{\beta,L}$. Moreover, by the isometry of $\iota$ and $\varphi_\beta(0)=1$,
\begin{align} \label{eq:minimax-lower-w1-proof5}
    d_{\mathcal Y}\left(\iota(g_1(x_0)),\iota(g_2(x_0))\right)
    =|g_1(x_0)-g_2(x_0)|
    =Lh_{\delta,\beta,L}^\beta.
\end{align}
By Lemma~\ref{lem:gaussian-product-kl}, \eqref{eq:minimax-lower-w1-proof2}, and the change of variables, we have
\begin{align}
\begin{split}
    \operatorname{KL}\left(P_{g_1,\tau}^{\otimes n},P_{g_2,\tau}^{\otimes n}\right)
    &=\frac{n}{2\tau^2}\int_{\mathcal X}g_2(x)^2f_X(x)\,dx\\
    &\le\frac{nf_X(x_0)L^2h_{\delta,\beta,L}^{2\beta}}{\tau^2}
    \int_{\overline B(x_0,h_{\delta,\beta,L})}
    \varphi_\beta\left(\frac{x-x_0}{h_{\delta,\beta,L}}\right)^2\,dx\\
    &=\frac{nf_X(x_0)J_\beta L^2h_{\delta,\beta,L}^{d+2\beta}}{\tau^2}\\
    &=\log(1/(4\delta)).
\end{split}
\label{eq:minimax-lower-w1-proof6}
\end{align}
Therefore, by \eqref{eq:minimax-lower-w1-proof5}, \eqref{eq:minimax-lower-w1-proof6}, and Lemma~\ref{lem:two-point-estimation} below, for every measurable map $\mathcal T_n:(\mathcal X\times\mathcal Y)^n\to\mathcal Y$,
\begin{align}
\begin{split}
    \max_{j\in\{1,2\}}
    P_{g_j,\tau}^{\otimes n}\left(
    d_{\mathcal Y}\left(\mathcal T_n(\mathcal D_n),\iota(g_j(x_0))\right)
    \ge\frac{Lh_{\delta,\beta,L}^\beta}{2}
    \right)
    &\ge\frac14
    \exp\left(-\operatorname{KL}\left(P_{g_1,\tau}^{\otimes n},P_{g_2,\tau}^{\otimes n}\right)\right)\\
    &\ge\frac14\exp\left(-\log(1/(4\delta))\right)\\
    &=\delta.
\end{split}
\label{eq:minimax-lower-w1-proof7}
\end{align}
Since $\delta\in(0,1/16]$,
\begin{align*}
    \log(1/(4\delta))
    =\log(1/\delta)-\log4
    \ge\frac12\log(1/\delta).
\end{align*}
By \eqref{eq:minimax-lower-w1-proof-add1} and $\beta/(2\beta+d)\le1/(d+2)$, we have
\begin{align*}
    \left(2J_\beta\right)^{-\beta/(2\beta+d)}
    &\ge
    \begin{cases}
        1, & \text{if} \quad 2J_\beta\le1,\\
        (2V_d)^{-1/(d+2)}, & \text{if} \quad 2J_\beta>1,
    \end{cases}
\end{align*}
which implies $\left(2J_\beta\right)^{-\beta/(2\beta+d)}\ge2C_{W,L}$. Therefore, by the above and the definitions of $C_{W,L}$ and $h_{\delta,\beta,L}$, we have
\begin{align}
\begin{split}
    \frac{Lh_{\delta,\beta,L}^\beta}{2}
    &=\frac12J_\beta^{-\beta/(2\beta+d)}
    L^{d/(2\beta+d)}
    \left(
    \frac{\tau^2\log(1/(4\delta))}{nf_X(x_0)}
    \right)^{\beta/(2\beta+d)}\\
    &\ge\frac12
    \left(2J_\beta\right)^{-\beta/(2\beta+d)}
    L^{d/(2\beta+d)}
    \left(
    \frac{\tau^2\log(1/\delta)}{nf_X(x_0)}
    \right)^{\beta/(2\beta+d)}\\
    &\ge C_{W,L}
    L^{d/(2\beta+d)}
    \left(
    \frac{\tau^2\log(1/\delta)}{nf_X(x_0)}
    \right)^{\beta/(2\beta+d)}.
\end{split}
\label{eq:minimax-lower-w1-proof8}
\end{align}
Since $g_1,g_2\in\mathcal G_{\beta,L}$, the above implies
\begin{align}
\begin{split}
    &\sup_{g\in\mathcal G_{\beta,L}}
    P_{g,\tau}^{\otimes n}\left(
    d_{\mathcal Y}\left(\mathcal T_n(\mathcal D_n),\iota(g(x_0))\right)
    \ge C_{W,L}L^{d/(2\beta+d)}
    \left(
    \frac{\tau^2\log(1/\delta)}{nf_X(x_0)}
    \right)^{\beta/(2\beta+d)}
    \right)\\
    &\ge\max_{j\in\{1,2\}}
    P_{g_j,\tau}^{\otimes n}\left(
    d_{\mathcal Y}\left(\mathcal T_n(\mathcal D_n),\iota(g_j(x_0))\right)
    \ge\frac{Lh_{\delta,\beta,L}^\beta}{2}
    \right).
\end{split}
\label{eq:minimax-lower-w1-proof9}
\end{align}
Therefore, \eqref{eq:minimax-lower-w1-proof7} and \eqref{eq:minimax-lower-w1-proof9} prove the theorem.

\begin{proposition}\label{prop:gaussian-holder-subclass-upper}
Suppose Conditions~\ref{ass:K1}, \ref{ass:D1}, \ref{ass:M1}(iv), \ref{ass:M2}(iv), and \ref{ass:M4} hold. For $\beta\in(0,1]$ and $L,\tau\in(0,\infty)$, if $Lh_0^\beta\le\tau$, then, uniformly over $g\in\mathcal G_{\beta,L}$, the model \eqref{eq:embeddedgaussian} satisfies Conditions~\ref{ass:M1}--\ref{ass:M4} and \ref{ass:S2}, with $\sigma=\sqrt2\tau$ and $L_W=L$.
\end{proposition}

\begin{proof}[Proof of \Cref{prop:gaussian-holder-subclass-upper}]
Let $\beta\in(0,1]$ and $L,\tau\in(0,\infty)$ be given such that $Lh_0^\beta\le\tau$. By the definition of $\mathcal G_{\beta,L}$,
\begin{align} \label{eq:gaussian-holder-subclass-upper-proof1}
    |g(x)-g(x_0)|
    \le L\|x-x_0\|^\beta
    \le Lh_0^\beta
    \le\tau, \quad x\in\overline B(x_0,h_0),
\end{align}
for every $g\in\mathcal G_{\beta,L}$. Therefore, by the definition of $M_\oplus$, \eqref{eq:embeddedgaussian}, the isometry of $\iota$, and Lemma~\ref{lem:embedded-gaussian-mean}, for every $g\in\mathcal G_{\beta,L}$,
\begin{align}
\begin{split}
    M_\oplus(x,m_\oplus(x_0))
    &=\E\left[d_{\mathcal Y}^2\left(\iota(g(x)+\varepsilon),\iota(g(x_0))\right)\right]\\
    &=\E\left[(g(x)-g(x_0)+\varepsilon)^2\right]\\
    &=(g(x)-g(x_0))^2+\tau^2\\
    &\le2\tau^2, \quad x\in\overline B(x_0,h_0).
\end{split}
\label{eq:gaussian-holder-subclass-upper-proof2}
\end{align}
Thus Condition~\ref{ass:M3} holds with $\sigma=\sqrt2\tau$.

Also, using the same Gaussian variable in the two conditional distributions and the isometry of $\iota$, for $x\in\overline B(x_0,h_0)$,
\begin{align}
\begin{split}
    W_1(P_{Y\mid X=x},P_{Y\mid X=x_0})
    &\le\E\left[d_{\mathcal Y}\left(\iota(g(x)+\varepsilon),\iota(g(x_0)+\varepsilon)\right)\right]\\
    &=|g(x)-g(x_0)|\\
    &\le L\|x-x_0\|^\beta.
\end{split}
\label{eq:gaussian-holder-subclass-upper-proof3}
\end{align}
Hence Condition~\ref{ass:S2} holds with $L_W=L$.

The verification of Conditions~\ref{ass:M1}(i)--(iii) and \ref{ass:M2}(i)--(iii) in the proof of Proposition~\ref{prop:gaussian-subclass-upper} uses only Conditions~\ref{ass:K1}, \ref{ass:D1}, \ref{ass:M1}(iv), \ref{ass:M2}(iv), and the local second-moment bound. Therefore, by \eqref{eq:gaussian-holder-subclass-upper-proof2}, the same argument shows that Conditions~\ref{ass:M1}(i)--(iii) and \ref{ass:M2}(i)--(iii) hold. Since $g\in\mathcal G_{\beta,L}$ was arbitrary and all the above bounds are independent of $g$, the conclusions hold uniformly over $g\in\mathcal G_{\beta,L}$. This proves the proposition.
\end{proof}

\section{Auxiliary results}\label{sec:supp-auxiliary}

This section collects auxiliary results used in the proofs for Sections \ref{sec:supp-second-order} and \ref{sec:supp-w1}. We first establish basic properties of the kernel weights and block estimators and then deterministic and probabilistic bounds for the distance-based aggregation.

\subsection{Properties of kernel weights}\label{subsec:kernel-weight-bounds}

In this subsection, we provide auxiliary results for the kernel weights and block estimators.

\begin{lemma}\label{lem:kernel-target-positive}
Suppose Conditions~\ref{ass:K1} and \ref{ass:D1} hold. Then
\begin{align}
    \frac{f_X(x_0)}{2}\le\E[K_h(X,x_0)]\le2f_X(x_0), &\quad h\in(0,h_0]. \label{eq:expectationbound}
\end{align}
In addition, if Condition~\ref{ass:M3} also holds, then, for every $\omega\in\mathcal Y$,
\begin{align} \label{eq:bigmhfinite}
    \widetilde M_h(x_0,\omega) \le 2\sigma^2+2d_{\mathcal Y}^2(\omega,m_\oplus(x_0)),
\end{align}
where the constant $\sigma\in(0,\infty)$ is defined in Condition~\ref{ass:M3}.
\end{lemma}

\begin{proof}[Proof of \Cref{lem:kernel-target-positive}]
By \eqref{eq:local-density-bound}, for $h\in(0,h_0]$, we have
\begin{align*}
    \E[K_h(X,x_0)]
    =\int_{\mathcal X}\frac{1}{h^d}K\left(\frac{x-x_0}{h}\right)f_X(x)\,dx
    =\int_{\overline{B}(0,1)}K(u)f_X(x_0+hu)\,du
\end{align*}
by Condition~\ref{ass:K1} and the change of variables. Since $h\in(0,h_0]$ and $u\in\overline{B}(0,1)$ imply $x_0+hu\in\overline B(x_0,h_0)$, \eqref{eq:expectationbound} follows directly from \eqref{eq:local-density-bound}.

In addition, for $\omega\in\mathcal Y$,
\begin{align*}
    d_{\mathcal Y}^2(Y,\omega)
    \le
    2d_{\mathcal Y}^2(Y,m_\oplus(x_0))
    +2d_{\mathcal Y}^2(\omega,m_\oplus(x_0))
\end{align*}
by the triangle inequality and the Cauchy--Schwarz inequality. Therefore, for $h\in(0,h_0]$ and $\omega\in\mathcal Y$, the above inequality and the definition of $\widetilde M_h$ in \eqref{def:bigmhtilde} imply
\begin{align} \label{eq:bigmhbound1}
    \widetilde M_h(x_0,\omega)
    &\le \frac{2\E\left[K_h(X,x_0)M_\oplus(X,m_\oplus(x_0))\right]}
    {\E[K_h(X,x_0)]}
    +2d_{\mathcal Y}^2(\omega,m_\oplus(x_0)).
\end{align}
For $h\in(0,h_0]$ and $x\in\mathcal X$, note that $K_h(x,x_0)>0$ only if $x\in\overline B(x_0,h)\subset\overline B(x_0,h_0)$ by Condition~\ref{ass:K1}. Hence, for $h\in(0,h_0]$,
\begin{align} \label{eq:bigmhbound2}
    \E\left[K_h(X,x_0)M_\oplus(X,m_\oplus(x_0))\right]
    \le \sigma^2 \E[K_h(X,x_0)]
\end{align}
by Condition~\ref{ass:M3}. By \eqref{eq:bigmhbound1} and \eqref{eq:bigmhbound2}, \eqref{eq:bigmhfinite} holds for every $\omega\in\mathcal Y$.
\end{proof}

\begin{lemma}\label{lem:kernel-mass}
Under Conditions~\ref{ass:K1} and \ref{ass:D1}, for $h\in(0,h_0]$,
\begin{align*}
    \Pbb\left(S_b<\frac{N_b\E[K_h(X,x_0)]}{2}\right)
    \le
    \exp\left(
        -\frac{3N_bh^d\E[K_h(X,x_0)]}{28\|K\|_\infty}
    \right),
\end{align*}
where $\|K\|_\infty:=\sup_{u\in\mathbb{R}^d}K(u)$.
\end{lemma}

\begin{proof}[Proof of \Cref{lem:kernel-mass}]
By Lemma~\ref{lem:kernel-target-positive}, $\E[K_h(X,x_0)]>0$. Also, for each $i\in I_b$, Condition~\ref{ass:K1} gives
\begin{align*}
    \E[K_h(X,x_0)]-K_h(X_i,x_0)
    &\le\frac{\|K\|_\infty}{h^d}, \\
    \operatorname{Var}\left[K_h(X_i,x_0)\right] \le    \E\left[K_h^2(X_i,x_0)\right] &\le \frac{\|K\|_\infty}{h^d} \E[K_h(X,x_0)].
\end{align*}
Therefore, applying Bernstein's inequality to
$\E[K_h(X,x_0)]-K_h(X_i,x_0)$ for $i\in I_b$ gives
\begin{align*}
    &\Pbb\left(S_b<\frac{N_b\E[K_h(X,x_0)]}{2}\right)\\
    &\le
    \exp\left(
        -\frac{N_b^2\E[K_h(X,x_0)]^2/4}
        {2N_b\|K\|_\infty h^{-d}\E[K_h(X,x_0)]
        +\|K\|_\infty h^{-d}N_b\E[K_h(X,x_0)]/3}
    \right)\\
    &=\exp\left(
        -\frac{3N_bh^d\E[K_h(X,x_0)]}{28\|K\|_\infty}
    \right).
\end{align*}
For details, see Corollary 2.11 of \cite{BoucheronLugosiMassart2013}. This proves the lemma.
\end{proof}

\begin{lemma}\label{lem:weight-square}
Under Conditions~\ref{ass:K1} and \ref{ass:D1}, on $S_b\ge N_b\E[K_h(X,x_0)]/2$,
\begin{align*}
    \sum_{i\in I_b}\lambda_{bi}^2
    \le
    \frac{2\|K\|_\infty}{N_bh^d\E[K_h(X,x_0)]}
    \le
    \frac{4\|K\|_\infty}{N_bh^df_X(x_0)}.
\end{align*}
\end{lemma}

\begin{proof}[Proof of \Cref{lem:weight-square}]
By Condition~\ref{ass:K1}, on $S_b>0$,
\begin{align*}
    \sum_{i\in I_b}\lambda_{bi}^2
    =\frac{\sum_{i\in I_b}K_h^2(X_i,x_0)}{S_b^2}\le\frac{\|K\|_\infty \sum_{i\in I_b}K_h(X_i,x_0)}{h^d S_b^2}
    =\frac{\|K\|_\infty}{h^dS_b}.
\end{align*}
Therefore, on $S_b\ge N_b\E[K_h(X,x_0)]/2$,
\begin{align*}
    \sum_{i\in I_b}\lambda_{bi}^2
    \le\frac{2\|K\|_\infty}{N_bh^d\E[K_h(X,x_0)]}.
\end{align*}
The above and Lemma~\ref{lem:kernel-target-positive} prove the lemma.
\end{proof}

\begin{lemma}\label{lem:block-second-moment}
Suppose Conditions~\ref{ass:K1} and \ref{ass:M1}--\ref{ass:M3} hold. Then
\begin{align} \label{eq:block-second-moment-part1}
    d_{\mathcal Y}\left(\widetilde m_h\left(x_0\right),m_\oplus\left(x_0\right)\right)
    \le\frac{\sigma}{\sqrt{\kappa_{\mathcal Y}}}.
\end{align}
Moreover, on $S_b>0$,
\begin{align} \label{eq:block-second-moment-part2}
    d_{\mathcal Y}\left(\overline m_{\oplus,\bm\lambda_b}\left(\bm X_b\right),m_\oplus\left(x_0\right)\right)
    \le\frac{\sigma}{\sqrt{\kappa_{\mathcal Y}}}.
\end{align}
Finally, for every $i\in I_b$ with $\lambda_{bi}>0$,
\begin{align} \label{eq:block-second-moment-part3}
    \E\left[d_{\mathcal Y}^2\left(Y_i,\overline m_{\oplus,\bm\lambda_b}\left(\bm X_b\right)\right)\,\middle|\,X_1,\ldots,X_n\right]
    \le2\left(1+\frac{1}{\kappa_{\mathcal Y}}\right)\sigma^2.
\end{align}
\end{lemma}

\begin{proof}[Proof of \Cref{lem:block-second-moment}]
By \eqref{eq:bigmhbound2} and the definition of $\widetilde M_h$ in \eqref{def:bigmhtilde},
\begin{align*}
    \widetilde M_h\left(x_0,m_\oplus\left(x_0\right)\right)\le\sigma^2.
\end{align*}
Therefore, Condition~\ref{ass:M2}(i) and the nonnegativity of $\widetilde M_h$ imply
\begin{align*}
    \kappa_{\mathcal Y}d_{\mathcal Y}^2\left(\widetilde m_h\left(x_0\right),m_\oplus\left(x_0\right)\right)
    &\le \widetilde M_h\left(x_0,m_\oplus\left(x_0\right)\right)
    -\widetilde M_h\left(x_0,\widetilde m_h\left(x_0\right)\right)
    \le\sigma^2.
\end{align*}
This proves \eqref{eq:block-second-moment-part1}.

On $S_b>0$, Condition~\ref{ass:M3} and $\sum_{i\in I_b}\lambda_{bi}=1$ give
\begin{align*}
    \overline M_{\oplus,\bm\lambda_b}\left(\bm X_b,m_\oplus\left(x_0\right)\right)
    =\sum_{i\in I_b}\lambda_{bi}M_\oplus\left(X_i,m_\oplus\left(x_0\right)\right)
    \le\sigma^2,
\end{align*}
where Condition~\ref{ass:M3} is applied only to indices with $\lambda_{bi}>0$. Omitting the indices with $\lambda_{bi}=0$, Condition~\ref{ass:M2}(ii) and the nonnegativity of $\overline M_{\oplus,\bm\lambda_b}$ imply
\begin{align*}
    \kappa_{\mathcal Y}d_{\mathcal Y}^2\left(\overline m_{\oplus,\bm\lambda_b}\left(\bm X_b\right),m_\oplus\left(x_0\right)\right)
    &\le \overline M_{\oplus,\bm\lambda_b}\left(\bm X_b,m_\oplus\left(x_0\right)\right)
    -\overline M_{\oplus,\bm\lambda_b}\left(\bm X_b,\overline m_{\oplus,\bm\lambda_b}\left(\bm X_b\right)\right)
    \le\sigma^2.
\end{align*}
This proves \eqref{eq:block-second-moment-part2}.

For every $i\in I_b$ with $\lambda_{bi}>0$, conditionally on $X_1,\ldots,X_n$, $Y_i$ has conditional distribution $P_{Y\mid X=X_i}$. Thus,
\begin{align*}
    \E\left[d_{\mathcal Y}^2\left(Y_i,\overline m_{\oplus,\bm\lambda_b}\left(\bm X_b\right)\right)\,\middle|\,X_1,\ldots,X_n\right]
    =M_\oplus\left(X_i,\overline m_{\oplus,\bm\lambda_b}\left(\bm X_b\right)\right).
\end{align*}
By \eqref{eq:block-second-moment-part2}, the triangle inequality, and the Cauchy--Schwarz inequality,
\begin{align*}
    M_\oplus\left(X_i,\overline m_{\oplus,\bm\lambda_b}\left(\bm X_b\right)\right)
    &\le2M_\oplus\left(X_i,m_\oplus\left(x_0\right)\right)
    +2d_{\mathcal Y}^2\left(\overline m_{\oplus,\bm\lambda_b}\left(\bm X_b\right),m_\oplus\left(x_0\right)\right)\\
    &\le2\sigma^2+\frac{2\sigma^2}{\kappa_{\mathcal Y}}
    =2\left(1+\frac{1}{\kappa_{\mathcal Y}}\right)\sigma^2.
\end{align*}
This proves \eqref{eq:block-second-moment-part3}.
\end{proof}

We now give a short remark on Condition~\ref{ass:M3}.

\begin{remark}\label{rem:moment-center}\leavevmode
Under Conditions~\ref{ass:M1}--\ref{ass:M2}, Condition~\ref{ass:M3} is equivalent, up to a change in the constant, to a local second-moment bound about a fixed reference point $y_0\in\mathcal Y$. Indeed, suppose that, for some $y_0\in\mathcal Y$ and $s\in(0,\infty)$,
\begin{align*}
    \sup_{x\in\overline B(x_0,h_0)}M_\oplus(x,y_0)\le s^2.
\end{align*}
Then, by Condition~\ref{ass:M2}(ii) applied with $N=1$, $\bm\lambda=(1)$, and $\bm x=(x_0)$, we have
\begin{align*}
    \kappa_{\mathcal Y}d_{\mathcal Y}^2(y_0,m_\oplus(x_0))
    \le M_\oplus(x_0,y_0)-M_\oplus(x_0,m_\oplus(x_0))
    \le s^2.
\end{align*}
Therefore, by the triangle inequality and the Cauchy--Schwarz inequality,
\begin{align*}
    \sup_{x\in\overline B(x_0,h_0)}M_\oplus(x,m_\oplus(x_0))
    &\le2\sup_{x\in\overline B(x_0,h_0)}M_\oplus(x,y_0)
    +2d_{\mathcal Y}^2(y_0,m_\oplus(x_0))\\
    &\le2\left(1+\frac{1}{\kappa_{\mathcal Y}}\right)s^2.
\end{align*}

Conversely, suppose Conditions~\ref{ass:M1}--\ref{ass:M3} hold and let $y_0\in\mathcal Y$ be fixed. Then, by the triangle inequality and the Cauchy--Schwarz inequality,
\begin{align*}
    \sup_{x\in\overline B(x_0,h_0)}M_\oplus(x,y_0)
    &\le2\sup_{x\in\overline B(x_0,h_0)}M_\oplus(x,m_\oplus(x_0))
    +2d_{\mathcal Y}^2(y_0,m_\oplus(x_0))\\
    &\le2\sigma^2+2d_{\mathcal Y}^2(y_0,m_\oplus(x_0)).
\end{align*}
Thus the choice of a fixed reference point changes only the constant in the local second-moment bound.
\end{remark}

\subsection{Bounds for the distance-based aggregation}\label{subsec:supp-aggregation}

In this subsection, we provide deterministic and probabilistic bounds for the distance-based aggregation. For $B\in\mathbb N$ and $p\in(1/2,1)$, define
\begin{align*}
    Q_B(p):=\Pbb\left(\operatorname{Bin}(B,p)\le\lfloor B/2\rfloor\right).
\end{align*}

\begin{lemma}\label{lem:mr-deterministic}
Let $y\in\mathcal Y$ and $r\in[0,\infty)$. Suppose that
\begin{align*}
    \sum_{b=1}^B \mathds{1}_{\left\{d_{\mathcal Y}\left(\widehat m_{h,b}(x_0),y\right)\le r\right\}}>\frac B2
\end{align*}
holds, where $\widehat m_{h,b}(x_0)$ is defined in \eqref{eq:means-selector} for $b=1, \ldots, B$. Then
\begin{align*}
    d_{\mathcal Y}\left(\widehat m^{\mathrm{RLFR}}_{h,B}(x_0),y\right)\le3r.
\end{align*}
\end{lemma}

\begin{proof}[Proof of \Cref{lem:mr-deterministic}]
By assumption, more than half of $\widehat m_{h,1}(x_0),\ldots,\widehat m_{h,B}(x_0)$ lie in $\overline B_{\mathcal Y}(y,r)$. Therefore, applying Proposition~8 of \cite{HsuSabato2016} to $(\mathcal Y,d_{\mathcal Y})$ proves the lemma.
\end{proof}

\begin{lemma}\label{lem:mr-probability}
Let $y\in\mathcal Y$, $r\in[0,\infty)$, and $p\in(1/2, 1)$. Suppose the block estimates are independent and
\begin{align*}
    \Pbb\left(d_{\mathcal Y}\left(\widehat m_{h,b}(x_0),y\right)\le r\right)\ge p, \quad b=1,\ldots,B.
\end{align*}
Then
\begin{align} \label{eq:3r-inequality}
    \Pbb\left(d_{\mathcal Y}\left(\widehat m^{\mathrm{RLFR}}_{h,B}(x_0),y\right)>3r\right)\le Q_B(p).
\end{align}
Moreover,
\begin{align}\label{eq:mr-KL-envelope}
    Q_B(p)\le\exp\left(-\frac{B}{2}\log\left(\frac{1}{4p(1-p)}\right)\right).
\end{align}
\end{lemma}

\begin{proof}[Proof of \Cref{lem:mr-probability}]
By \Cref{lem:mr-deterministic}, the event
\begin{align*}
    d_{\mathcal Y}\left(\widehat m^{\mathrm{RLFR}}_{h,B}(x_0),y\right)>3r
\end{align*}
can occur only if at most $\lfloor B/2\rfloor$ of $\widehat m_{h,1}(x_0),\ldots,\widehat m_{h,B}(x_0)$ lie in $\overline B_{\mathcal Y}(y,r)$. Since these estimates are independent and
\begin{align*}
    \Pbb\left(\widehat m_{h,b}(x_0)\in\overline B_{\mathcal Y}(y,r)\right)\ge p,
    \quad b=1,\ldots,B,
\end{align*}
we have
\begin{align*}
    &\Pbb\left(\sum_{b=1}^B\mathds{1}_{\left\{\widehat m_{h,b}(x_0)\in\overline B_{\mathcal Y}(y,r)\right\}}\le\lfloor B/2\rfloor\right)\le
    \Pbb\left(\operatorname{Bin}(B,p)\le\lfloor B/2\rfloor\right)
    =Q_B(p).
\end{align*}
Therefore, \eqref{eq:3r-inequality} holds.

To prove \eqref{eq:mr-KL-envelope}, let $Z\sim\operatorname{Bin}(B,p)$. For every $t<0$, Markov's inequality gives
\begin{align*}
    Q_B(p)
    &=\Pbb\left(Z\le\frac{B}{2}\right)=\Pbb\left(e^{tZ}\ge e^{tB/2}\right)\le e^{-tB/2}\E\left[e^{tZ}\right]
    =e^{-tB/2}\left(1-p+pe^t\right)^B.
\end{align*}
Since $p>1/2$, we may take $e^t=(1-p)/p$. Then
\begin{align*}
    Q_B(p)
    &\le\left(2\sqrt{p(1-p)}\right)^B
    =\exp\left(-\frac{B}{2}\log\left(\frac{1}{4p(1-p)}\right)\right).
\end{align*}
This proves \eqref{eq:mr-KL-envelope}.
\end{proof}

\begin{lemma}\label{lem:binomial-calibration}
The block count $B_\delta$ in \eqref{eq:Bdelta-binomial} is odd. It satisfies $B_{0.10}=11$, $B_{0.05}=19$, and $B_{0.01}=37$, and
\begin{align}\label{eq:Bdelta-log-envelope}
    \frac{\log(1/\delta)}{\log(16/5)}
    \le B_\delta
    \le\left\lceil\frac{\log(1/\delta)}{\frac12\log(64/55)}\right\rceil,
    \quad \delta\in(0,1/2].
\end{align}
In particular, $B_\delta\asymp\log(1/\delta)$ uniformly over $\delta\in(0,1/2]$.
\end{lemma}

\begin{proof}[Proof of \Cref{lem:binomial-calibration}]
Fix $\delta\in(0,1/2]$ and $p\in(1/2,1)$. For $m\in\mathbb N$, binomial recursion gives
\begin{align*}
    Q_{2m}(p)-Q_{2m-1}(p)
    &=(1-p)\Pbb\left(\operatorname{Bin}(2m-1,p)=m\right)>0,\\
    Q_{2m-1}(p)-Q_{2m+1}(p)
    &=p(2p-1)\Pbb\left(\operatorname{Bin}(2m-1,p)=m-1\right)>0.
\end{align*}
Thus $Q_{2m}(p)>Q_{2m-1}(p)$ and the sequence $Q_{2m-1}(p)$ is strictly decreasing in $m\in\mathbb{N}$. Therefore, the smallest block count satisfying $Q_B(p)\le\delta$ is odd.

For $p=11/16$, direct evaluation gives
\begin{align*}
    Q_{11}\left(\frac{11}{16}\right)<&0.10<Q_9\left(\frac{11}{16}\right), \\
    Q_{19}\left(\frac{11}{16}\right)<&0.05<Q_{17}\left(\frac{11}{16}\right),\\
    Q_{37}\left(\frac{11}{16}\right)<&0.01<Q_{35}\left(\frac{11}{16}\right).
\end{align*}
Therefore, $B_{0.10}=11$, $B_{0.05}=19$, and $B_{0.01}=37$. By \eqref{eq:mr-KL-envelope} with $p=11/16$,
\begin{align*}
    Q_B\left(\frac{11}{16}\right)
    \le\exp\left(-\frac{B}{2}\log\left(\frac{64}{55}\right)\right).
\end{align*}
 On the other hand,
\begin{align*}
    Q_B\left(\frac{11}{16}\right)
    \ge\Pbb\left(\operatorname{Bin}\left(B,\frac{11}{16}\right)=0\right)
    =\left(\frac{5}{16}\right)^B.
\end{align*}
Since $Q_{B_\delta}(11/16)\le\delta$, it follows that
\begin{align*}
    B_\delta\ge\frac{\log(1/\delta)}{\log(16/5)}.
\end{align*}
This proves \eqref{eq:Bdelta-log-envelope}.
\end{proof}

\subsection{Properties for pointwise lower bounds}
\label{subsec:lower-bound-auxiliary}

We provide three properties for the isometric embedding $\iota:\R\to\mathcal Y$ and the Gaussian model in \eqref{eq:embeddedgaussian} that are used in both pointwise minimax lower bounds.

\begin{lemma}\label{lem:embedded-gaussian-mean}
Suppose Condition~\ref{ass:M1}(iv) holds. Under \eqref{eq:embeddedgaussian},
\begin{align*}
    m_\oplus(x)=\iota(g(x)), \quad x\in\mathcal X.
\end{align*}
\end{lemma}

\begin{proof}[Proof of \Cref{lem:embedded-gaussian-mean}]
Fix $x\in\mathcal X$. For $\omega\in\mathcal Y$, the Cauchy--Schwarz inequality and the triangle inequality give
\begin{align*}
    d_{\mathcal Y}^2(\iota(g(x)+t),\omega)+d_{\mathcal Y}^2(\iota(g(x)-t),\omega)
    &\ge\frac12\left(d_{\mathcal Y}(\iota(g(x)+t),\omega)+d_{\mathcal Y}(\iota(g(x)-t),\omega)\right)^2\\
    &\ge\frac12d_{\mathcal Y}^2(\iota(g(x)+t),\iota(g(x)-t))\\
    &=2t^2, \quad t\in\R.
\end{align*}
Therefore, for $\varepsilon\sim N(0,\tau^2)$ and $\omega\in\mathcal Y$, the above inequality and symmetry give
\begin{align}\label{eq:pointwise-lower-bounds-part1}
    \E\left[d_{\mathcal Y}^2(\iota(g(x)+\varepsilon),\omega)\right]
    &=\frac12\E\left[d_{\mathcal Y}^2(\iota(g(x)+\varepsilon),\omega)+d_{\mathcal Y}^2(\iota(g(x)-\varepsilon),\omega)\right]
    \ge\tau^2,
\end{align}
where the inequality becomes an equality at $\omega=\iota(g(x))$ by the isometry of $\iota$. Moreover,
\begin{align*}
    d_{\mathcal Y}\left(\iota(g(x)+\varepsilon),\omega\right)
    \le|\varepsilon|+d_{\mathcal Y}\left(\iota(g(x)),\omega\right),
\end{align*}
so $M_\oplus(x,\omega)<\infty$. Thus, \eqref{eq:pointwise-lower-bounds-part1} and equality at $\omega=\iota(g(x))$ show that $\iota(g(x))$ is a minimizer of $M_\oplus(x,\cdot)$, and Condition~\ref{ass:M1}(iv) with $N=1$ gives uniqueness. Since $x\in\mathcal X$ was arbitrary, the result follows.
\end{proof}

For probability measures $P$ and $Q$ on the same measurable space, define the Kullback--Leibler divergence by
\begin{align*}
    \operatorname{KL}(P,Q):=
    \begin{cases}
        \displaystyle\int\log\left(\frac{dP}{dQ}\right)dP, & \text{if $P$ is absolutely continuous with respect to $Q$},\\
        +\infty, & \text{otherwise},
    \end{cases}
\end{align*}
where $dP/dQ$ denotes the Radon--Nikodym derivative.

\begin{lemma}\label{lem:gaussian-product-kl}
Let $g_1$ and $g_2$ be two functions in the Gaussian model \eqref{eq:embeddedgaussian}, and let $P_{g_1,\tau}$ and $P_{g_2,\tau}$ denote the corresponding joint distributions, so that $P_{g_1,\tau}^{\otimes n}$ and $P_{g_2,\tau}^{\otimes n}$ are the corresponding distributions of $\mathcal D_n$. Then
\begin{align*}
    \operatorname{KL}\left(P_{g_1,\tau}^{\otimes n},P_{g_2,\tau}^{\otimes n}\right)
    =\frac{n}{2\tau^2}\int_{\mathcal X}(g_1(x)-g_2(x))^2f_X(x)\,dx.
\end{align*}
\end{lemma}

\begin{proof}[Proof of \Cref{lem:gaussian-product-kl}]
It is well-known that
\begin{align*}
    \operatorname{KL}\left(N(g_1(x),\tau^2),N(g_2(x),\tau^2)\right)
    =\frac{(g_1(x)-g_2(x))^2}{2\tau^2}, \quad x\in\mathcal X.
\end{align*}
Since $\iota:\R\to\iota(\R)$ is a bimeasurable bijection, the Kullback--Leibler divergence is invariant under the transformation $\iota$. Since the marginal distribution of $X$ is the same under both models,
\begin{align*}
    \operatorname{KL}\left(P_{g_1,\tau},P_{g_2,\tau}\right)
    &=\int_{\mathcal X}\operatorname{KL}\left(N(g_1(x),\tau^2),N(g_2(x),\tau^2)\right)f_X(x)\,dx\\
    &=\frac{1}{2\tau^2}\int_{\mathcal X}(g_1(x)-g_2(x))^2f_X(x)\,dx.
\end{align*}
By independence of the observations,
\begin{align*}
    \operatorname{KL}\left(P_{g_1,\tau}^{\otimes n},P_{g_2,\tau}^{\otimes n}\right)
    =n\operatorname{KL}\left(P_{g_1,\tau},P_{g_2,\tau}\right).
\end{align*}
Combining the above proves the lemma.
\end{proof}

\begin{lemma}\label{lem:two-point-estimation}
Let $P$ and $Q$ be two distributions of $\mathcal D_n$ with respective targets $m_P,m_Q\in\mathcal Y$ satisfying $d_{\mathcal Y}(m_P,m_Q)=a>0$. Then, for every measurable estimator $\mathcal T_n:(\mathcal X\times\mathcal Y)^n\to\mathcal Y$,
\begin{align*}
    \max\left\{
    P\left(d_{\mathcal Y}\left(\mathcal T_n(\mathcal D_n),m_P\right)\ge\frac a2\right),
    Q\left(d_{\mathcal Y}\left(\mathcal T_n(\mathcal D_n),m_Q\right)\ge\frac a2\right)
    \right\}
    \ge\frac14\exp\left(-\operatorname{KL}(P,Q)\right).
\end{align*}
\end{lemma}

\begin{proof}[Proof of \Cref{lem:two-point-estimation}]
Define the binary test $\psi$ as
\begin{align*}
    \psi(\mathcal D_n):=
    \begin{cases}
        1, & \text{if} \quad  d_{\mathcal Y}\left(\mathcal T_n(\mathcal D_n),m_P\right)\le d_{\mathcal Y}\left(\mathcal T_n(\mathcal D_n),m_Q\right),\\
        2, & \text{otherwise}.
    \end{cases}
\end{align*}
By the Bretagnolle--Huber inequality, it holds that
\begin{align}\label{eq:two-point-testing}
    P\left(\psi(\mathcal D_n)=2\right)+Q\left(\psi(\mathcal D_n)=1\right)
    \ge\frac12\exp\left(-\operatorname{KL}(P,Q)\right).
\end{align}
For more details, see Lemma 2.6 of \cite{Tsybakov2009}. If $\psi(\mathcal D_n)=2$ under $P$, the triangle inequality gives
\begin{align*}
    d_{\mathcal Y}\left(\mathcal T_n(\mathcal D_n),m_P\right)\ge\frac a2.
\end{align*}
Similarly, if $\psi(\mathcal D_n)=1$ under $Q$, then
\begin{align*}
    d_{\mathcal Y}\left(\mathcal T_n(\mathcal D_n),m_Q\right)\ge\frac a2.
\end{align*}
Therefore, combining the above, we have
\begin{align*}
    &P\left(d_{\mathcal Y}\left(\mathcal T_n(\mathcal D_n),m_P\right)\ge\frac a2\right)
    +Q\left(d_{\mathcal Y}\left(\mathcal T_n(\mathcal D_n),m_Q\right)\ge\frac a2\right)\\
    &\ge P\left(\psi(\mathcal D_n)=2\right)+Q\left(\psi(\mathcal D_n)=1\right)\\
    &\ge\frac12\exp\left(-\operatorname{KL}(P,Q)\right).
\end{align*}
Taking the larger of the two probabilities proves the lemma.
\end{proof}

\fi

\end{document}